\documentclass[10pt]{amsart}

\usepackage{accents,amsmath,amsthm,amsfonts,amscd,amssymb,amscd,
	bbm,bm,bookmark,
	caption,
	doi,
	enumerate,enumitem,
	graphicx,
latexsym,lmodern,
	mathrsfs,mathtools,
	scalerel,stackengine,subcaption,
	xparse}

\usepackage[utf8]{inputenc}

\usepackage[rightcaption]{sidecap}

\usepackage[greek,english]{babel}

\usepackage{float}

\newcommand{\seq} [1]{\left\{ #1 \right\}} 
\newcommand{\abs}[1]{\left\vert #1 \right\vert}
\newcommand{\set}[1]{\left\{ #1 \right\}}
\newcommand{\cset}[2]{\left\{ #1 \, : \, #2 \right\}} 
\newcommand{\norm}[1]{\left\Vert #1 \right\Vert}
\newcommand{\ip}[2]{\left\langle #1, #2 \right\rangle} 

\newcommand\ubar[1]{\ThisStyle{
		\ensurestackMath{\stackengine{-0.25pt}{\SavedStyle#1}
			{\SavedStyle\underline{\hphantom{#1}}}{U}{c}{F}{F}{S}}}}
\newcommand{\oc}[1]{\accentset{\circ}{#1}}	
\newcommand{\ol}[1]{\overline{#1}}
\newcommand{\cl}[1]{\overline{#1}} 
\newcommand{\lap}{\Delta}  
\newcommand{\fdiv}{\mathrm{div}} 
\renewcommand{\div}{\mathrm{div}\,}
\newcommand{\pdiv}[1]{\fdiv\left( #1 \right)} 

\newcommand{\ac}[0]{\mathrm{ac}}  
\newcommand{\init}[0]{\mathrm{in}} 
\newcommand{\loc}[0]{\mathrm{loc}} 
\DeclareMathOperator*{\dist}{dist}
\DeclareMathOperator{\Dom}{Dom}
\DeclareMathOperator{\Lip}{Lip}  
\DeclareMathOperator{\Graph}{Graph}
\DeclareMathOperator{\Per}{Per} 
\newcommand{\trans}[0]{\top} 
\DeclareMathOperator{\Tr}{Tr}
\DeclareMathOperator*{\essinf}{ess\,inf}
\DeclareMathOperator*{\esssup}{ess\,sup}
\newcommand{\wt}[1]{\widetilde{#1}}

\def\sfa{\mathsf{a}}
\def\sfA{{\mathsf A}}
\def\sfc{\mathsf{c}}
\def\bbR{\mathbb{R}}
\def\bbN{\mathbb{N}}
\def\scA{\mathscr{A}}
\def\scE{\mathscr{E}}
\def\scF{\mathscr{F}}
\def\scG{\mathscr{G}}
\def\bbS{\mathbb{S}}
\def\bbW{\mathbb{W}}
\def\cL{\mathcal{L}}
\def\cH{\mathcal{H}}
\def\scD{\mathscr{D}}
\def\cX{\mathcal{X}}
\def\cV{\mathcal{V}}
\def\cP{\mathcal{P}}
\def\cM{\mathcal{M}}

\usepackage[backend=biber, style=alphabetic, mincitenames=2, minbibnames = 3, sorting=nyt]{biblatex}
\usepackage[autostyle]{csquotes}
\renewbibmacro*{url+urldate}{}
\AtEveryBibitem{\clearfield{issn}}
\AtEveryCitekey{\clearfield{issn}}
\AtBeginBibliography{\small}

\numberwithin{equation}{section}

\newtheorem{theorem}{Theorem}[section]
\newtheorem{theorem*}{Theorem}
\newtheorem{remark}[theorem]{Remark}
\newtheorem{lemma}[theorem]{Lemma}

\newtheorem{proposition}[theorem]{Proposition}
\newtheorem{corollary}[theorem]{Corollary}

\theoremstyle{definition}

\newtheorem{claim}{Claim}

\newlist{wellPosednessHypotheses-g}{enumerate}{1}
\newlist{singularLimitHypotheses-g}{enumerate}{1}

\newlist{conditionsOnA}{enumerate}{1}
\newlist{conditionsOnC}{enumerate}{1}
\newlist{conditionsOnG}{enumerate}{1}
\newlist{conditionsOnF}{enumerate}{1}
\newlist{compatibilityCondition}{enumerate}{1}
\newlist{standardAssumptions}{enumerate}{1}
\newlist{wellPreparedG}{enumerate}{1}
\newlist{energyConvergenceG}{enumerate}{1}

\setlist[conditionsOnA]{label = \bfseries(A\arabic*)}
\setlist[conditionsOnC]{label = \bfseries(C\arabic*)}
\setlist[conditionsOnG]{label = \bfseries(G\arabic*)}
\setlist[conditionsOnF]{label = \bfseries(F\arabic*)}
\setlist[compatibilityCondition]{label = \bfseries (F\&G)}
\setlist[standardAssumptions]{label = \bfseries(SA)}
\setlist[wellPreparedG]{label = \bfseries(WP)}
\setlist[energyConvergenceG]{label = \bfseries(EC)}

\newcounter{wellPosednessHypotheses-g}
\newcounter{singularLimitHypotheses-g}
\newcounter{singularLimitHypotheses-c}
\newcounter{fConditions}

\title[HS with surface tension \& undercooling from parabolic PKS with nonlinear diffusion]{Hele-Shaw flow with surface tension \& kinetic undercooling as a sharp interface limit of a fully parabolic Patlak-Keller-Segel system with nonlinear diffusion}
\author{Michael Rozowski}

\address{Department of Mathematical Sciences, Carnegie Mellon University, Pittsburgh, PA 15213, United States.}
\email{mrozowsk@andrew.cmu.edu}

\thanks{Supported in part by NSF grants
	DMS-2009236 and DMS-2307342.}

\begin{document}	

\begin{abstract}
A large population limit of the parabolic-parabolic Patlak-Keller-Segel (PKS) system with degenerate, nonlinear diffusion, e.g., of porous medium-type $-\frac{m}{m-1}\div(\rho \nabla \rho^{m-1})$, is studied. We show, asymptotically, a sharp interface develops separating a region containing organisms arranged in a constant-in-time, uniform density from a region without organisms. Under an energy convergence hypothesis, we prove the emergent interface evolves according to a Hele-Shaw free boundary problem with surface tension and kinetic undercooling, and the free boundary satisfies a contact angle-type condition with the fixed boundary. 

Further, we show that, for well-prepared initial data, phase separation in these systems is, roughly, the result of some compatibility between an antiderivative for the population pressure and the convex conjugate of an antiderivative of the chemical destruction kinetics. When compatible, an energy for which the parabolic-parabolic PKS system is a gradient flow is a penalized Modica-Mortola functional. 
\end{abstract}

\maketitle

\section{Introduction}
\subsection{Large population limits of Patlak-Keller-Segel systems}
This work is concerned with the \emph{parabolic-parabolic Patlak-Keller-Segel system} \cite{Keller.Segel_1970_InitiationSlimeMold,Keller.Segel_1971_ModelChemotaxis,Patlak_1953_RandomWalkPersistence} (abbreviated parabolic-parabolic PKS system)
with degenerate, nonlinear diffusion and strictly increasing chemical destruction, the motivating examples of which are porous medium diffusion ($m>1$) and power law destruction ($q \ge 2$):
\begin{equation}\label{eq:PKS}
	\left\{
	\begin{aligned}
		&\alpha \partial_t n - \lap(n^m) + \chi \pdiv{n \nabla S} = 0, \\ 
		&\beta \partial_t S - \delta \lap S = n - \sfc S^{q-1}.  
	\end{aligned}
	\right.
\end{equation} 
In this model, $n$ denotes the density of some microorganisms and  $S$ the concentration of some chemical signal produced by the organisms and towards which they are attracted ($\chi>0$, chemoattractant).
The parameters $\alpha,\beta \ge 0$ are constants denoting the timescales for the relaxation of the organism density and chemoattractant concentration, respectively. 
The chemotactic sensitivity $\chi > 0$ is a constant (in this work) that quantifies the strength of the chemoattraction; $\delta > 0$ is the diffusivity of the chemoattractant in the medium; and $\sfc > 0$ is the medium's intrinsic rate of destruction of the chemoattractant. 
The nonlinear diffusion term models the mutual repulsion of the organisms. By rewriting it as $-\lap (n^m) = -\frac{m}{m-1}\pdiv{n\nabla n^{m-1}}$, the function $n \mapsto \frac{m}{m-1}n^{m-1}$ is interpreted as a ``population pressure'' field among the organisms that can be attributed to their natural incompressibility. This nonlinear diffusion is in accordance with Darcy's law: the organisms diffuse in a manner that decreases their local pressure.

The simultaneous, antagonistic effects of nonlinear diffusion due to the pressure $-\frac{m}{m-1}\pdiv{n \nabla n^{m-1}}$ and aggregation due to the attractive chemotaxis $\chi\pdiv{n\nabla S}$ are responsible for various lines of inquiry, especially in the case $q=2$. Of particular interest has been the possibility of finite time blow-up of $n$, which corresponds to the existence of a time $T^\ast > 0$ for which $\lim_{t \uparrow T^\ast} \norm{n(t,\cdot)}_{L^\infty(\Omega)} = +\infty$. 

In dimension $d$ with $q=2$ and $\alpha,\beta>0$, it is known that the existence of global-in-time  weak solutions is guaranteed whenever $m > 2 - \frac{2}{d}$ (see  \cite{Mimura_2017_VariationalFormulationFully} when $\phi$ has homogeneous Dirichlet conditions for the parabolic-parabolic system). Otherwise, blow-up can occur for sufficiently large initial mass of organisms $\int n_\init \,dx$ whenever $m \le 2 - \frac{2}{d}$ (see \cite{Cieslak.Stinner_2012_FinitetimeBlowupGlobalintime} for $d \ge 3$ with supercritical diffusion or \cite{Laurencot.Mizoguchi_2017_FiniteTimeBlowup} in a ball or the full space for $d=3,4$ with critical diffusion).  Similarly for the so-called \emph{parabolic-elliptic PKS system} ($\alpha>0$, $\beta=0$), it is known \cite{Sugiyama_2006_GlobalExistenceSubcritical,Bedrossian.Rodriguez.ea_2011_LocalGlobalWellposedness,Bedrossian.Rodriguez_2014_InhomogeneousPatlakKellerSegelModels} that blow-up may occur when $m<2-\frac{2}{d}$, while solutions are in $L^\infty(0,T;L^\infty(\Omega))$ when $m>2-\frac{2}{d}$. In the \emph{elliptic-parabolic PKS system} ($\alpha=0$, $\beta>0$), solutions are also in $L^\infty(0,T;L^\infty(\Omega))$ whenever $m>2$ \cite{Mellet.Rozowski_2024_VolumepreservingMeancurvatureFlow}, although the same critical exponent $m=2 - \frac{2}{d}$ is expected to apply in this situation.

\medskip
Several recent works (see \cite{Kim.Mellet.ea_2023_DensityconstrainedChemotaxisHeleShaw,Mellet_2024_HeleShawFlowSingular,Mellet.Rozowski_2024_VolumepreservingMeancurvatureFlow} or the survey \cite{Kim.Mellet.ea_2024_AggregationdiffusionPhenomenaMicroscopic}) have shown with linear destruction ($q=2$) and sufficiently stiff pressure ($m>2$), so global-in-time weak solutions are known, that phase separation occurs in a singular limit corresponding formally to that of a large population observed at an appropriate scale of time and space. Phase separation in this context refers to the emergence of sharp interfaces that partition the organisms' environment $\cl\Omega$ into evolving regions $E(t) \subset \cl\Omega$ of constant organism density $\rho_+>0$. That is, rescaled solutions of \eqref{eq:PKS} converge to $\rho(t,x) = \rho_+ \chi_{E(t)}(x)$. To understand the migration of the  population, these recent works derive an equation governing the motion of the free boundary $\partial E(t)$ that partitions the environment $\cl\Omega$. 

To introduce the appropriate rescalings of \eqref{eq:PKS} for the singular limits, one formally considers a family of solutions of \eqref{eq:PKS} for which the total population of organisms is large. Since the population is conserved, this corresponds to large initial data
\[
\int n_\varepsilon^\init({\bar x}) \,d\bar{x} = \varepsilon^{-d} \quad \text{with}\quad 0 < \varepsilon \ll 1.
\]
Then, one rescales the independent variables $(\ol{t}, \ol{x})$ to obtain zoomed-out variables $t\coloneq \varepsilon^2 \ol{t}/\tau$ and  $x\coloneq \varepsilon \ol{x}$, where $\tau>0$ is an adjustable time scale of an observer relative to the  parabolic scaling $\tau=1$. The functions $\rho_\varepsilon(t,x) \coloneq n_\varepsilon(\bar{t},\bar{x})$ and  $\phi_\varepsilon(t,x) \coloneq S_\varepsilon(\bar{t},\bar{x})$ then solve
\begin{equation}\label{eq:PKS-scaling}
	\begin{cases}
		\alpha \tau^{-1}  \partial_t \rho_\varepsilon - \lap (\rho_\varepsilon^m) + \chi \pdiv{\rho_\varepsilon \nabla \phi_\varepsilon} = 0 & \text{ in }  \left]0,+\infty\right[\times \Omega, \\
		\beta  \tau^{-1} \partial_t\phi_\varepsilon - \delta \Delta \phi_\varepsilon= \varepsilon^{-2}\left( \rho_\varepsilon -\sfc\phi_\varepsilon^{q-1} \right) & \text{ in }   \left]0,+\infty\right[\times \Omega.
	\end{cases}
\end{equation} 
The equations are now set in a bounded domain $\Omega$ with no-flux boundary conditions
\begin{equation}\label{eq:BCs}
	(-\nabla(\rho_\varepsilon^m)   +\chi  \rho_\varepsilon\nabla \phi_\varepsilon)\cdot \vec{n} = 0, \qquad \nabla \phi_\varepsilon\cdot \vec{n}=0  \quad \mbox{on } \partial\Omega,
\end{equation}
where $\vec{n}$ is the outer unit normal to the fixed boundary $\partial\Omega$, and given initial conditions
\[
\rho_\varepsilon(0,x)=\rho_\varepsilon^\init(x), \qquad \phi_\varepsilon(0,x) = \phi_\varepsilon^\init(x) \qquad\mbox{ in } \Omega.
\]
As a consequence of the rescaling, the size of the population is
\[
\forall \, \varepsilon > 0 \qquad \int_{\Omega} \rho_\varepsilon^\init(x)\, dx=1.
\]
The formality of this rescaling is due to the fact that after rescaling space we ultimately set the equations in a fixed bounded domain $\Omega$ independent of $\varepsilon$; otherwise, we must study a family of domains $\Omega_\varepsilon$. When the environment is dilation invariant, e.g.,  the full space $\Omega = \bbR^d$ or a wedge, this rescaling is not formal. 

\medskip
Depending on the time scale at which the system is observed $\tau_\varepsilon$ and the relative time scales of relaxation $\alpha_\varepsilon/\beta_\varepsilon$, various dynamics governing the evolution of $\partial E(t)$ have been derived in the large population limit $\varepsilon \to 0^+$ when $q=2$ and $m>2$. For $\beta_\varepsilon = 0$ and $\tau_\varepsilon = \alpha_\varepsilon/\varepsilon$ (the limiting case of $\alpha_\varepsilon/\beta_\varepsilon \gg \varepsilon$ and $\tau_\varepsilon \sim \alpha_\varepsilon/\varepsilon$) the system \eqref{eq:PKS-scaling} reduces to a \emph{parabolic-elliptic PKS system}, and \citeauthor{Mellet_2024_HeleShawFlowSingular} showed \cite{Mellet_2024_HeleShawFlowSingular} that the region $E(t)$ containing the limiting distribution of organisms   $\lim_{\varepsilon\to0^+}\rho_\varepsilon(t)$ evolves according to a \emph{Hele-Shaw free boundary problem with surface tension}:
\begin{equation} \label{eq:HS-ST}
	\text{For all } t >0, \qquad 	
	\begin{cases}
		\lap p(t,x) = 0 & \text{for all } x \in  E(t), \\
		\nabla p(t,x)\cdot \vec{n}(x) = 0 & \text{for all } x \in \partial\Omega, \\
		\rho_+V(t,x) = - \nabla p(t,x) \cdot \vec{\nu}(t,x) & \text{for all } x \in \big(\partial E(t)\big) \cap \Omega, \\
		p(t,x)  = \rho_+ \gamma_0 \chi\kappa(t,x) & \text{for all } x \in  \big(\partial E(t)\big) \cap \Omega, \\
		\vec{\nu}(t,x) \cdot \vec{n}(x) = 0 & \text{for all } x \in \partial_{\partial\Omega}\big[ \big(\partial E(t)\big) \cap \partial\Omega \big].
	\end{cases}
\end{equation} 
Above, $p$ is a pressure \emph{related} to a potential $\big(\frac{m}{m-1} \rho_\varepsilon^{m-1} - \chi\phi_\varepsilon\big)/\varepsilon$ for the organism flux  $ -\rho_\varepsilon\nabla\big( \frac{m}{m-1}\rho_\varepsilon^{m-1} - \chi\phi_\varepsilon \big)/\varepsilon$ (see \eqref{eq:approximatePressure}); $V(t,x)$ is the normal component of the velocity of the free boundary in the direction $\vec{\nu}(t,x)$ of the outer unit normal at $x \in \partial E(t)$; $\kappa(t,x)$ is the scalar mean-curvature (sum of principle curvatures) of the free boundary at $x\in\partial E(t)$ (with the sign convention that $\kappa(t,x)>0$ wherever $\partial E(t)$ is locally convex); $\rho_+>0$ is the asymptotic density of organisms; and $\gamma_0>0$ is a constant describing the strength of surface tension. By $\partial_{\Sigma}(U \cap \Sigma)$, we mean the boundary of the set $U \cap \Sigma$ relative to the subspace topology of $\Sigma$ in $\bbR^d$. The final equation is a contact angle condition: the free boundary must meet the fixed boundary orthogonally.

In an opposite regime $\alpha_\varepsilon=0$ and $\tau_\varepsilon = \beta_\varepsilon$ (the limiting case of $\alpha_\varepsilon/\beta_\varepsilon \ll \varepsilon$ and $\tau_\varepsilon \sim \beta_\varepsilon$), the system reduces to an \emph{elliptic-parabolic PKS system}, and Mellet and this author showed \cite{Mellet.Rozowski_2024_VolumepreservingMeancurvatureFlow}  that the region $E(t)$ evolves by a \emph{volume-preserving mean-curvature flow}:
\begin{equation} \label{eq:VPMCF}
	\text{For all } t > 0, \qquad 	
	\begin{cases}
		V(t,x) = -\kappa(t,x) + \Lambda(t) &\text{ for all } x \in \big(\partial E(t)\big) \cap \Omega,  \\
		\abs{E(t)} = \abs{E(0)}, \\
		\vec{\nu}(t,x)\cdot \vec{n}(x) = 0 & \text{ for all } x \in \partial_{\partial\Omega}\big[ \big(\partial E(t)\big) \cap \partial\Omega \big].
	\end{cases}
\end{equation} 
Above, $\Lambda(t)$ is a Lagrange multiplier for the volume constraint $\abs{E(t)} = \abs{E(0)}$ and is constant in space; it is the limit of a potential for the organism velocity. 

This work examines an intermediate regime $\alpha_\varepsilon = \beta_\varepsilon\varepsilon$ and $\tau_\varepsilon = \beta_\varepsilon$ (the limiting case of $\alpha_\varepsilon/\beta_\varepsilon \sim \varepsilon$ and $\tau_\varepsilon \sim \beta_\varepsilon$), and its main contribution rigorously establishes the result announced in \cite{Mellet.Rozowski_2024_VolumepreservingMeancurvatureFlow} that, in the limit $\varepsilon\to0^+$, solutions of \eqref{eq:PKS-scaling} asymptotically satisfy a Hele-Shaw free boundary problem with surface tension and kinetic undercooling. 

\medskip
We set $\alpha_\varepsilon/\tau_\varepsilon = \varepsilon\alpha_0$, $\tau_\varepsilon =\beta_\varepsilon/\beta_0$ in \eqref{eq:PKS-scaling}, and we are led to the following system (which is a \emph{parabolic-parabolic PKS system} with nonlinear diffusion):
\begin{equation}   \tag{$P_\varepsilon$} \label{eq:PP-PKSeps}
	\begin{cases}
		\displaystyle\alpha_0\partial_t\rho_\varepsilon - \frac{1}{\varepsilon}\pdiv{\rho_\varepsilon \nabla f'(\rho_\varepsilon)} + \frac{\chi}{\varepsilon}\div(\rho _\varepsilon\nabla \phi_\varepsilon) = 0 \quad & \text{in }  \left]0,+\infty\right[\times \Omega, \\
		\displaystyle\varepsilon\beta_0\partial_t \phi_\varepsilon - \varepsilon\pdiv{\sfA(x)\nabla\phi_\varepsilon}= \frac{1}{\varepsilon}\left( \rho_\varepsilon - \sfc(x)g'(\phi_\varepsilon) \right) & \text{in }  \left]0,+\infty\right[\times \Omega, \\
		\displaystyle(-\rho_\varepsilon\nabla f'(\rho_\varepsilon) + \rho_\varepsilon\nabla\phi_\varepsilon) \cdot \vec{n} = \sfA(x)\nabla\phi_\varepsilon \cdot \vec{n} = 0	&	\text{on } \left]0,+\infty\right[ \times \partial\Omega, \\
		\displaystyle\rho_\varepsilon(0,\cdot) = \rho^\init_\varepsilon, \qquad \phi_\varepsilon(0,\cdot) = \phi^\init_\varepsilon & \text{in } \Omega.
	\end{cases}
\end{equation} 
For the purpose of describing the intermediacy of this model's sharp interface limit to the aforementioned models' sharp interface limits, we have introduced timescales $\alpha_0,\beta_0>0$ independent of $\varepsilon$. We delay stating the precise assumptions on $f,\sfA,\sfc,g$ until Section \ref{sec:mainResults}, but we point out that, as in earlier works on this sharp interface limit, we consider a class of pressure laws $f'$ including the power law we mentioned earlier. In contrast to prior works, we admit the possibility that the medium's intrinsic destruction rate of the chemoattractant $\sfc$ may vary over the environment $x \mapsto \sfc(x)$ and that the destruction kinetics may be nonlinear $v \mapsto g'(v)$. We will require that the diffusion coefficient $x \mapsto \sfA(x) \in \bbR^{d\times d}$ be uniformly elliptic; only the diffusion of the organisms is degenerate.

One aim of this paper is to highlight that phase separation in PKS systems is a consequence of some compatibility between $f$ and $g$. Previous works on these systems' large population/sharp interface limits have considered, on the one hand, only the case of $\sfc(x) \equiv \ubar{\sfc}>0$ and $g(v) = v^2/2$. 
On the other hand, they allow rather general $f = f(u)$, provided 
\begin{enumerate}
	\item $u\mapsto f(u) - u^2/(2\ubar{\sfc})$ forms a double-well potential;
	\item $f$ satisfies more than quadratic growth as $u \to +\infty$, i.e., there exists some $m>2$, and $K,R > 0$ such that $f(u) \ge Ku^m$ whenever $u>R$.
\end{enumerate}
In this work, we will show that phase separation in these large population limits occurs whenever $f$ and $g$ are compatible in the following rough sense:
\begin{enumerate}
	\item[(1\textsuperscript{$\circ$})] $u \mapsto f(u) - \ubar{\sfc}\,g^\ast(u/\ubar{\sfc})$ forms a double-well potential, where $g^\ast$ is the Legendre transform\footnote{The Legendre transform of a convex function $h \colon \bbR \to \bbR \cup \set{+\infty}$ is $h^\ast(u) \coloneq \sup_{v\in\bbR}\set{uv - h(v)}$.} of $g$;
	\item[(2\textsuperscript{$\circ$})] $f$ grows faster than $g^\ast$ as $u \to +\infty$.
\end{enumerate} 
As a particular example, we show that if\footnote{We denote by $\iota_C$ the indicator function of the set $C$, where $\iota_C(u)\coloneq0$ if $u\in C$ and $\iota_C(u)\coloneq+\infty$ if $u\not\in C$.} $f(u) = u^m/(m-1) + \iota_{\left[0,+\infty\right[}(u)$ with $m>1$ and $g(v) = \abs{v}^q/q$ with $q\ge2$, then phase separation occurs in the limit $\varepsilon\to0^+$ provided $m > q' \coloneq q/(q-1)$, which recovers the previously known condition of $m>2$ in the case $g(v) = v^2/2$. 

Another aim is to show that $\sfc$ is responsible for a confinement effect in the limit $\varepsilon\to0^+$. More precisely, if $\sfc$ satisfies some nondegeneracy conditions, the organisms asymptotically aggregate into a region $E(t)$ satisfying $E(t) \subset \set{\sfc = {\ubar\sfc}}$ for each $t$, where ${\ubar\sfc} \coloneq \min_{x\in\cl\Omega}\sfc(x) > 0$; that is, the organisms aggregate on the region in which the chemical towards which they are attracted decays the most slowly. 

This paper's main purpose is rigorously exhibit the convergence as $\varepsilon\to0^+$ of a solution of \eqref{eq:PP-PKSeps} to a solution of the following \emph{Hele-Shaw free boundary problem with surface tension and kinetic undercooling}:
\begin{equation}  \tag{$P_0$}  \label{eq:HS-STKU}
	\text{For all } t > 0, \,	
	\left\{\begin{aligned}
		&\lap p(t,x) = 0 && \text{for all } x \in  E(t), \\
		& \nabla p(t,x) \cdot \vec{n}_0(x) = 0 && \text{for all } x\in\partial\Omega_0, \\
		&\rho_+V(t,x) = - \frac{1}{\alpha_0}\nabla p(t,x) \cdot \vec{\nu}(t,x) && \text{for all } x \in \big(\partial E(t)\big) \cap \Omega_0, \\
		&p(t,x) + \frac{\beta_0\gamma_0\chi}{\alpha_0} \nabla p(t,x) \cdot \vec{\nu}(t,x) = \rho_+ \gamma_0 \chi \kappa_\sfA(t,x) && \text{for all } x \in  \big(\partial E(t)\big) \cap \Omega_0, \\
		&\frac{\sfA(x)\vec{\nu}(t,x)}{\abs{\vec{\nu}(t,x)}_{\sfA(x)}}\cdot \vec{n}(x) = 0 && \text{for all } x \in \partial_{\partial\Omega}\big[\big(\partial E(t)\big) \cap \partial\Omega\big], \\
		&\frac{\sfA(x) \vec{\nu}(t,x)}{\abs{\vec{\nu}(t,x)}_{\sfA(x)}}\cdot \vec{n}_0(x) = \abs{\vec{n}_0(x)}_{\sfA(x)} && \text{for all } x \in \partial_{(\partial\Omega_0) \cap \Omega} \big[ \big(\partial E(t)\big) \cap \big(\partial\Omega_0\big) \cap \Omega\big].
	\end{aligned}
	\right.
\end{equation}
For brevity, we have introduced the slow decay region $\Omega_0 \coloneq \set{\sfc = \ubar{\sfc}}$ where the organisms are asymptotically concentrated, its interior $\oc{\Omega}_0$, and its outer unit normal vector $\vec{n}_0$. For each $x\in\cl\Omega$, the function $\bbR^d \ni p \mapsto \abs{p}_{\sfA(x)}$ is the norm induced by the uniformly elliptic matrix $\sfA$, i.e., $\abs{p}_{\sfA(x)} \coloneq \sqrt{p \cdot \sfA(x)p}$.  We denote by $\kappa_\sfA(t,x)$ the anisotropic mean-curvature of the free boundary at $x \in \partial E(t)$. If $\sfA \equiv I_d$, it coincides with the scalar mean-curvature and we use the same sign convention as in \eqref{eq:HS-ST}. Otherwise, we introduce the surface tension function $\sigma(x,p) \coloneq \abs{p}_{\sfA(x)}$ on $\cl\Omega \times \bbS^{d-1}$, and then 
\[
\kappa_\sfA(t,x) = \fdiv_{\partial E(t)} \nabla_p\sigma(x,\vec{\nu}(t,x)) + \nabla_x\sigma(x,\vec{\nu}(t,x)) \cdot \vec{\nu}(t,x).
\]
The fifth equation in \eqref{eq:HS-STKU} is the analogue of the contact angle condition from \eqref{eq:HS-ST}; the anisotropy and inhomogeneity of the diffusivity adjust it. It expresses the orthogonality between the fixed boundary's normal $\vec{n}(x)$ and the Cahn-Hoffman vector $\nabla_p\sigma(x,\vec{\nu}(t,x)) = \sfA(x)\vec{\nu}(t,x)/\abs{\vec{\nu}(t,x)}_{\sfA(x)}$ associated to the surface energy density $\sigma$ and oriented by $\vec{\nu}$. 

The sixth equation is a prescribed contact angle condition between the free boundary and the confining region's boundary $\partial\Omega_0$ in $\Omega$; it emerges from the confinement effect owed to $\sfc$. In fact, it reduces to a tangential contact angle condition: $\vec{\nu}(t,x)\cdot \vec{n}_0(x) = 1$. To see this, multiply by $\abs{\vec{\nu}(t,x)}_{\sfA(x)}$, which shows there is equality in the Cauchy-Schwarz inequality for the inner product $(p,q) \mapsto q \cdot \sfA(x)p$ induced by $\sfA(x)$, which means $\sfA^{1/2}(x) {\nu}(t,x)$ and $\sfA^{1/2}(x) \vec{n}_0(x)$ are parallel. Since $\sfA^{1/2}(x)$ is invertible, $\abs{\vec{\nu}(t,x)} = \abs{\vec{n}_0(x)}=1$, and each of $\vec{\nu}(t,x)$ and $\vec{n}_0(x)$ is outward directed, the only possibility is $\vec{\nu}(t,x) = \vec{n}_0(x)$. In contrast to where the free boundary meets $\partial\Omega$, neither the anisotropy nor the inhomogeneity of the diffusivity $\sfA(x)$ affect the contact angle along $(\partial\Omega_0)\cap\Omega$. We prefer to write the condition as in \eqref{eq:HS-STKU} as opposed to $\vec{\nu}(t,x)\cdot\vec{n}_0 = 0$ to make clearer the connection between the contact angles that emerge and the limiting energy \eqref{eq:G0}; see also \eqref{eq:G0-differentBCs}. These last two equations in \eqref{eq:HS-STKU} are in agreement with Young's law in an anisotropic setting; see \cite{Philippis.Maggi_2015_RegularityFreeBoundaries}.

\subsection{Intermediacy of \eqref{eq:HS-STKU} to \eqref{eq:HS-ST} and \eqref{eq:VPMCF}}
Suppose for a moment that $\sfA \equiv I_d$ and $\sfc \equiv \ubar{\sfc} > 0$. 

In the system \eqref{eq:PP-PKSeps}, neither the chemoattractant concentration nor the organism density is in a quasi-steady state, so this system is in a sense intermediate to both the classical parabolic-elliptic PKS model ($\alpha>0$, $\beta=0$, $q=2$ in \eqref{eq:PKS}) and the elliptic-parabolic PKS model ($\alpha=0$, $\beta>0$, $q=2$) \cite{Mellet.Rozowski_2024_VolumepreservingMeancurvatureFlow}. We would like to describe how its sharp interface limit \eqref{eq:HS-STKU} is also, in a sense, intermediate to both of these systems' sharp interface limits, \eqref{eq:HS-ST} and  \eqref{eq:VPMCF}, respectively.

Formally, in the limit of fast chemical relaxation $\beta_0 \to 0^+$, \eqref{eq:HS-STKU} converges to \eqref{eq:HS-ST}, which is consistent with the formal convergence of the corresponding diffuse interface approximations, namely the parabolic-parabolic system \eqref{eq:PP-PKSeps} to the parabolic-elliptic system ($\alpha_0>0$, $\beta_0=0$).

Further, one can formally recover a volume-preserving mean-curvature flow in the limit of fast organism relaxation $\alpha_0 \to 0^+$, but not necessarily \eqref{eq:VPMCF}. In particular, we claim that the rescaled pressure $(\rho_+\gamma_0\chi)^{-1}p_{\alpha_0}$ should converge to a Lagrange multiplier $\Lambda(t,x)$ that preserves the volume of the individual connected components of $E(t)$. This is distinct from the Lagrange multiplier we derive in \cite{Mellet.Rozowski_2024_VolumepreservingMeancurvatureFlow} that preserves the total volume $\abs{E(t)}$ but not necessarily that of different connected components (since $\Lambda$ is independent of $x$). On the one hand, the fourth equation of \eqref{eq:HS-STKU} indicates that the normal derivative of $p_{\alpha_0}$ along the free boundary vanishes as $\alpha_0\to0^+$, and combining this with the fact that $p_{\alpha_0}(t)$ is harmonic in $E(t)$ implies that a limit $p_0(t)$ should be constant on each connected component of $E(t)$. On the other hand, eliminating the normal derivative of $p_{\alpha_0}$ from the fourth equation using the third equation shows all terms in the equality $\beta_0 V_{\alpha_0} = -\kappa_{\alpha_0} + (\chi\gamma_0\rho_+)^{-1}p_{\alpha_0}$ are of the same order in $\alpha_0$, so this equality should pass to the limit $\alpha_0\to0^+$, which would yield $\beta_0V_0 = -\kappa_0 + p_0$. This convergence would also be consistent with the formal convergence of the corresponding diffuse interface approximations as $\alpha_0\to0^+$, i.e., of the parabolic-parabolic system to the elliptic-parabolic system ($\alpha_0=0$, $\beta_0>0$).

\subsection{Setting some constants to \texorpdfstring{$1$}{1}}

While keeping the constants $\alpha_0,\beta_0,\chi$ in \eqref{eq:PP-PKSeps} allows to see how they manifest in the limiting problem \eqref{eq:HS-STKU}, it will be less cumbersome to work without them. We set 
\begin{gather*}
	\tilde{t}\coloneq \frac{t}{\beta_0}, \qquad \tilde{\rho}(\tilde{t},x) \coloneq \sqrt{\chi}\rho(\beta_0 \tilde{t},x), \qquad \tilde{\phi}(\tilde{t},x) \coloneq \sqrt{\chi}\phi(\beta_0 \tilde{t},x), \\
	\tilde{f}(u) \coloneq f\big(\frac{u}{\sqrt{\chi}}\big), \qquad \tilde{g}(v) \coloneq g\big(\frac{v}{\sqrt{\chi}}\big),\qquad  \tilde{\sfc} \coloneq \chi\sfc, \qquad \tilde{\alpha}_0 \coloneq \frac{\alpha_0}{\beta_0},
\end{gather*}
which produces the system 
\begin{equation*}   
	\begin{cases}
		\displaystyle \tilde{\alpha}_0 \partial_{\tilde{t}} \tilde{\rho}_\varepsilon - \frac{1}{\varepsilon}\pdiv{\tilde{\rho}_\varepsilon \nabla \tilde{f}'(\tilde{\rho}_\varepsilon)} + \frac{1}{\varepsilon}\div(\tilde{\rho}_\varepsilon\nabla \tilde{\phi}_\varepsilon) = 0, \\
		\displaystyle\varepsilon\partial_{\tilde{t}} \tilde{\phi}_\varepsilon - \varepsilon\pdiv{\sfA(x)\nabla\tilde{\phi}_\varepsilon}= \frac{1}{\varepsilon}\left( \tilde{\rho}_\varepsilon - \tilde{\sfc}(x)\tilde{g}'(\tilde{\phi}_\varepsilon) \right).
	\end{cases}
\end{equation*} 
This is just \eqref{eq:PP-PKSeps} with $\chi = \beta_0 = 1$, so we will state and prove our results for this case.
\subsection{Outline of the paper}
In Section \ref{sec:mainResults}, we state the assumptions on the functions $f,\sfA,\sfc,g$ appearing in \eqref{eq:PP-PKSeps} and we give precise statements of our convergence results. Section \ref{sec:gammaConvergence} is dedicated to proving that a energy for which \eqref{eq:PP-PKSeps} is a gradient flow $\Gamma$-converges to a weighted perimeter functional. 

The remainder of the paper is dedicated to the proof of the convergence as $\varepsilon\to0^+$ of a weak solution to \eqref{eq:PP-PKSeps} to a $BV$ notion of solution for \eqref{eq:HS-STKU}. Section \ref{subsec:continuityEq} shows the standard result that, viewing the first equation of \eqref{eq:PP-PKSeps} as a continuity equation with velocity $\vec{v}_\varepsilon$, a subsequence of solutions $\seq{(\rho_\varepsilon, \vec{v}_\varepsilon)}$ converges to a solution $(\rho,\vec{v})$ of another continuity equation, which is the first part of our weak formulation of \eqref{eq:HS-STKU}. Section \ref{subsec:phaseSeparation} is devoted to the proof of phase separation of $(\rho_\varepsilon, \phi_\varepsilon)$. Then, Section \ref{sec:normalVelocity} uses an energy convergence assumption to show the existence of a normal velocity $V$ of the free boundary, and then Section \ref{subsec:derivationUndercoolingCurvature} derives this and a distributional formulation of its anisotropic mean-curvature $\kappa_\sfA$  from \eqref{eq:PP-PKSeps}. The former of these is shown to be equal to the normal component of $\vec{v}$ in a weak sense, which is the second part of our weak formulation of \eqref{eq:HS-STKU}. Finally, we derive the pressure $p$ and the final part of our weak formulation of \eqref{eq:HS-STKU} in Section \ref{subsec:derivationOfPressureEq}.

\subsection{Acknowledgements}
This paper was completed as part of the author's dissertation advised by Antoine Mellet at the University of Maryland, College Park. The author would like to extend warm thanks to Antoine for his patience and generosity during the completion of this work.
\section{Main results} \label{sec:mainResults}

\subsection{Assumptions on coefficients and nonlinearities} \label{subsec:assumptions}
For the diffusion coefficient $\sfA$, we suppose  
\begin{conditionsOnA}
	\item \label{hyp-on-A:regularity} $\sfA = (\sfA^{i,j})_{i,j=1}^d$ is a $(d \times d)$-symmetric matrix-valued function with each component $\sfA^{i,j}$ in $C^1(\cl\Omega)$.
	\item \label{hyp-on-A:ellipticity} There exists a constant ${\ubar \sfA} > 0$ such that for all $x \in \cl\Omega$ and $p \in \bbR^d$ we have $p \cdot \sfA(x)p \ge {\ubar \sfA}\abs{p}^2$.
\end{conditionsOnA}
As a consequence of \ref{hyp-on-A:regularity}, there exists a constant ${\ol \sfA} > 0$ such that $p \cdot \sfA(x)p \le {\ol \sfA}\abs{p}^2$ for all $p \in \bbR^d$. In view of \ref{hyp-on-A:ellipticity}, $\sfA(x)$ has a unique symmetric, positive definite square root $\sfA^{1/2}(x)$, and 
\begin{equation} \label{eq:Aip}
	\bbR^d \times \bbR^d \ni (p,q) \mapsto q \cdot \sfA(x)p \in \bbR, \qquad
	\bbR^d \ni p \mapsto \abs{p}_{\sfA(x)} \coloneq \sqrt{p \cdot \sfA(x)p} = \abs{\sfA^{1/2}(x)p}
\end{equation} 
are (resp.) an inner product and its induced norm.

For $\sfc$, we delay stating more precise technical assumptions until later. 
\begin{conditionsOnC}
	\item \label{hyp-c:regularity} $\sfc \in C^1(\cl\Omega)$ and ${\ubar \sfc} \coloneq \min_{x\in\cl\Omega} \sfc(x) > 0$.
	\setcounter{singularLimitHypotheses-c}{\value{conditionsOnCi}}
\end{conditionsOnC}

\medskip
The following assumptions on the nonlinearities $f$ and $g$ are sufficient for the existence of the weak solutions of \eqref{eq:PP-PKSeps} we use and ensure that $f$ and $g$ are compatible for phase separation. 
\begin{conditionsOnG}
	\item \label{hyp-g:regularity} $g \in C^{1,1}_\loc(\bbR)$ is  strictly convex and $g(0) = g'(0) = 0$. 
	\item \label{hyp-g:growthCondition} There exists $q>1$ and $K_1,R_1>0$ such that $g(v) \ge K_1\abs{v}^q$ for all $\abs{v} > R_1$. \setcounter{singularLimitHypotheses-g}{\value{conditionsOnGi}}
\end{conditionsOnG}
Since only $g'$ enters the equation  \eqref{eq:PP-PKSeps}, we can  add a constant to $g$ to reduce to the case $g(0) = 0$, so this is no restriction. The condition $g'(0) = 0$ is natural from a modeling point-of-view: the rate of chemoattractant destruction should vanish if none is present. Similarly, the strict convexity of $g$ ensures $g'$ is increasing, so the rate of chemical destruction increases with concentration.
For $f$ we require 
\begin{conditionsOnF}
	\item \label{hyp-f:regularity} $f \in C(\left[0,+\infty\right[) \cap C^1(\left]0,+\infty\right[)$ is strictly convex, satisfies $f(0)= \displaystyle\lim_{u \to 0^+}f'(u) = 0$, and is extended to $\left]-\infty,0\right[$ by $+\infty$.
	\item \label{hyp-f:growthCondition} With $q$ from \ref{hyp-g:growthCondition}, there exists $m > q' \coloneq q/(q-1)$ and $K_2,R_2 > 0$ such that $f(u) \ge K_2 u^m$ for all $u > R_2$. \setcounter{fConditions}{\value{conditionsOnFi}}
\end{conditionsOnF}
In view of the mass constraint $\int_\Omega \rho\,dx=1$ and that only $\nabla f'(\rho)$ enters the equation \eqref{eq:PP-PKSeps}, the function $f$ is defined up to an affine function, so $f(0) = \lim_{u\to0^+}f'(u) = 0$ is no restriction.

We need to assume that $f$ and $g$ are compatible for phase separation to occur. This compatibility is described, in part, by the ordering of growth rates $m>q'$ from the coercivity conditions in \ref{hyp-g:growthCondition}, \ref{hyp-f:growthCondition}. We must further suppose 
\begin{compatibilityCondition}
	\item \label{hyp:compatibilityConditionF&G} $\displaystyle-\sfa \coloneq \min_{u \ge 0} \left\{\frac{f(u) - [{\ubar \sfc}\, g]^\ast(u)}{u}\right\} \in\left]-\infty,0\right[$, and there exists a unique minimizer $\rho_+ \in \left]0,+\infty\right[$.
\end{compatibilityCondition}
The number $-\sfa$ is the slope of a supporting line for the graph of $f - [{\ubar \sfc}\,g]^\ast$ that connects $(0,0)$ to exactly one other point on the graph; see Figure \ref{subfig:subVsSupercritical}. The existence of a minimizer in the definition of $-\sfa$ is already implied by \ref{hyp-f:regularity}, \ref{hyp-f:growthCondition}, \ref{hyp-g:regularity}, \ref{hyp-g:growthCondition}. 
The point is that there is a \emph{unique, positive minimizer} $\rho_+$ with a \emph{negative minimum}; that the minimum is nonpositive is implied by \ref{hyp-g:regularity}, \ref{hyp-f:regularity}. Given this compatibility condition \ref{hyp:compatibilityConditionF&G}, the function 
\begin{equation} \label{eq:potentialW}
	W \colon \bbR \to \left[0,+\infty\right[,\quad 	W(u) \coloneq f(u) - {\ubar\sfc}\,g^\ast\big(\frac{u}{{\ubar \sfc}}\big) + \sfa u
\end{equation}
is a nonnegative double-well potential with zeros at $\set{0,\rho_+}$; see Lemma \ref{lemma:potentialW} and Figure \ref{subfig:W}. If we omit the uniqueness of the minimizer $\rho_+$ in \ref{hyp:compatibilityConditionF&G}, it is possible to obtain multi-well potentials, which would result in multi-phase free boundary problems, the analysis of which is outside the scope of this paper.

\medskip
The quintessential examples of nonlinearities that our results apply to are power laws: 
\begin{equation} \label{eq:powerLawNonlinearities}
	f_m(u) \coloneq 
	\begin{cases}
		\displaystyle\frac{u^m}{m-1}	&	\text{if } u \ge 0, \\
		+\infty			&	\text{if } u < 0
	\end{cases}
	\quad\text{and}\quad g_q(v) \coloneq \frac{\abs{v}^q}{q} \quad \text{with } m > q' \coloneq \frac{q}{q-1}>1.
\end{equation}
Since $g' \in \Lip_\loc(\bbR)$ by \ref{hyp-g:regularity}, we must impose $q \ge 2$. Then \ref{hyp:compatibilityConditionF&G} is satisfied with 
\begin{equation} \label{eq:supportingSlopeAndDensity}
	\begin{gathered}
		-\sfa = \frac{f(\rho_+) - {\ubar \sfc}\, g^\ast(\rho_+/{\ubar \sfc})}{\rho_+} = \frac{1}{m-1}\rho_+^{m-1} - \frac{1}{{\ubar \sfc}^{q'-1} q'}\rho_+^{q'-1} < 0, \quad 
		\rho_+ = \Big[\frac{1}{{\ubar \sfc}} \Big(1 - \frac{1}{q'} \Big)^{\frac{1}{q'-1}}\Big]^{\frac{q'-1}{m-q'}} > 0.
	\end{gathered}
\end{equation}

\medskip
Most of these assumptions we will make throughout the paper. For  brevity, we christen them. 
\begin{standardAssumptions}
	\item \label{hyp-standardAssumptions} The \emph{standard assumptions} are satisfied when $\sfA$ satisfies \ref{hyp-on-A:regularity} and \ref{hyp-on-A:ellipticity}; $\sfc$ satisfies \ref{hyp-c:regularity}; $g$ satisfies \ref{hyp-g:regularity} and \ref{hyp-g:growthCondition}; and $f$ satisfies \ref{hyp-f:regularity} and \ref{hyp-f:growthCondition}.
\end{standardAssumptions}
The compatibility condition \ref{hyp:compatibilityConditionF&G} is necessary for the construction of $W$ \eqref{eq:potentialW}. To make clear when this construction is important, we exclude \ref{hyp:compatibilityConditionF&G} from \ref{hyp-standardAssumptions} and mention it explicitly.

\subsection{\texorpdfstring{$\Gamma$}{\textGamma}-convergence of the energy \texorpdfstring{$\scG_\varepsilon$}{G\_\textepsilon}} \label{subsec:GammaConvergence}
The system \eqref{eq:PP-PKSeps} has the structure of a gradient flow on the product $\bbW_2(\cl\Omega)\times L^2(\Omega)$, so it is natural to study the convergence of its energy as $\varepsilon \to 0^+$. Any $(\rho_\varepsilon, \phi_\varepsilon)$ satisfying a reasonable notion of solution for the system \eqref{eq:PP-PKSeps} should dissipate the following energy:
\begin{equation} \label{eq:unaugmentedEnergy}
	\scE_\varepsilon[\rho,\phi] \coloneq 
	\displaystyle\frac{1}{\varepsilon} \int_\Omega f(\rho) - \rho\phi + \sfc(x)g(\phi) \,dx + \frac{\varepsilon}{2}\int_\Omega \abs{\nabla\phi}_{\sfA(x)}^2 \,dx;
\end{equation}
namely,
\[
	\frac{d}{dt} \scE_\varepsilon[\rho_\varepsilon,\phi_\varepsilon] = - \int_\Omega \rho_\varepsilon \abs{\nabla \frac{\big(f'(\rho_\varepsilon) - \phi_\varepsilon\big)}{\varepsilon}}^2 + \varepsilon \big(\partial_t \phi_\varepsilon\big)^2 \,dx \le 0.
\]
Essential to the singular limit of the evolution equation we study is the double-well potential property of the difference $u \mapsto f(u) - [{\ubar \sfc}\,g]^\ast(u)$. We thus add and subtract each of  $[{\ubar \sfc}\,g]^\ast(\rho)$ and ${\ubar \sfc}\,g(\phi)$ to the first integrand of $\scE_\varepsilon$:
\begin{align*}
	\scE_\varepsilon[\rho,\phi] = \frac{1}{\varepsilon}\int_\Omega \big( f(\rho) - [{\ubar \sfc}\,g]^\ast(\rho) \big) + \big([{\ubar \sfc}\,g]^\ast(\rho) -\rho\phi + {\ubar \sfc}\, g(\phi) \big) \,dx  
	+ \frac{\varepsilon}{2} \int_\Omega \abs{\nabla\phi}_{\sfA(x)}^2\,dx + \frac{1}{\varepsilon} \int_\Omega (\sfc(x) - {\ubar \sfc})g(\phi) \,dx. 
\end{align*}
The second and third integrals are clearly nonnegative (recall \ref{hyp-on-A:ellipticity}, \ref{hyp-g:regularity}), and the second term in the first integrand is nonnegative by the Fenchel-Young inequality. However, even with this double-well potential property, the functional $\scE_\varepsilon$ is not  bounded below for small $\varepsilon$, so we use the constant $\sfa$ defined by the compatibility condition  \ref{hyp:compatibilityConditionF&G} and the mass constraint $\int_\Omega \rho \,dx = 1$ to introduce the constant 
\begin{equation}
	\scA_\varepsilon \coloneq \frac{\sfa}{\varepsilon} = \frac{1}{\varepsilon} \int_\Omega \sfa \rho \,dx,
\end{equation}
and we work instead with the augmented ``centered'' energy 
\begin{equation} \label{eq:augmentedEnergy}
	\begin{aligned}
		\scG_\varepsilon[\rho,\phi] \coloneq \scE_\varepsilon[\rho,\phi] + \scA_\varepsilon 
		&= \frac{1}{\varepsilon}\int_\Omega \underbrace{\big( f(\rho) - [{\ubar \sfc}\,g]^\ast(\rho) + \sfa\rho \big)}_{\eqcolon W(\rho)} + \big( [{\ubar \sfc}\,g]^\ast(\rho) - \rho\phi + {\ubar \sfc}\, g(\phi) \big) \,dx \\
		&\qquad+ \frac{\varepsilon}{2} \int_\Omega \abs{\nabla\phi}_{\sfA(x)}^2\,dx + \frac{1}{\varepsilon} \int_\Omega (\sfc(x) - {\ubar \sfc})g(\phi)\,dx.
	\end{aligned}
\end{equation}
The effect of this is the introduction of the nonnegative double-well potential $W$ with zeros at $\set{0,\rho_+}$ (recall \eqref{eq:potentialW} and see Figure \ref{fig:potentials}, Lemma \ref{lemma:potentialW}). Because we have only added a constant to the energy, $\scG_\varepsilon$ is also dissipated by solutions of \eqref{eq:PP-PKSeps}. In view of Lemma \ref{lemma:potentialW}-\ref{lemma:Wnonneg}, the functionals $\seq{\scG_\varepsilon}_{\varepsilon>0}$ are bounded below uniformly-in-$\varepsilon$ (and in particular nonnegative), so studying their $\Gamma$-convergence is now worthwhile. 

The term $\varepsilon^{-1}\int_\Omega [{\ubar \sfc}\,g]^\ast(\rho) - \rho\phi + {\ubar \sfc}\,g(\phi)\,dx$ is a penalty enforcing that the range of $(\rho,\phi)$ be near the graph of the subdifferential $\Graph(\partial [{\ubar \sfc}\,g]^\ast) = \cset{(u,v) \in \bbR_+ \times \bbR}{v \in \partial[{\ubar \sfc}\,g]^\ast(u)}$ that characterizes the equality case of the Fenchel-Young inequality $[{\ubar \sfc}\,g]^\ast(u) - uv + {\ubar \sfc}\,g(v) \ge 0$. Because the double-well potential $W$ already enforces that the range of $\rho$ be near the location of the wells $\set{0,\rho_+}$, this additional penalty enf

\medskip
The next calculation, although elementary, permits rewriting the energy $\scG_\varepsilon$ in a manner that presents it as a Modica-Mortola functional in the $\phi$ variable with a penalty $\varepsilon^{-1}\int_\Omega f(\rho) - \rho(\phi-\sfa) + f^\ast(\phi-\sfa)\,dx$ enforcing that the range of $(\phi,\rho)$ be near the region $\Graph(\partial f^\ast(\cdot-\sfa)) \coloneq \cset{(v,u) \in \bbR\times\bbR_+}{u \in \partial f^\ast(v-\sfa)}$. The following rewriting is at the heart of the emergence of  mean-curvature in the limit $\varepsilon\to0^+$ of \eqref{eq:PP-PKSeps}. By adding and subtracting $f^\ast(v-\sfa)$:
\begin{align*}
	W(u) + \big([{\ubar \sfc}\,g]^\ast(u) - uv + {\ubar \sfc}\,g(v)\big) 
	= \big(f(u) - u(v-\sfa) + f^\ast(v-\sfa)\big) + \big( {\ubar \sfc}\,g(v) - f^\ast(v-\sfa)  \big);
\end{align*}
therefore,
\begin{align}
	\scG_\varepsilon[\rho,\phi] 
	&=\frac{1}{\varepsilon}\int_\Omega{\ubar \sfc}\,g(\phi) - f^\ast(\phi-\sfa)  \,dx + \frac{\varepsilon}{2}\int_\Omega \abs{\nabla\phi}_{\sfA(x)}^2 +  \frac{1}{\varepsilon}\int_\Omega (\sfc(x) - {\ubar \sfc})g(\phi)\,dx \nonumber  \\
	&\quad+ \frac{1}{\varepsilon}\int_\Omega f(\rho) - \rho(\phi-\sfa) + f^\ast(\phi-\sfa) \,dx   \nonumber \\
	&\eqcolon \scF_{\sfc,\varepsilon}[\phi] + \frac{1}{\varepsilon}\int_\Omega f(\rho) - \rho(\phi-\sfa) + f^\ast(\phi-\sfa)\,dx \label{eq:GepsModicaMortola} \\
	&\ge \scF_{\sfc,\varepsilon}[\phi]. \label{eq:GFineq}
\end{align}

\begin{figure}
	\begin{subfigure}[t]{0.45\textwidth}
		\includegraphics[scale=0.9]{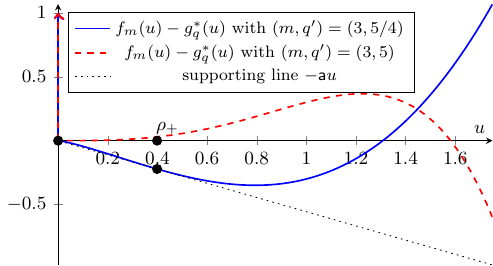}
		\subcaption{\small
			When \ref{hyp-g:growthCondition}, \ref{hyp-f:growthCondition}, \ref{hyp:compatibilityConditionF&G} are satisfied,  there exists a supporting line $u \mapsto -\sfa u$ intersecting the graph of the function $f - [{ \sfc}\,g]^\ast$ at $(0,0)$ and precisely one other point.}
	\label{subfig:subVsSupercritical}
\end{subfigure}\quad%
\begin{subfigure}[t]{0.45\textwidth}
	\includegraphics[scale=0.9]{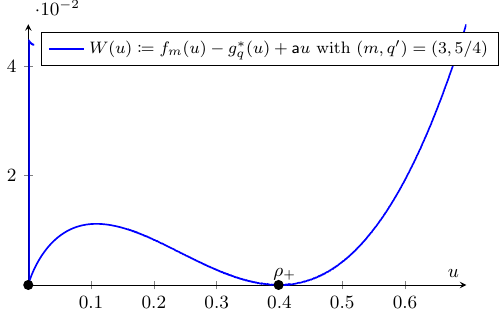}
	\subcaption{\small
		The double-well potential $W$ in $\scG_\varepsilon$ is formed by subtracting the supporting line $u \mapsto -\sfa u$ from $f - [\underline{\sfc}\,g]^\ast$. Implicitly, $W$ contributes forward-backward diffusion to the evolution of $\rho$.}
	\label{subfig:W}
\end{subfigure}
\captionsetup{width=\linewidth}
\caption{\small 
	Double-well potentials constructed from differences of powers $f_m$ and $g_q^\ast$; see \eqref{eq:powerLawNonlinearities}. 
} 
\label{fig:potentials}	
\end{figure}

We introduce the functions
\begin{gather}
{W_\ast} \colon \bbR \to \left[0,+\infty\right[, \qquad {W_\ast}(v) \coloneq \inf_{u \ge 0}\big\{W(u) + \big({\ubar \sfc} \,g^\ast(u/{{\ubar \sfc}}) -uv + {\ubar \sfc}\, g(v) \big)\big\} 
= {\ubar \sfc}\,g(v) - f^\ast(v-\sfa), \label{eq:Wstar} \\
W_{\ast\sfc} \colon \cl\Omega \times \bbR \to \left[0,+\infty\right[, \qquad W_{\ast\sfc}(x,v) \coloneq W_\ast(v) + (\sfc(x) - {\ubar\sfc})g(v), \label{eq:WstarC}
\end{gather}
which are essential for studying the limit $\varepsilon\to0^+$. The function $W_\ast$ is a nonnegative double-well potential with wells $\set{0,\phi_+}$ (see Lemma \ref{lemma:potentialWstar}), while $W_{\ast\sfc}$ is an inhomogeneous perturbation of ${W_\ast}$ (see Lemma \ref{lemma:potentialWstarC}). We view $W_{\ast\sfc}$ as a nonnegative potential with wells $\set{0,v_+(x)}$, one of which is space-dependent $v_+(x) = \phi_+\chi_{\set{\sfc = {\ubar \sfc}}}(x)$. 

Thus, the functionals $\scF_{\sfc,\varepsilon}, \scF_\varepsilon \colon H^1(\Omega) \to [0,+\infty]$ defined by
\begin{equation}\label{eq:ModicaMortolaF}
\begin{aligned}
	\scF_{\sfc,\varepsilon}[\phi] \coloneq \frac{1}{\varepsilon}\int_\Omega W_{\ast\sfc}(x,\phi) \,dx + \frac{\varepsilon}{2} \int_\Omega \abs{\nabla\phi}_{\sfA(x)}^2 \,dx 
	\ge \frac{1}{\varepsilon}\int_\Omega W_\ast(\phi) \,dx + \frac{\varepsilon}{2}\int_\Omega \abs{\nabla\phi}_{\sfA(x)}^2 \,dx 
	\eqcolon \scF_\varepsilon[\phi]
\end{aligned}
\end{equation}
are inhomogeneous, anisotropic Modica-Mortola functionals. Equality \eqref{eq:GepsModicaMortola} shows when \ref{hyp-g:growthCondition}, \ref{hyp-f:growthCondition}, \ref{hyp:compatibilityConditionF&G} are satisfied that  $\scG_\varepsilon$, fundamental to PKS systems, is simply $\scF_{\sfc,\varepsilon}$ with an additional penalty. 

The identity \eqref{eq:GepsModicaMortola} presents two viewpoints on phase separation for PKS. From the viewpoint of the definition \eqref{eq:augmentedEnergy} of $\scG_\varepsilon$, phase separation of the density is the result of a double-well potential $W$ for $\rho$, and this induces the phase separation of the chemoattractant through the penalty $[{\ubar \sfc}\,g]^\ast(\rho) - \rho\phi + {\ubar \sfc}\,g(\phi)$ coupling $(\rho,\phi)$. On the other hand, the identity \eqref{eq:GepsModicaMortola} shows directly that $\scG_\varepsilon$ is a Modica-Mortola functional with a double-well potential $W_\ast$ for $\phi$ and a penalty $f(\rho) - \rho(\phi-\sfa) + f^\ast(\phi-\sfa)$ coupling $(\rho,\phi)$.  
The emergence of these dual double-well potentials, $W$ for $\rho$ and $W_\ast$ for $\phi$, is a consequence of the compatibility of $f$ and $g$ expressed by \ref{hyp-g:growthCondition}, \ref{hyp-f:growthCondition}, \ref{hyp:compatibilityConditionF&G}. 

As a final passing remark, this rewriting of $\scG_\varepsilon$ as \eqref{eq:GepsModicaMortola} is independent of the uniqueness of the minimizer $\rho_+$ from \ref{hyp:compatibilityConditionF&G}. If we do not impose uniqueness, then the dual potential $W_\ast$ would have more than two wells, and $\scF_{\sfc,\varepsilon}$ would be a multi-well Modica-Mortola functional. 

\medskip
The term $\varepsilon^{-1}\int_\Omega \big(\sfc(x) - {\ubar \sfc}\big)g(\phi)\,dx$ is a soft penalty on the support of $\phi$; it is responsible for the limits being concentrated on the level set $\set{\sfc = {\ubar \sfc}}$. For phase separation to occur, there must be sufficient volume in this level set to support a probability density with range $\set{0,\rho_+}$. If $\sfc$ is constant, then there is no confinement effect, and there is sufficient volume whenever $\Omega$ has sufficient volume. We assume:
\begin{conditionsOnC}[start=\value{singularLimitHypotheses-c}+1]
\item \label{hyp-c:measure}  The region $\Omega_0 \coloneq \set{\sfc = {\ubar \sfc}}$ is a regular closed set, i.e.,  $\Omega_0 = \cl{{\oc\Omega}_0}$, $\rho_+\abs{{\oc\Omega}_0} > 1$, and $\partial\Omega_0 \in C^{0,1}$.
\setcounter{singularLimitHypotheses-c}{\value{conditionsOnCi}}	
\end{conditionsOnC}
The assumption that $\Omega_0$ is a regular closed set implies each of its connected components has nonempty interior, so the limits are not a priori concentrated on lower-dimensional regions.

\begin{theorem}[$\Gamma$-convergence of $\scG_\varepsilon$ \eqref{eq:augmentedEnergy}] \label{thm:gammaConvergence}
Let $\Omega \subset \bbR^d$ ($d\ge 2$) be a bounded domain with Lipschitz boundary, make the standard assumptions \ref{hyp-standardAssumptions} and suppose  \ref{hyp:compatibilityConditionF&G}, \ref{hyp-c:measure}, and  \ref{hyp-g:growthNear0-gamma}. As $\varepsilon \to 0^+$, the functionals $\seq{\scG_\varepsilon}$ $\Gamma$-converge w.r.t.\ strong $L^1(\Omega)^2$ to $\scG_0 \colon L^1(\Omega)^2 \to [0,+\infty]$ defined by 
\begin{equation}\label{eq:G0}
	\scG_0[\rho,\phi] \coloneq 
	\left\{\begin{aligned}
		&\gamma\int_\Omega \abs{\sfA^{1/2}(x) \vec{\nu}(x)}d \vert\nabla \phi\vert(x) &&
		\begin{aligned}[t]
			&\text{if } \phi \in BV(\Omega; \set{0,\phi_+}) , \, \rho =  \frac{\rho_+}{\phi_+} \phi  \,\,\, \cL^d\text{-a.e.,} \\
			&\chi_{\cl\Omega \setminus \Omega_0}\phi= 0 \,\,\,\cL^d\text{-a.e., }   \text{and } \int_\Omega \rho \,dx = 1,
		\end{aligned}  \\
		&+\infty && \text{otherwise in } L^1(\Omega)^2,
	\end{aligned}
	\right.
\end{equation}  
where $-\vec{\nu}$ is the Radon-Nikodym derivative of $\nabla\phi$ w.r.t.\ its total variation $\abs{\nabla\phi}$, $\rho_+$ is defined by \ref{hyp:compatibilityConditionF&G}, $\phi_+\coloneq {[{\ubar \sfc}\,g]^\ast}'(\rho_+)$, and the constant $\gamma>0$ is
\begin{equation}\label{eq:surfaceTensionCoeff}
	\gamma \coloneq \frac{1}{\phi_+}\int_0^{\phi_+} \sqrt{2{W_\ast}(v)} \,dv.
\end{equation}
More precisely,
\begin{enumerate}[label = (\roman*)]
	\item ($\Gamma$-$\liminf$ inequality) \label{thm:liminfIneq} For any positive sequence $\seq{\varepsilon_n}$ converging to $0$, pair $(\rho,\phi) \in  L^1(\Omega)^2$, and sequence of pairs $\seq{(\rho_n, \phi_n)} \subset L^1(\Omega)^2$ converging w.r.t.\ $L^1(\Omega)^2$ to $(\rho,\phi)$, we have 
	\begin{equation}\label{eq:gammaLiminfG}
		\scG_0[\rho,\phi] \le \liminf_{n\to\infty} \scG_{\varepsilon_n}[\rho_n, \phi_n].
	\end{equation}
	
	\item ($\Gamma$-$\limsup$ inequality) \label{thm:limsupIneq} If $g$ satisfies \ref{hyp-g:growthNear0-gamma}, then for any positive sequence $\seq{\varepsilon_n}$ converging to $0$ and pair $(\rho,\phi) \in  L^1(\Omega)^2$, there exists a sequence of pairs $\seq{(\rho_n, \phi_n)} \subset L^1(\Omega)^2$ converging w.r.t.\ $L^1(\Omega)^2$ to $(\rho,\phi)$ such that
	\begin{equation}\label{eq:gammaLimsupG}
		\scG_0[\rho,\phi] \ge \limsup_{n\to\infty} \scG_{\varepsilon_n}[\rho_n, \phi_n].
	\end{equation}		
\end{enumerate}
\end{theorem}

Since $\phi \in BV(\Omega; \set{0, \phi_+})$, we have $\phi = \phi_+\chi_E$ for some set $E$ of finite perimeter in $\Omega$. The object $\nabla\phi$ is an $\bbR^d$-valued Radon measure; it is absolutely continuous w.r.t.\ its total variation  $\abs{\nabla\phi}$ with corresponding Radon-Nikodym derivative $-\vec{\nu}$ equal to the measure-theoretic inner unit normal vector to the reduced boundary $\partial^\ast E$. 
The fact that $\scG_0$ is lower semicontinuous for strong $L^1(\Omega)$-convergence can be seen as a consequence of the Reshetnyak lower semicontinuity theorem  \cite[Theorem 2.38]{Ambrosio.Fusco.ea_2000_FunctionsBoundedVariation}.

For $(\rho,\phi) \in \Dom(\scG_0)$, we have the complementarity condition $\chi_{\cl\Omega\setminus\Omega_0} \phi = 0$ $\cL^d$-a.e.\ in $\Omega$, so up to an $\cL^d$-null set $E \subset \Omega_0$. We see $\scG_0$ accounts for the weighted perimeter of $E$ on $(\partial\Omega_0) \cap \Omega$, but it ignores the part of $\partial^\ast E$ on $\partial\Omega$, as in the classical perimeter functional. Indeed, for $(\rho,\phi) \in \Dom(\scG_0)$, 
\begin{equation} \label{eq:G0-differentBCs}
\scG_0[\rho,\phi] = \gamma\int_{\oc{\Omega}_0} \abs{\vec{\nu}(x)}_{\sfA(x)} \,d\vert\nabla\phi\vert(x) + \gamma\int_{(\partial\Omega_0) \cap \Omega} \abs{\vec{n}_0(x)}_{\sfA(x)} \,d\vert\nabla\phi\vert(x),
\end{equation}
where $\vec{n}_0$ is the outer unit normal vector to  $\partial\Omega_0$, which coincides with $\vec{\nu}$ on $(\partial^\ast E) \cap \partial\Omega_0$. The above formula highlights the differing treatment of $\partial\Omega$ and $(\partial\Omega_0)\cap\Omega$ in $\scG_0$. The second term behaves as a wetting energy in the classical theory of capillarity, with $x \mapsto \gamma \abs{\vec{n}_0(x)}_{\sfA(x)}$ playing the role of a relative adhesion coefficient between the free surface of the chemoattractant distribution and the boundary $(\partial\Omega_0)\cap\Omega$. Formula \eqref{eq:G0-differentBCs} is one way to see why the contact angles on $\partial\Omega$ and $(\partial\Omega_0) \cap \Omega$ differ in \eqref{eq:HS-STKU}. 

\medskip
We make some comments about the proof of Theorem \ref{thm:gammaConvergence}. After establishing a compactness theorem for $\scG_\varepsilon$, the proof of the $\Gamma$-$\liminf$ inequality is straightforward thanks to the lower bound $\scG_\varepsilon[\rho,\phi] \ge \scF_\varepsilon[\phi]$ (recall \eqref{eq:GFineq}, \eqref{eq:ModicaMortolaF}), but additional effort is needed for the $\Gamma$-$\limsup$ inequality. We first construct a recovery sequence for $\seq{\scF_{\sfc,\varepsilon}}$, which also satisfies $\scF_{\sfc,\varepsilon}[\phi] \le \scG_\varepsilon[\rho,\phi]$.  As in the proofs of  \cite[Theorem 2.1-(ii)]{Mellet_2024_HeleShawFlowSingular} and \cite[Theorem 2.3-(ii)]{Mellet.Rozowski_2024_VolumepreservingMeancurvatureFlow}, we can adapt a recovery sequence for $\seq{\scF_{\sfc,\varepsilon}}$ to construct one for $\seq{\scG_\varepsilon}$ since the case of equality  $\scF_{\sfc,\varepsilon}[\phi] = \scG_\varepsilon[\rho,\phi]$ is understood. Indeed, the identity \eqref{eq:GepsModicaMortola} for $\scG_\varepsilon$ shows, given $\phi$, equality is characterized by $\rho$ for which the penalty $\varepsilon^{-1}\int_\Omega f(\rho) - \rho(\phi-\sfa) + f^\ast(\phi-\sfa)\,dx$ vanishes. To obtain equality in the Fenchel-Young inequality, $\rho$ must satisfy 
\begin{equation} \label{eq:defRhoGammaConvergence}
\text{a.e.\ } \, x \in \Omega, 	\qquad \rho(x) \in \partial f^\ast(\phi(x)-\sfa).
\end{equation}
Thus, given a recovery sequence $\seq{\phi_n}$ for  $\scF_{\sfc,\varepsilon_n}$, we define $\rho_n$ pointwise to satisfy \eqref{eq:defRhoGammaConvergence} with $\phi_n$ in place of $\phi$. Some care must be taken to ensure that a function $\rho_n$ satisfying \eqref{eq:defRhoGammaConvergence} can be constructed that satisfies the mass constraint $\int_\Omega \rho_n \,dx = 1$. We also note, as in \cite{Mellet.Rozowski_2024_VolumepreservingMeancurvatureFlow}, there is no mass constraint on the $\phi$ variable for either $\scG_\varepsilon$ or $\scF_{\sfc,\varepsilon}$; the constraint $\int_\Omega \phi \,dx = \phi_+/\rho_+$ only arises in the limit $\varepsilon \to 0^+$. That is, the mass constraint on $\rho$ is retained and transferred to $\phi$. The emergence of this mass constraint is a consequence of the coupling between $\rho,\phi$ through the penalties $\varepsilon^{-1}\int_\Omega [{\ubar \sfc}\,g]^\ast(\rho) - \rho\phi + g(\phi)\,dx$ or $\varepsilon^{-1}\int_\Omega f(\rho) - \rho(\phi-\sfa) + f^\ast(\phi-\sfa)\,dx$ and it is why the mean-curvature-driven flows are volume-preserving.

Several technical differences arise in the construction of a recovery sequence for the Modica-Mortola functionals $\seq{\scF_{\sfc,\varepsilon}}$ in comparison to classical results. First, if the wells $\set{0,v_+(x)}$ of $W_{\ast\sfc}(x,\cdot)$ varied through $\Omega$ in a Lipschitz fashion and $W_\sfc(x,\cdot)$ satisfied the appropriate growth conditions near its wells, then as $\varepsilon \to 0^+$ the $\Gamma$-convergence w.r.t.\ strong $L^1(\Omega)$ of $\seq{\scF_{\sfc,\varepsilon}}$ towards a weighted perimeter functional would be a special case of the results of \citeauthor{Bouchitte_1990_SingularPerturbationsVariational} \cite[Theorem 3.5]{Bouchitte_1990_SingularPerturbationsVariational} or of \citeauthor{Owen.Sternberg_1991_NonconvexVariationalProblems} \cite[Theorem]{Owen.Sternberg_1991_NonconvexVariationalProblems}.  In our case, one of the wells is discontinuous, $v_+(x) = \phi_+\chi_{\Omega_0}(x)$. Despite this,   $\seq{\scF_{\sfc,\varepsilon}}$ still $\Gamma$-converges toward a weighted perimeter. This a consequence of the domain of $\scG_0$ only containing functions supported up to the discontinuity set $\partial\set{\sfc = {\ubar \sfc}}$ of the discontinuous well $v_+$, but not across it. 

The construction of the recovery sequence for $\seq{\scF_{\sfc,\varepsilon}}$ relies on a construction of an optimal profile $\omega$ connecting the locations of the wells of $W_{\ast\sfc}$ in $\Omega_0$. To ensure the contribution to the energy far from the interface is negligible, the optimal profile $z\mapsto \omega(z)$ should converge sufficiently fast as $\abs{z} \to +\infty$, i.e., away from the interface. In classical results (with the sign convention that $\omega$ is monotone decreasing), exponential convergence is obtained by requiring the double-well potential is no flatter than quadratic near its smaller well and no steeper than quadratic near its larger well, and the Cauchy problem for $\omega$ admits a unique solution whenever the potential is no flatter than quadratic near both wells. For the examples we have in mind, the potential $W_{\ast\sfc}$ is the difference between a power law and some translated and truncated power law (see Lemma \ref{lemma:potentialWstarC}-\ref{lemma:WstarC-formula}), which may be flatter than quadratic near its wells. These considerations impose some technical assumptions on the growth of $g$ near $0$. Namely,
\begin{conditionsOnG}[start = \value{singularLimitHypotheses-g}+1]
\item \label{hyp-g:growthNear0-gamma} Let $R_1, K_1$ be from \ref{hyp-g:growthCondition}. There exists $q_0 \ge 2$ and $1+\frac{q_0}{2} \le q_1 \le q_0$ such that for all $\abs{v} \le \min\set{1,1/R_1}$ it holds that $\abs{v}^{q_0}/K_1 \le g(v) \le K_1 \abs{v}^{q_1}$. \setcounter{wellPosednessHypotheses-g}{\value{conditionsOnGi}}
\end{conditionsOnG}
Near zero, if we quantify the flattness, $g \gtrsim \vert\cdot\vert^{q_0}$, then we can take $g$ to be flatter than quadratic, $g \lesssim \vert \cdot \vert^{q_1}$,

\subsection{Weak solutions of \texorpdfstring{\eqref{eq:PP-PKSeps}}{PP-PKS\_\textepsilon}} \label{subsec:weakSolutions}

We make precise the notion of weak solution to the system \eqref{eq:PP-PKSeps} that we use. Make the standard assumptions \ref{hyp-standardAssumptions}. A \emph{weak solution} of \eqref{eq:PP-PKSeps} is a triple  $(\rho_\varepsilon,\vec{v}_\varepsilon,\phi_\varepsilon)$ for which
\begin{gather*}
\rho_\varepsilon \in  L^\infty(0,T; L^m(\Omega)), \qquad \vec{v}_\varepsilon \in L^2(\left]0,T\right[\times\Omega, \rho_\varepsilon dx\,dt; \bbR^d),  \\
\phi_\varepsilon \in H^1(0,T; L^2(\Omega)) \cap L^2(0,T; H^2(\Omega)) \cap L^\infty(0,T; L^q(\Omega)). 
\end{gather*}
\begin{enumerate}
\item (continuity equation) The pair $(\rho_\varepsilon, \vec{v}_\varepsilon)$ satisfies a continuity equation $\partial_t\rho_\varepsilon + \pdiv{\rho_\varepsilon \vec{v}_\varepsilon} = 0$ on $\left]0,T\right[\times\Omega$ with no-flux boundary conditions $\rho_\varepsilon \vec{v}_\varepsilon \cdot \vec{n} = 0$ on $\left]0,T\right[ \times \partial\Omega$ in the sense of distributions:
\[
\int_\Omega \rho_\varepsilon^\init(x)\zeta(0,x)\,dx + \int_0^T \int_\Omega \big(\partial_t\zeta(t,x) + \vec{v}_\varepsilon(t,x)\cdot\nabla\zeta(t,x) \big)\rho_\varepsilon(t,x)dx \,dt = 0 
\]
for all test functions $\zeta \in C_c^1(\left[0,T\right[ \times \cl{\Omega})$. The mass flux $\vec{j}_\varepsilon \coloneq \rho_\varepsilon \vec{v}_\varepsilon$ satisfies 
\begin{equation*}
	\int_0^T \int_\Omega \vec{j}_\varepsilon \cdot \xi \,dx\,dt = \frac{1}{\varepsilon\alpha_0}\int_0^T \int_\Omega  \big[\rho_\varepsilon f'(\rho_\varepsilon) - f(\rho_\varepsilon) \big] \div\xi + \rho_\varepsilon\nabla\phi_\varepsilon \cdot \xi \,dx \,dt 
\end{equation*}
for all test vector fields $\xi \in C^1_c(\left]0,T\right[ \times \cl\Omega; \bbR^d)$ such that $\xi\cdot \vec{n} = 0$ on $[0,T] \times \partial\Omega$.

\item (parabolic equation) The equation $\displaystyle \varepsilon\partial_t\phi_\varepsilon - \varepsilon\pdiv{\sfA\nabla\phi_\varepsilon} = \varepsilon^{-1} (\rho_\varepsilon - \sfc g'(\phi_\varepsilon))$ is satisfied a.e.\ in $\left]0,T\right[\times\Omega$; the boundary condition  $\sfA\nabla\phi_\varepsilon \cdot \vec{n} = 0$ on $\left]0,T\right[\times\partial\Omega$ is satisfied in the sense of normal trace; and the initial condition $\phi_\varepsilon(0,\cdot) = \phi_\varepsilon^\init$ is satisfied a.e.\ in $\Omega$.

\item (energy dissipation inequality) The triple $(\rho_\varepsilon, \vec{v}_\varepsilon, \phi_\varepsilon)$ satisfies, for each $t \in [0,T]$, 
\begin{equation}\label{eq:GDissipationIneq}
	\begin{gathered}
		\scG_{\varepsilon}[\rho_{\varepsilon}(t), \phi_{\varepsilon}(t)] + \int_0^t \scD_\varepsilon(s) \,ds \le \scG_\varepsilon[\rho^\init_\varepsilon, \phi^\init_\varepsilon], \quad \text{where } \scD_\varepsilon(t) \coloneq \int_\Omega \rho_\varepsilon(t) \abs{\vec{v}_\varepsilon(t)}^2 + \varepsilon \big(\partial_t\phi_\varepsilon(t)\big)^2 \,dx.
	\end{gathered}
\end{equation}
We recall the functional $\scG_\varepsilon$, defined by \eqref{eq:augmentedEnergy}, is the ``centered'' version of the functional $\scE_\varepsilon$ for which \eqref{eq:PP-PKSeps} is more obviously a gradient flow on $\bbW_2(\cl\Omega) \times L^2(\Omega)$. 
\end{enumerate}

\subsection{Singular limit of \texorpdfstring{\eqref{eq:PP-PKSeps}}{PP-PKS\_\textepsilon} and derivation of \texorpdfstring{\eqref{eq:HS-STKU}}{HS-STKU}, Hele-Shaw flow with surface tension and kinetic undercooling} \label{subsec:singularLimit}
We make assumptions that are standard when deriving $BV$ solutions of perimeter-dissipated free boundary problems. We name them to make their use more apparent.
\begin{wellPreparedG}
\item \label{hyp:well-preparedG} Given a positive sequence $\seq{\varepsilon_n}$ converging to zero, we say a sequence $\seq{(\mu_{\varepsilon_n},\varphi_{\varepsilon_n})}$ of pairs of functions on $\cl\Omega$ is \emph{well-prepared} provided \(\displaystyle\ol{\scG} \coloneq \sup_{n\in\bbN} \scG_{\varepsilon_n}[\mu_{\varepsilon_n}, \varphi_{\varepsilon_n}] < +\infty\).
\end{wellPreparedG}

\begin{energyConvergenceG} 
\item  \label{hyp:energyConvergenceG} Given a positive sequence $\seq{\varepsilon_n}$ converging to zero, we say a pair of functions $(\mu_0,\varphi_0)$ on $[0,T] \times \cl\Omega$ and a sequence $\seq{(\mu_{\varepsilon_n},\varphi_{\varepsilon_n})}$ of pairs of functions on $[0,T]\times\cl\Omega$ satisfy the \emph{energy convergence condition} provided 
\[
\lim_{n\to\infty} \int_0^T \scG_{\varepsilon_n}[\mu_{\varepsilon_n}(t), \varphi_{\varepsilon_n}(t)]\,dt = \int_0^T \scG_0[\mu_0(t),\varphi_0(t)]\,dt.
\]
\end{energyConvergenceG}

\medskip
If $\sfc$ is not constant, then a confinement effect emerges in the limit $\varepsilon\to0^+$, and the limiting free boundary problem is posed on $\Omega_0 \coloneq \set{\sfc = {\ubar \sfc}}$, and we require two further assumptions on $\sfc$ and one on $g$. The first is a nondegeneracy condition on the growth rates of $\sfc$ and $\nabla\sfc$ near $\partial\Omega_0$ in $\Omega$:
\begin{conditionsOnC}[start = \value{singularLimitHypotheses-c}+1]
\item \label{hyp-c:grad-nondegeneracy} Suppose $\partial\Omega_0 \in C^{3,s}$ for some $0<s<1$ and let $\vec{n}_0$ be the outward unit normal vector to $\partial\Omega_0$. There exists a neighborhood $U$ of $\partial\Omega_0$ and a constant $K_4>0$ such that for every $\xi \in C^1(\cl\Omega; \bbR^d)$ such that $\xi \cdot \vec{n}_0 = 0$ on $\partial\Omega_0$ it holds that $\abs{\nabla \sfc(x) \cdot \xi(x)} \le K_4(\sfc(x) - {\ubar \sfc})\norm{\xi}_{C^1(\cl U)^d}$ for all $x \in U$. Further, there exists a constant $K_5>0$ and neighborhood $U'$ of $\partial\Omega_0$ such that $\sfc(x) - {\ubar \sfc} \ge K_5 \dist(x,\partial\Omega_0)^2$ for all $x\in U'$. \setcounter{singularLimitHypotheses-c}{\value{conditionsOnCi}}
\end{conditionsOnC}
The gradient estimate is used in the convergence of \eqref{eq:PP-PKSeps} to \eqref{eq:HS-STKU} to tame a term contributed to the first variation of the energy by the support penalty $\varepsilon^{-1}\int_\Omega (\sfc - {\ubar \sfc})g(\phi_\varepsilon)\,dx$. The regularity of the boundary and the lower growth estimate are crucial in obtaining a uniform estimate on the sequence of approximate pressures $\seq{p_\varepsilon}$ in some dual space (see Section \ref{subsubsec:derivationOfPressure-inhomogeneousDestruction}), and thus in deriving our weak formulation of \eqref{eq:HS-STKU}. In Appendix \ref{appendix:nondegeneracy}, we give a simple sufficient condition for the growth estimates in  \ref{hyp-c:grad-nondegeneracy}: there exists a constant $\lambda>0$ such that for all $x$ in some neighborhood of $\partial\Omega_0$ we have $\lap\sfc (x) \ge \lambda$. 

Secondly, we require that
\begin{conditionsOnC}[start = \value{singularLimitHypotheses-c}+1]
\item \label{hyp-c:interiorBoundary} $(\partial \Omega_0) \cap \partial\Omega = \emptyset$.
\end{conditionsOnC}
This assumption is not solely technical. If $(\partial\Omega_0) \cap \partial\Omega \ne\emptyset$, contact angle hysteresis or pinning may arise at the ``corner'' $C = \partial_{\partial\Omega}\big[ \big(\partial\Omega_0\big) \cap \big(\partial\Omega\big) \big]$. Consider the case where $\sfA \equiv I_d$ in \eqref{eq:HS-STKU}, so on the one hand $\partial E(t)$ meets $\partial\Omega$ orthogonally but on the other it meets $(\partial\Omega_0) \cap \Omega$ tangentially. If a triple junction advances along $\partial\Omega$ toward $C$, then, depending on the angle that $\cl{(\partial\Omega_0)\cap\Omega}$ meets $\partial\Omega$, the triple junction may be pinned in $C$ for a period of time until the contact angle adjusts from $\pi/2$ to $0$. In the case that $\cl{(\partial\Omega_0) \cap \Omega}$ intersects $\partial\Omega$ orthogonally, so if a triple junction enters $C$ then the contact angle adjusts instantaneously. We thus only consider cases with no corners, one case of which is described by \ref{hyp-c:interiorBoundary}.

\begin{SCfigure}
\includegraphics[scale=0.8]{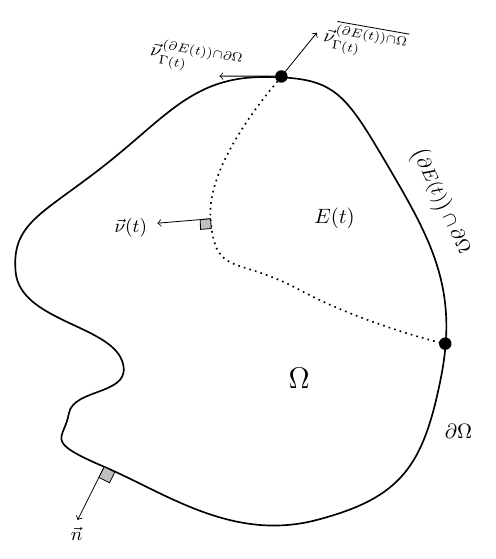}
\caption{\small 
	The limiting distributions of organisms and chemoattractant are $\rho(t) = \rho_+ \chi_{E(t)}$ and $\phi(t) = \frac{\phi_+}{\rho_+}\rho(t)$. The vector $\vec{\nu}(t)$ is the outer unit normal vector to the free boundary $\partial E(t)$ while $\vec{n}$ is the outer unit normal to the fixed boundary $\partial\Omega$. The vectors $\vec{\nu}^{\cl{(\partial E(t)) \cap \Omega}}_{\Gamma(t)}$ and $\vec{\nu}^{(\partial E(t))\cap\partial\Omega}_{\Gamma(t)}$ are unit co-normal vectors. The former is tangent to $\partial E(t)$ while  normal to $\partial_{\partial\Omega}\big[(\partial E(t)) \cap \partial\Omega \big]$; the latter is tangent to $\partial\Omega$ while normal to $\partial_{\partial\Omega}\big[(\partial E(t)) \cap \partial\Omega \big]$. Contact angle conditions are enforced at the dots. If $\sfc$ is not constant, then we view the above as some connected component of $\Omega_0$, with $\Omega_0$, $\vec{n}_0$, and $\vec{\nu}^{(\partial E(t)) \cap \partial\Omega_0}_{\Gamma(t)}$ in place of $\Omega$, $\vec{n}$, and $\vec{\nu}^{(\partial E(t))\cap\partial\Omega}_{\Gamma(t)}$. \label{fig:contactAngle}}	
	\end{SCfigure}

	Only in Section \ref{subsubsec:derivationOfPressure-inhomogeneousDestruction} will \ref{hyp-c:interiorBoundary} arise. There, we will consider  separately the two cases of ``no corners'': first, $\sfc$ can be a positive constant function, i.e., $\sfc \equiv {\ubar \sfc} > 0$, or second, $\sfc$ can satisfy \ref{hyp-c:grad-nondegeneracy}, \ref{hyp-c:interiorBoundary}. In the former case, only the orthogonality between the Cahn-Hoffman vector $\sfA\vec{\nu}/\abs{\vec{\nu}}_\sfA$ and $\vec{n}$ (that is, the fifth equation of \eqref{eq:HS-STKU}) emerges in the limit since $\Omega_0 = \cl\Omega$. While in the latter case, only the tangential condition in \eqref{eq:HS-STKU} emerges due to the limits $(\rho,\phi)$ being concentrated on $\Omega_0$ and \ref{hyp-c:interiorBoundary}.
	
	We also require a more precise description of the growth of $g$ near $0$:
	\begin{conditionsOnG}[start = \value{singularLimitHypotheses-g}+2]
\item \label{hyp-g:growthNear0} Let $m$ be from \ref{hyp-f:growthCondition}. For all $\abs{v} \le 1/R_1$, it holds that $g(v) \le K_1 \abs{v}^{m/(m-1)}$. 
\end{conditionsOnG}
This hypothesis is used in Section \ref{subsubsec:derivationOfPressure-inhomogeneousDestruction} to obtain a weighted density estimate as part of the derivation of \eqref{eq:weakFormFBCs}. We note that $q$ from \ref{hyp-g:growthCondition} is compatible with \ref{hyp-g:growthNear0} since $m>q'$, so $q > m' \coloneq m/(m-1)$.

\begin{theorem} \label{thm:main}
Make the standard assumptions \ref{hyp-standardAssumptions} and suppose $f$ and $g$ are compatible for phase separation  \ref{hyp:compatibilityConditionF&G}. Given a positive sequence $\seq{\varepsilon_n}$ converging to $0$, let $\seq{(\rho_{\varepsilon_n}, \vec{v}_{\varepsilon_n},  \phi_{\varepsilon_n})}$ be weak solutions to \eqref{eq:PP-PKSeps} with initial data $\seq{(\rho^\init_{\varepsilon_n}, \phi^\init_{\varepsilon_n})}$ that is well-prepared, i.e., satisfies \ref{hyp:well-preparedG}. Let $\vec{j}_{\varepsilon_n} \coloneq \rho_{\varepsilon_n}\vec{v}_{\varepsilon_n}$.
\begin{enumerate}[label=(\roman*)]
	\item (phase separation) \label{thm:phaseSeparation} If $\sfc$ satisfies \ref{hyp-c:measure}, then  
	there exists a (not relabeled) subsequence of $\seq{\varepsilon_n}$ as well as functions
	\begin{equation}
		\begin{gathered}
			\rho \in L^\infty(0,T; BV(\Omega; \set{0, \rho_+})) \cap  BV(\left]0,T\right[ \times \Omega) \cap C^{\frac{1}{2}}([0,T]; L^1(\Omega))	\\
			\text{satisfying, for each } t \ge 0, \text{ } \rho(t)\chi_{\cl\Omega \setminus \Omega_0} = 0 \text{ a.e.\ in } \Omega, \\
			\text{and } \vec{j} \in L^2(0,T; L^{\frac{2m}{m+1}}(\Omega; \bbR^d))
		\end{gathered}
	\end{equation}
	such that $\lim_{n\to\infty} \rho_{\varepsilon_n} = \rho$ in $L^\infty(0,T; L^1(\Omega))$ and $\vec{j}_{\varepsilon_n} \rightharpoonup \vec{j}$ weakly in $L^2(0,T; L^{\frac{2m}{m+1}}(\Omega;\bbR^d))$. Along this subsequence, $\lim_{n\to\infty} \phi_{\varepsilon_n} = \frac{\phi_+}{\rho_+}\rho$ in $C([0,T]; L^1(\Omega))$, where $\phi_+ \coloneq {[{\ubar \sfc}\,g]^\ast}'(\rho_+)$.

	\item (continuity equation) \label{thm:continuityEq} There exists a velocity field $\vec{v} \in L^2(\left]0,T\right[ \times \Omega, d\rho dt; \bbR^d)$ such that $\vec{j} = \rho \vec{v}$ $dx\,dt$-a.e.\ in $[0,T]\times\Omega$ and $(\rho,\vec{v})$ satisfy the continuity equation with no-flux boundary conditions in the distributional sense: 
	\begin{equation} \label{eq:limitContinuityEquation}
		\forall \, \zeta \in C_c^1(\left[0,T\right[ \times \cl\Omega) \qquad \int_\Omega \rho(0,\cdot)\zeta(0,\cdot)\,dx + \int_0^T \int_\Omega \rho \partial_t \zeta  \,dx \,dt + \int_0^T \int_\Omega \rho \vec{v} \cdot \nabla\zeta \,dx\,dt = 0.
	\end{equation}

	\item (conditional convergence of \eqref{eq:PP-PKSeps} to \eqref{eq:HS-STKU}) \label{thm:conditionalConvergence} Assume moreover that $(\rho,\phi)$ from the phase separation result and the subsequence that converges to it $\seq{(\rho_{\varepsilon_n}, \phi_{\varepsilon_n})}$ satisfy the energy convergence condition \ref{hyp:energyConvergenceG}.

	Then along a further subsequence, there exists $V \in L^2(\left]0,T\right[\times\Omega, \abs{\nabla\rho}dt)$ such that 
	\begin{equation}\label{eq:weakNormalComponent}
		\int_0^T\int_\Omega \rho \vec{v} \cdot \nabla\zeta \,dx\,dt= \int_0^T \int_\Omega  V\zeta\,d\vert\nabla\rho(t)\vert(x)\,dt \qquad \forall \zeta \in C^1_c(\left[0,T\right[ \times \cl\Omega).
	\end{equation}
	
	\medskip
	Consider the \emph{approximate pressures} $\seq{p_{\varepsilon_n}}$ defined by
	\begin{equation} \label{eq:approximatePressure}
		p_{\varepsilon_n}(t,x) \coloneq \rho_{\varepsilon_n}(t,x) \frac{ f'(\rho_{\varepsilon_n}(t,x)) + \sfa - \phi_{\varepsilon_n}(t,x)} {\varepsilon_n} - m_{\varepsilon_n}(t,x).
	\end{equation}
	Define for $\zeta \in L^1(\Omega)$ its mean over a measurable set $U \subset \Omega$:  $(\zeta)_U \coloneq \abs{U}^{-1}\int_U \zeta(x)\,dx$. Let $\seq{\Omega_0^k}$ be the connected components of $\Omega_0$, and then for a.e.\ $t\in[0,T]$ the function $x\mapsto m_{\varepsilon_n}(t,x)$ is defined such that $(p(t))_{\Omega_0^k}=0$ for each $k$. The approximate pressures $\seq{p_{\varepsilon_n}}$ admit a subsequence converging weakly-$\ast$ in $L^2(0,T; \cX^\ast)$ to some $p \in L^2(0,T; \cX^\ast)$ (see \eqref{eq:pressureSpace}). 
	
	Then, for any test vector field $\xi \in \cV$ (see \eqref{eq:testVectorFields}), the pressure $p$ satisfies
	\begin{equation}\label{eq:weakFormFBCs}
		\begin{aligned}
			&\int_0^T\int_\Omega \rho \vec{v} \cdot \xi \,dx \,dt - \frac{1}{\alpha_0}\int_0^T\ip{p(t)}{\div\xi(t)}\,dt \\
			&= -\frac{\gamma_0}{\alpha_0}\int_0^T \int_\Omega V\vec{\nu}\cdot\xi \,d\vert\nabla\rho(t)\vert(x) \,dt \\
			&\quad -\frac{\gamma_0}{\alpha_0}\int_0^T \int_\Omega \Big[\div\xi - \frac{\sfA\vec{\nu}}{\abs{\vec{\nu}}_\sfA} \otimes \frac{\vec{\nu}}{\abs{\vec{\nu}}_\sfA} : (\nabla\xi)^\trans 
			+ \frac{1}{2} \frac{\vec{\nu}}{\abs{\vec{\nu}}_\sfA}\otimes \frac{\vec{\nu}}{\abs{\vec{\nu}}_\sfA} : \big( \xi^k\partial_k\sfA \big) \Big] \abs{\vec{\nu}}_\sfA\,d\vert\nabla\rho(t)\vert(x) \,dt, 
		\end{aligned} 
	\end{equation}
	where $\ip{\cdot}{\cdot}$ is the duality pairing of $\cX^\ast\times \cX$, $\gamma_0 \coloneq \gamma \phi_+/\rho_+$, and for a.e.\ $t\in[0,T]$ the vector field $-\vec{\nu}(t)$ is the Radon-Nikodym derivative of $\nabla\rho(t)$ w.r.t.\ its total variation $\abs{\nabla\rho(t)}$.
\end{enumerate}
\end{theorem}

The approximate pressures $\seq{p_\varepsilon}$ converge weakly-$\ast$ in $L^2(0,T; \cX^\ast)$ to some $p$. The definition of the space $\cX$ depends on whether $\sfc$ is a positive constant function or not. To define it, let $0<s<1$ be arbitrary if $\sfc$ is a positive constant function or let it be from  \ref{hyp-c:grad-nondegeneracy} if $\sfc$ satisfies \ref{hyp-c:grad-nondegeneracy}, \ref{hyp-c:interiorBoundary}.  Then,   
\begin{equation} \label{eq:pressureSpace}
\cX \coloneq 
\left\{\begin{aligned}
	&\cset{\zeta \in C^{0,s}(\cl\Omega)}{(\zeta)_\Omega = 0} \text{ if } \sfc \text{ is a positive constant function,}  \\
	&\begin{aligned}[t]
		&\cset{\zeta \in C^{1,s}(\cl\Omega)}{ \text{for all } k \,\, \zeta\vert_{\partial\Omega_0^k} = 0 \text{ and }  (\zeta)_{\Omega_0^k} = 0}  \text{ if } \sfc \text{ satisfies \ref{hyp-c:grad-nondegeneracy}, \ref{hyp-c:interiorBoundary}, }  \\
		&\qquad g \text{ satisfies \ref{hyp-g:growthNear0}, and } f \text{ is defined by \eqref{eq:powerLawNonlinearities}.}
	\end{aligned}
\end{aligned}
\right.
\end{equation}
The space of test vector fields $\cV$ used in \eqref{eq:weakFormFBCs} as part of our weak formulation of \eqref{eq:HS-STKU} is (with the same $s$ as in the definition of $\cX$)
\begin{equation} \label{eq:testVectorFields}
\cV \coloneq 
\left\{\begin{aligned}
	&\cset{\eta \in L^2(0,T; C^{1,s}(\cl\Omega;\bbR^d))}{\eta \cdot \vec{n}\vert_{[0,T]\times\partial\Omega} = 0 }  \text{ if } \sfc \text{ is a positive constant function,} \\
	&\begin{aligned}[t]
		&\cset{\eta \in L^2(0,T; C^{2,s}(\cl\Omega;\bbR^d))}{\eta \cdot \vec{n}\vert_{[0,T]\times\partial\Omega} =    \div\eta\vert_{[0,T] \times \partial\Omega_0} = \eta \cdot \vec{n}_0\vert_{[0,T] \times \partial\Omega_0} = 0} \text{ if } \sfc  \\
		&\qquad\text{ satisfies \ref{hyp-c:grad-nondegeneracy}, \ref{hyp-c:interiorBoundary}, } g \text{ satisfies \ref{hyp-g:growthNear0}, and } f \text{ is defined by \eqref{eq:powerLawNonlinearities}}.
	\end{aligned}
\end{aligned}
\right.
\end{equation}

We emphasize in the case of $\sfc$ not a positive constant that the derivation of \eqref{eq:weakFormFBCs} is only carried out for the particular case of porous medium diffusion, namely $f(u) = u^m/(m-1) + \iota_{\left[0,+\infty\right[}(u)$ with $m>q'$ (recall \ref{hyp-f:growthCondition}, \ref{hyp-g:growthCondition}). While we anticipate that \eqref{eq:weakFormFBCs} arises for more general nonlinearities, for instance, $f$ that are comparable to porous medium diffusion for both small and large $u$, we do not pursue this analysis.

We stress that it is the triple \eqref{eq:limitContinuityEquation}, \eqref{eq:weakNormalComponent}, \eqref{eq:weakFormFBCs} that is our weak formulation of \eqref{eq:HS-STKU}. We briefly sketch why a smooth solution of this triple solves \eqref{eq:HS-STKU}. On the one hand, the continuity equation \eqref{eq:limitContinuityEquation} encodes the incompressibility of the velocity field $\vec{v}(t)$ on $E(t)$. On the other hand, \eqref{eq:weakFormFBCs} encodes Darcy's law $\rho_+\vec{v}(t) = -\frac{1}{\alpha_0}\nabla p(t)$ on $E(t)$, so the combination of these equations implies the pressure $p$ is harmonic in the region $E(t)$ containing the organisms (i.e., the first equation in \eqref{eq:HS-STKU}). The equality \eqref{eq:weakNormalComponent} shows $V(t)$ is equal, in a weak sense, to the normal component of the velocity field $\vec{v}(t)$ along the free boundary $\partial E(t)$, which, in view of the equality $\rho_+\vec{v} = -\frac{1}{\alpha_0}\nabla p$, establishes the velocity law of the free boundary (the third equation in \eqref{eq:HS-STKU}).

A detailed argument that recovers the fourth and fifth equations of \eqref{eq:HS-STKU} is carried out by \citeauthor{Kraus_2011_DegenerateNondegenerateStefan}   \cite[Theorem 3.1]{Kraus_2011_DegenerateNondegenerateStefan} for an analogous weak formulation of the Stefan problem with anisotropic and inhomogeneous surface tension (Gibbs-Thomson effect). The strategy is to test \eqref{eq:weakFormFBCs} with vector fields normal to the free boundary in $\Omega$ and supported away from $\partial\Omega$ (and hence away from triple junctions). Derivatives on $\xi$ are removed by the divergence theorem on hypersurfaces, which allows to derive the fourth equation of \eqref{eq:HS-STKU}. The fifth equation of \eqref{eq:HS-STKU} arises from testing \eqref{eq:weakFormFBCs} with vector fields normal to $\partial\Omega$. The boundary terms that arise from the earlier application of the divergence on hypersurfaces necessarily vanish, which implies the orthogonality condition. Deriving the tangential condition is a bit more involved but is otherwise similar.

\subsection{Relationship to prior work}
\citeauthor{Caginalp_1989_StefanHeleShawType}  \cite{Caginalp_1989_StefanHeleShawType} derived the Hele-Shaw free boundary problem with surface tension and kinetic undercooling \eqref{eq:HS-STKU} (in addition to Stefan problems with such effects) using matched asymptotic expansions from solutions $(u_\varepsilon,\phi_\varepsilon)$ of phase-field equations. In particular, 
\[
\partial_tu_\varepsilon + \frac{\ell}{2} \partial_t \phi_\varepsilon = K_\varepsilon \lap u_\varepsilon, \quad \alpha \varepsilon^2 \partial_t \phi_\varepsilon = \lap \phi_\varepsilon + \frac{1}{a}{W_\ast}(\phi_\varepsilon) + 2u_\varepsilon, \quad \ell_\varepsilon, K_\varepsilon \sim \frac{1}{\varepsilon^2}, \, a \sim \varepsilon^2 \text{ with } 0<\varepsilon \ll 1.
\]
where its solutions $(u_\varepsilon, \phi_\varepsilon)$ are, resp., the temperature of a substance and a diffuse phase field describing the state of phase. The function ${W_\ast}$ is a double-well potential with wells of equal depth located at the values of pure phase. The parameters $\ell_\varepsilon$, $K_\varepsilon$ are related to the latent heat of phase transition and thermal diffusivity of the material; $\alpha$ is related to a microscopic relaxation time; $\varepsilon$ is related to the thickness of the interface.

With regard to rigorous sharp interface limits,  \citeauthor{Soner_1995_ConvergencePhasefieldEquations}  \cite{Soner_1995_ConvergencePhasefieldEquations} showed, unconditionally, that smooth solutions of the phase field equations
\[
\partial_t u_\varepsilon + (1-\phi_\varepsilon^2) \partial_t\phi_\varepsilon = \lap u_\varepsilon,
\qquad \varepsilon \partial_t\phi_\varepsilon - \varepsilon\lap\phi_\varepsilon = - \frac{1}{\varepsilon}{W_\ast}'(\phi_\varepsilon) + (1-\phi^2_\varepsilon)u_\varepsilon,
\] 
converge as $\varepsilon \to0^+$ to an integral varifold solution of the two-phase Hele-Shaw problem with surface tension (so-called Mullins-Sekerka flow) with kinetic undercooling. Then,  \citeauthor{Caginalp.Chen_1998_ConvergencePhaseField}  \cite{Caginalp.Chen_1998_ConvergencePhaseField} showed uniform convergence of classical solutions of another set of phase field equations to a classical solution of \eqref{eq:HS-STKU}. 
These results are thus only valid until the formation of the first geometric singularity. 

Short time existence and uniqueness of classical solutions to Mullins-Sekerka with kinetic undercooling was proven by \citeauthor{Yu_1996_QuasisteadyStefanProblem}  \cite{Yu_1996_QuasisteadyStefanProblem}. For the corresponding Stefan problem,  \citeauthor{Chen.Reitich_1992_LocalExistenceUniqueness} \cite{Chen.Reitich_1992_LocalExistenceUniqueness} proved the local existence and uniqueness of solutions.  \citeauthor{Yi_1997_AsymptoticBehaviourSolutions} \cite{Yi_1997_AsymptoticBehaviourSolutions} used this local existence result to establish local existence of classical solutions for \eqref{eq:HS-STKU} (with a Dirichlet condition on the fixed boundary) as a limit w.r.t.\ vanishing specific heat $\varepsilon \to 0^+$ of solutions to the corresponding Stefan problem, i.e., where the first equation $\lap p = 0$ in \eqref{eq:HS-STKU} is replaced by $\varepsilon \partial_t p_\varepsilon - \lap p_\varepsilon = 0$.  \citeauthor{Gunther.Prokert_2005_HeleShawTypeDomainEvolution} studied \cite{Gunther.Prokert_2005_HeleShawTypeDomainEvolution} a variant of \eqref{eq:HS-STKU} where the surface tension coefficient varies over the interface $\partial E(t)$ (e.g., due to surfactants); they prove short time existence and uniqueness of solutions in Sobolev spaces, improvement of regularity, and continuous dependence on initial data.  Finally,  \citeauthor{Mucha_2007_LimitKineticTerm} in \cite{Mucha_2007_LimitKineticTerm} exhibited short time existence of unique solutions to \eqref{eq:HS-STKU} in anisotropic Besov spaces, studied the limit as the influence of the undercooling term vanishes ($\beta_0 \to 0^+$ in our model), and verified that the limit obtained satisfies \eqref{eq:HS-ST}.

\section{\texorpdfstring{$\Gamma$}{\textGamma}-convergence of the energy \texorpdfstring{$\scG_\varepsilon$}{G\_\textepsilon} (Proof of Theorem \texorpdfstring{\ref{thm:gammaConvergence}}{2.1})} \label{sec:gammaConvergence}

\subsection{The double-well potentials \texorpdfstring{$W$}{W}, \texorpdfstring{$W_\ast$}{W\_*}, and \texorpdfstring{$ W_{\ast\sfc}$}{W\_*c}.}

We first record some properties of the double-well potential $W$ for $\rho$ appearing in the definition \eqref{eq:augmentedEnergy} of $\scG_\varepsilon$. 
\begin{lemma}[the double-well potential $W$] \label{lemma:potentialW}
	Let $W$ be defined by \eqref{eq:potentialW}.
	\begin{enumerate}[label=(\roman*)]
		\item \label{lemma:Wnonneg} $W \in C(\left[0,+\infty\right[) \cap C^1(\left]0,+\infty\right[)$ is nonnegative.
		\item \label{lemma:Wzeros} There exists a constant $\rho_+ > 0$ such that $W(u) = 0$ iff $u \in \set{0,\rho_+}$. 
		\item \label{lemma:Wcoercive} There exist constants $c_1 > 0$ and $R_3 > 0$ such that $W(u) \ge c_1 u^m$ for any $u > R_3$.
	\end{enumerate}
\end{lemma}
\begin{proof}
	\ref{lemma:Wnonneg}, \ref{lemma:Wzeros}: The continuous differentiability follows from that of $f$ and $g^\ast$, the latter of which follows from the strict convexity and finiteness of $g$ on $\bbR$ \cite[Chapter E, Section 4.1]{Hiriart-Urruty.Lemarechal_2001_FundamentalsConvexAnalysis}. The nonnegativity is implied by the definition of the constant $-\sfa$, whose existence is guaranteed by the compatibility condition \ref{hyp:compatibilityConditionF&G}.  The positive zero $\rho_+$ is implied by the definition of $-\sfa$, while $W(0) = 0$ by $f(0)=0$ \ref{hyp-f:regularity} and $g^\ast(0) = 0$ since $g(v) \ge 0$ and $g(0) = 0$ \ref{hyp-g:regularity}. 
	
	\ref{lemma:Wcoercive}: This is implied by the growth conditions \ref{hyp-f:growthCondition}, \ref{hyp-g:growthCondition} and that $m>q'>1$.  Indeed, the growth condition \ref{hyp-g:growthCondition} implies a dual growth condition on the Legendre transform: there exist $K_1',R_1'>0$ such that ${\ubar\sfc}\,g^\ast(u/{\ubar\sfc})=[{\ubar\sfc}\,g]^\ast(u) \le K_1' u^{q'}$ for all $u > R_1'$. Then for small enough $c_1$ we have for all $u$ large enough that
	$
	W(u) \coloneq f(u) - [{\ubar\sfc}\, g]^\ast(u) + \sfa u \ge K_2 u^m - K_1' u^{q'} + \sfa u \ge c_1 u^m
	$
	by the inequality $m > q' > 1$.
\end{proof}

We now examine the dual double-well potential ${W_\ast}$ for $\phi$. This establishes that the functional $\scF_\varepsilon$ defined by \eqref{eq:ModicaMortolaF} is an anisotropic Modica-Mortola functional.
\begin{lemma}[the double-well potential ${W_\ast}$] \label{lemma:potentialWstar}
	Let ${W_\ast}$ be defined by \eqref{eq:Wstar}. 
	\begin{enumerate}[label = (\roman*)]
		\item \label{lemma:Wstar-nonneg} ${{W_\ast}} \in C^1(\bbR)$ is nonnegative.
		\item \label{lemma:Wstar-zeros} ${{W_\ast}}(v) = 0$ iff $v \in \set{0,\phi_+}$, where $\phi_+ \coloneq {[{\ubar\sfc }\,g]^\ast}'(\rho_+) = {g^\ast}'(\rho_+/{\ubar\sfc}) = f'(\rho_+) + \sfa>0$.
		\item \label{lemma:Wstar-formula}${{W_\ast}}(v) = {\ubar\sfc}\,g(v) - f^\ast(v-\sfa)$, where $\sfa$ is from \ref{hyp:compatibilityConditionF&G}.
		\item \label{lemma:Wstar-coercive} There exist constants  $c_2>0$ and $R_4 > 0$ such that ${W_\ast}(v) \ge c_2\abs{v}^q$ for all $\abs{v} > R_4$.
		\item \label{lemma:Wstar-ub} For all $v \in \bbR$, ${W_\ast}(v) \le \min\left\{{\ubar\sfc}\,g(v), \, [{\ubar\sfc}\,g]^\ast(\rho_+) - \rho_+v + {\ubar\sfc}\,g(v) \right\}$.
		\item \label{lemma:Wstar-near0} If $g$ satisfies \ref{hyp-g:growthNear0-gamma}, then for all $\abs{v}\le \min\set{1,a,1/R_1}$, we have ${\ubar\sfc}\,\abs{v}^{q_0}/K_1 \le {W_\ast}(v) \le {\ubar\sfc}\,K_1 \abs{v}^{q_1}$.
		\item \label{lemma:Wstar-nearPhip} For all $v \in [(1/2)\phi_+,\, (3/2)\phi_+]$ we have ${W_\ast}(v) \le {\ubar\sfc}\norm{g''}_{L^\infty([\frac{1}{2}\phi_+,\frac{3}{2}\phi_+])} \abs{v-\phi_+}^2$
	\end{enumerate}
\end{lemma}

\begin{proof}
	\ref{lemma:Wstar-zeros}:  Note ${W_\ast}(v) = 0$ iff $W(u) = 0$ and $[{\ubar\sfc}\, g]^\ast(u) - uv + {\ubar\sfc}\,g(v) = 0$. The zeros of $W$ are $\set{0,\rho_+}$ and equality in the Fenchel-Young inequality holds iff $v = {[{\ubar\sfc}\,g]^\ast}'(u)$ (recall $g^\ast$ is $C^1$ since $g$ is strictly convex and $C^1$). Since $g^\ast$ is nonnegative and $g^\ast(0)=0$, we conclude ${g^\ast}'(0) = 0$, so  ${W_\ast}(v)=0$ iff $v \in \set{{[{\ubar\sfc}\,g]^\ast}'(0),\, {[{\ubar\sfc}\,g]^\ast}'(\rho_+)} = \set{0,\,{g^\ast}'(\rho_+/{\ubar\sfc})}$. To derive the second formula for $\phi_+$, we use the first-order optimality condition for the minimization appearing in the compatibility condition \ref{hyp:compatibilityConditionF&G} and substitute $\phi_+$ for ${[{\ubar\sfc}\,g]^\ast}'(\rho_+)$. 
	
	\ref{lemma:Wstar-formula}: From the definitions of ${W_\ast}$  \eqref{eq:Wstar} and $W$ \eqref{eq:potentialW},
	\begin{gather*}
		{W_\ast}(v)
		= \inf_{u \ge 0} \left\{  f(u) + \sfa u - uv \right\} + {\ubar\sfc}\,g(v) = {\ubar\sfc}\, g(v) - \sup_{u \ge 0} \left\{(v-\sfa)u - f(u)\right\} = {\ubar\sfc}\,g(v) - f^\ast(v-\sfa),
	\end{gather*}
	where the last equality follows since $f(u) = +\infty$ for $u < 0$.
	
	\ref{lemma:Wstar-nonneg}: The nonnegativity is implied by the definition of $W_\ast$ \eqref{eq:Wstar}, nonnegativity of $W$ (Lemma \ref{lemma:potentialW}-\ref{lemma:Wnonneg}), and the Fenchel-Young inequality. In view of \ref{lemma:Wstar-formula}, the regularity is implied by that of $g$ and $f^\ast$.
	
	\ref{lemma:Wstar-coercive}: We use the expression from \ref{lemma:Wstar-formula}. The coercivity condition \ref{hyp-f:growthCondition} on $f$ implies a dual growth condition on $f^\ast$: there exists $K'_2,R_2'$ such that $f^\ast(v) \le K_2'v^{m'}$ for all $v > R_2'$, where $m'\coloneq m/(m-1)$. Since $f(u)=+\infty$ for $u<0$, this implies $f^\ast(v) =0$ for $v <0$. Also, \ref{hyp-f:growthCondition} implies $q > m' > 1$, so the growth condition on $g$ \ref{hyp-g:growthCondition} gives for a sufficiently small $c_2>0$ and all sufficiently large $\abs{v}$   that 
	$
	{W_\ast}(v) \ge {\ubar\sfc}K_1\abs{v}^q - K_2'\abs{v}^{m'} \ge c_2 \abs{v}^q.
	$
	
	\ref{lemma:Wstar-ub}: In the infimum \eqref{eq:Wstar} defining ${W_\ast}$, the zeros $\set{0,\rho_+}$ of $W$ are competitors, so evaluating this expression at each of them gives the claimed bound. 
	
	\ref{lemma:Wstar-near0}: We note that $f^\ast(v-\sfa)=0$ for $v\le \sfa$. The compatibility condition \ref{hyp:compatibilityConditionF&G} ensures $\sfa>0$, so for $v$ near $0$ we have ${W_\ast}(v) = {\ubar\sfc}\,g(v)$. Then  \ref{hyp-g:growthNear0-gamma} gives the estimates.
	
	\ref{lemma:Wstar-nearPhip}: We use \ref{lemma:Wstar-ub}. Let $R(u,v) \coloneq [{\ubar\sfc}\,g]^\ast(u) - uv + {\ubar\sfc}\,g(v)$, which is convex in each variable separately. We have $R(\rho_+,\phi_+) = 0$ so convexity gives $R(\rho_+,v) \le -\partial_v R(\rho_+,v)(\phi_+ - v)$. Thus, 
	\begin{align*}
		{W_\ast}(v) \le R(\rho_+,v) \le -\partial_v R(\rho_+,v)(\phi_+ - v) = {\ubar\sfc}\,(g'(\phi_+)- g'(v))(\phi_+ - v) 
		\le {\ubar\sfc}\, \norm{g''}_{L^\infty([\frac{1}{2}\phi_+,\frac{3}{2}\phi_+])}\abs{\phi_+-v}^2,
	\end{align*}
	where we have used $\rho_+ = {\ubar\sfc}\,g'(\phi_+)$ and $g'\in\Lip_\loc(\bbR)$, \ref{hyp-g:regularity}.
	
\end{proof}

Next is the perturbed double-well potential $W_{\ast\sfc}$ arising from spatial dependence of $\sfc$. The next lemma verifies the functional $\scF_{\sfc,\varepsilon}$ defined by \eqref{eq:ModicaMortolaF} that bounds $\scG_\varepsilon$ below \eqref{eq:GFineq} is an inhomogeneous, anisotropic Modica-Mortola functional.   
\begin{lemma}[the potential $W_{\ast\sfc}$] \label{lemma:potentialWstarC}
	Let $W_{\ast\sfc}$ be defined by \eqref{eq:WstarC}. 
	\begin{enumerate}[label = (\roman*)]
		\item \label{lemma:WstarC-regularity} $W_{\ast\sfc} \in C(\cl\Omega \times \bbR)$, $W_{\ast\sfc}(\cdot,v) \in C^{1,1}(\cl\Omega)$ for all $v \in\bbR$, $W_{\ast\sfc}(x,\cdot) \in C^1(\bbR)$ for all $x\in\cl\Omega$, and $W_{\ast\sfc}$ is nonnegative.
		\item $W_{\ast\sfc}(x,v) = 0$ iff either $x \in \cl\Omega \setminus \Omega_0$ and $v = 0$ or $x \in \Omega_0$ and $v \in \set{0,\phi_+}$.
		\item \label{lemma:WstarC-formula} $W_{\ast\sfc}(x,v) = \sfc(x)g(v) - f^\ast(v-\sfa)$.
		\item \label{lemma:WstarC-lb} $W_{\ast\sfc}(x,v) \ge (\sfc(x) - {\ubar \sfc}) g(v)$ for all $x \in \cl\Omega$ and $v \in \bbR$.
	\end{enumerate}
\end{lemma}

\begin{proof}
	Combining the previous lemma with the definition \eqref{eq:WstarC} of $W_{\ast\sfc}$ as ${W_\ast}$ perturbed by the nonnegative function $(x,v)\mapsto(\sfc(x) - {\ubar\sfc})g(v)$ gives the corresponding properties of $W_{\ast\sfc}$.

\end{proof}

The proofs of $\Gamma$-convergence and phase separation and the derivation of the limiting velocity law are facilitated by a lower bound on the energy density of the Modica-Mortola functionals $\scF_{\sfc,\varepsilon}$ and $\scF_\varepsilon$, which arises from the classical Modica-Mortola trick of combining Young's inequality with the chain rule:
\begin{equation} \label{eq:MMtrick:psibar}
	\scF_\varepsilon[\phi] = \frac{1}{\varepsilon} \int_\Omega {W_\ast}(\phi) \,dx + \frac{\varepsilon}{2}\int_\Omega \abs{\nabla\phi}_{\sfA(x)}^2 \,dx 
	\ge \int_\Omega  \abs{\sqrt{2{W_\ast}(\phi)}\nabla\phi}_{\sfA(x)} \,dx \eqcolon \int_\Omega \abs{\nabla\psi}_{\sfA(x)}\,dx. 
\end{equation}
Above, $\psi(x) \coloneq F(\phi(x))$ for a transformation 
\begin{equation} \label{eq:auxF}
	F \colon \bbR \to \left[0,+\infty\right[, \qquad  	F(v) \coloneq
	\begin{cases}
		0	&	\text{if } v \le 0, \\
		\displaystyle\int_0^v \sqrt{2{W_\ast}(s)} \,ds & \text{if } 0 < v < \phi_+, \\
		\displaystyle \int_0^{\phi_+} \sqrt{2{W_\ast}(s)} \,ds & \text{if } v \ge \phi_+.
	\end{cases}
\end{equation}

\medskip

\begin{lemma}[the transformation $F$] \label{lemma:transformationF}
	The function $F$ defined in \eqref{eq:auxF} is bounded, nondecreasing, Lipschitz, and continuously differentiable on $\bbR$.
\end{lemma}

\begin{proof}
	Lemma \ref{lemma:potentialWstar} shows ${W_\ast}$ is nonnegative, continuous, and vanishes exactly at $\set{0,\phi_+}$, which implies $F$ is nondecreasing and $C^1$. That $F$ is bounded is immediate from its definition, and that $F'$ is bounded follows from the continuity of ${W_\ast}$ on $[0,\phi_+]$.
\end{proof}

\subsection{A compactness result}
We first show 
\begin{lemma}[coercivity] \label{lemma:coercivity}
	Suppose \ref{hyp-on-A:ellipticity},  \ref{hyp-g:growthCondition},  \ref{hyp-f:growthCondition}, and \ref{hyp-c:regularity} are satisfied. 
	\begin{enumerate} [label = (\roman*)]
		\item \label{lemma:coercivity-Eeps} There exist constants $c,C>0$ such that for any $\varepsilon>0$ and $(\rho,\phi)$
		\begin{equation*}
			\scE_\varepsilon[\rho,\phi] \ge \frac{c}{\varepsilon}\Big( \int_\Omega \rho^m \,dx + \int_\Omega \abs{\phi}^q \,dx + \varepsilon^2 \int_\Omega \abs{\nabla\phi}^2\,dx \Big) - \frac{C}{\varepsilon}.
		\end{equation*}
		
		\item \label{lemma:coercivity-aprioriEst} Given a  positive sequence $\seq{\varepsilon_n}$ and a well-prepared sequence $\seq{(\rho_n, \phi_n)}$ \ref{hyp:well-preparedG},  there exists a constant $C>0$,  independent of both $\seq{\varepsilon_n}$ and  $\seq{(\rho_n,\phi_n)}_{n}$, such that  
		\[
		\norm{\rho_n}_{L^m(\Omega)} \le C (\ol\scG  \varepsilon_n)^{1/m} + C \quad\text{and}\quad \norm{\phi_n}_{L^q(\Omega)} \le C (\ol\scG \varepsilon_n)^{1/q} + C	\qquad \forall\, n \in \bbN.
		\] 
	\end{enumerate}
\end{lemma}
We recall when $f$ and $g$ are compatible for phase separation \ref{hyp:compatibilityConditionF&G} that $\scG_\varepsilon \coloneq \scE_\varepsilon + \scA_\varepsilon$ with a positive constant $\scA_\varepsilon \coloneq \sfa/\varepsilon$, so $\scG_\varepsilon \ge \scE_\varepsilon$. Consequently, $\scG_\varepsilon$ satisfies the same coercivity inequality as in Part \ref{lemma:coercivity-Eeps}. 
\begin{proof}

	The estimate is trivial if $\scE_\varepsilon[\rho,\phi] = +\infty$, so suppose otherwise. The coercivity of $f$ \ref{hyp-f:growthCondition} and $g$ \ref{hyp-g:growthCondition}, the positive minimum ${\ubar\sfc}$ of $\sfc$ \ref{hyp-c:regularity}, and uniform ellipticity of $\sfA$ \ref{hyp-on-A:ellipticity} give 
	\begin{gather*}
		K_2 R_2^m + f(u) \ge K_2u^m\text{ for all } u \ge 0, \quad {\ubar\sfc}K_1 R_q^q + \sfc(x)g(v) \ge {\ubar\sfc}K_1\abs{v}^q \text{ for all } (x,v)\in\cl\Omega \times \bbR, \\
		\abs{p}^2_{\sfA(x)} \ge \ubar{\sfA} \abs{p}^2 \text{ for all } (x,p) \in \cl\Omega \times \bbR^d,
	\end{gather*} 
	which provides
	\[
	\scE_\varepsilon[\rho,\phi] \ge \frac{1}{\varepsilon}\int_\Omega \big(K_2\rho^m - K_2 R_2^m\big) - \rho\phi + {\ubar\sfc}\big(K_1\abs{\phi}^q - K_1R_1^q\big)\,dx + \frac{\varepsilon}{2}\int_\Omega \ubar{\sfA}\abs{\nabla\phi}^2\,dx.
	\]
	Then given $0 < \delta < {\ubar\sfc}K_1$, Young's inequality $uv \le \frac{1}{(q\delta)^{q'/q}q'}u^{q'} + \delta v^q$ gives 
	\[
	\ge \frac{1}{\varepsilon}\int_\Omega \frac{K_2}{2}\rho^m + \big(\frac{K_2}{2}\rho^m - \frac{1}{(q\delta)^{q'/q}q'} \rho^{q'} \big) + \big({\ubar\sfc}K_1 -\delta\big) \abs{\phi}^q \,dx + c\varepsilon\int_\Omega \abs{\nabla\phi}^2\,dx - \frac{C}{\varepsilon}\abs{\Omega}.
	\]
	Since $m > q'$ (cf.\ \ref{hyp-f:growthCondition}) and $K_1,\delta,q,q'>0$, the function $\left[0,+\infty\right[ \ni u \mapsto \frac{K_1}{2} u^m - \frac{1}{(q\delta)^{q'/q}q'} u^{q'}$ is bounded below, which implies the estimate in  \ref{lemma:coercivity-Eeps}.
	
	\medskip
	\ref{lemma:coercivity-aprioriEst}: The estimate from \ref{lemma:coercivity-Eeps} and well-preparedness $\ol\scG \coloneq \sup_{n\in\bbN} \scG_{\varepsilon_n}[\rho_n,\phi_n] < +\infty$ give:
	\[
	\int_\Omega \rho_n^m \,dx + \int_\Omega \abs{\phi_n}^q \,dx \le \frac{\varepsilon_n}{c}\scG_{\varepsilon_n}[\rho_n,\phi_n] + C \le  C\ol{\scG}\varepsilon_n +C.
	\]
\end{proof}

The following shows well-prepared sequences $\seq{(\rho_n,\phi_n)}$ admit a convergent subsequence, and that the energy $\scG_{\varepsilon_n}[\rho_n,\phi_n]$ evaluated along such sequences is eventually no smaller than our candidate for the $\Gamma$-limit evaluated at the limit $(\rho,\phi)$. It also justifies the use of the $L^1(\Omega)^2$-topology for the $\Gamma$-convergence of $\scG_\varepsilon$. We recall that classical results \cite{Sternberg_1988_EffectSingularPerturbation,Fonseca.Tartar_1989_GradientTheoryPhase,Bouchitte_1990_SingularPerturbationsVariational,Owen.Sternberg_1991_NonconvexVariationalProblems} on the $\Gamma$-convergence of Modica-Mortola-type functionals toward the perimeter are also w.r.t.\ the $L^1(\Omega)$-topology, so the fact that it is also the natural choice for $\scG_\varepsilon$ is unsurprising in view of the relationship \eqref{eq:GepsModicaMortola} between it and the Modica-Mortola functional $\scF_{\sfc,\varepsilon}$.
\begin{proposition}[compactness] \label{prop:compactness}
	Make the standard assumptions \ref{hyp-standardAssumptions}, suppose $f$ and $g$ are compatible \ref{hyp:compatibilityConditionF&G}, and $\sfc$ satisfies \ref{hyp-c:measure}. Let $\seq{\varepsilon_n}$ be a positive sequence converging to zero and $\seq{(\rho_n, \phi_n)}$ be well-prepared \ref{hyp:well-preparedG}. There exists a (not relabeled) subsequence, $\rho \in BV(\Omega; \set{0,\rho_+})$, and $\phi \in BV(\Omega; \set{0,\phi_+})$ such that 
	\begin{enumerate}[label = (\roman*)]
		\item \label{prop:compactness-rhoConvergence} $\rho_n \xrightarrow{n \to \infty} \rho$ strongly in $L^s(\Omega)$ for all $1 \le s < m$, weakly in  $L^m(\Omega)$, and a.e.; 
		
		\item \label{prop:compactness-phiConvergence} $\phi_n \xrightarrow{n\to\infty} \phi$ strongly in $L^r(\Omega)$ for all $1 \le r < q$, weakly in  $L^q(\Omega)$, and a.e.;
		
		\item \label{prop:compactness-domainOfLimit} $\displaystyle\phi = \frac{\phi_+}{\rho_+}\rho$ a.e.\ in  $\Omega$,  $\chi_{\cl\Omega \setminus \Omega_0} \phi = 0$ a.e.\ in $\Omega$, and $\displaystyle\int_\Omega \rho \,dx = 1$; 
		
		\item \label{prop:compactness-lsc} $\displaystyle \liminf_{n\to\infty} \scG_{\varepsilon_n}[\rho_n,\phi_n] \ge \scG_0[\rho,\phi]$. 
		
	\end{enumerate}
\end{proposition}

\begin{remark}[strong convergence] 
		\label{rmk:strongConvergence-oneFromTheOther} Whenever a sequence $\seq{(\rho_n,\phi_n)}$ is well-prepared and one of the sequences $\seq{\rho_n}$ or $\seq{\phi_n}$ converges strongly in $L^1(\Omega)$ to $\rho_+\chi_E$ or $\phi_+\chi_E$ (resp.), then the other sequence converges in $L^1(\Omega)$ to the appropriate multiple of $\chi_E$ \emph{without passing to a subsequence}. When $g(v)=v^2/2$, the penalty coupling $(\rho,\phi)$ in the definition \eqref{eq:augmentedEnergy} of $\scG_\varepsilon$ is  $\varepsilon^{-1}\int_\Omega [{\ubar\sfc}\,g]^\ast(\rho) -\rho\phi + {\ubar\sfc}\,g(\phi)\,dx = \varepsilon^{-1}\int_\Omega \frac{1}{2{\ubar\sfc}}(\rho - {\ubar\sfc}\,\phi)^2 \,dx$, so this is obvious. Generally, it follows from the  compactness and uniqueness of limit argument we carry out in the following proof. We use this in Section \ref{subsubsec:Geps-recoverySequence} to deduce the  $L^1$-convergence of our recovery sequence for $\rho$ from that of $\phi$.
\end{remark}

\begin{proof}
	Set $\psi_n(x)\coloneq F(\phi_n(x))$. We will show $\seq{\psi_n}$ is bounded in $BV(\Omega)$. Since $BV(\Omega)$ is compactly embedded into $L^1(\Omega)$, this will give a strongly convergent subsequence. We will use this to obtain a strongly convergent subsequence of $\seq{\phi_n}$. 
	
	To obtain an $L^1(\Omega)$ bound, we note Lemma \ref{lemma:transformationF} shows $F$ is Lipschitz (and vanishes at $0$ from its definition \eqref{eq:auxF}), so $\abs{\psi_n}\le C\abs{\phi_n}$; it thus suffices to bound $\seq{\phi_n}$ in $L^1(\Omega)$. For this aim, Lemma \ref{lemma:coercivity}-\ref{lemma:coercivity-Eeps} shows $\seq{\phi_n}$ is bounded in $L^q(\Omega)$. Since $q>1$ and we work on a bounded domain, $L^q(\Omega) \hookrightarrow L^1(\Omega)$. 
	
	Since $\scF_\varepsilon[\phi]\le \scF_{\sfc,\varepsilon}[\phi] \le \scG_{\varepsilon}[\rho,\phi]$ by \eqref{eq:GFineq} and  \eqref{eq:ModicaMortolaF}, the Modica-Mortola trick \eqref{eq:MMtrick:psibar} and uniform ellipticity \ref{hyp-on-A:ellipticity} give 
	\[
	\ol\scG \ge \scF_{\varepsilon_n}[\phi_n] \ge \int_\Omega \abs{\nabla\psi_n}_{\sfA(x)}\,dx \ge \sqrt{\ubar{\sfA}}\int_\Omega \abs{\nabla\psi_n}\,dx,
	\]
	so $\seq{\nabla\psi_n}$ is bounded in $L^1(\Omega)^d$. Hence, $\seq{\psi_n}$ is bounded in $BV(\Omega)$.
	
	\medskip
	We now pass to a not relabled subsequence of $\seq{\psi_n}$ that converges strongly in $L^1(\Omega)$ to $\gamma \phi$. We note that $\phi \in BV(\Omega)$ by lower semicontinuity of the $BV(\Omega)$-seminorm for this convergence. 
	
	\medskip
	We claim $\phi_n \to \phi$ strongly in $L^1(\Omega)$ as well. Fix a small $\delta > 0$ and consider the open neighborhood $V_\delta \coloneq \cset{v \in \left[0,+\infty\right[}{\dist\big({v, \set{0,\phi_+}\big) < \frac{\delta}{2}}}$; we claim there exists $C_\delta > 0$ such that 
	\[
	\abs{F(v) - \gamma v}^q \le C_\delta {W_\ast}(v)	\qquad \forall\, v \in \bbR \setminus V_\delta.
	\] 
	To see the estimate for large $v \in \bbR \setminus V_\delta$, note the function $v \mapsto F(v) - \gamma v$ is Lipschitz on $\bbR$ (Lemma \ref{lemma:transformationF}) so it grows at most linearly, while ${W_\ast}$ grows at least as fast as the $q$-th power as $\abs{v}\to+\infty$, where $q>1$, (Lemma \ref{lemma:potentialWstar}-\ref{lemma:Wstar-coercive}). To obtain the estimate for small $v \in \bbR \setminus V_\delta$, first observe $v\mapsto F(v) - \gamma v$ and $W_\ast$ have the same zeros $\set{0,\phi_+}$ (recall the definitions of $F$ \eqref{eq:auxF} and $\gamma$ \eqref{eq:surfaceTensionCoeff}, as well as Lemma \ref{lemma:potentialWstar}-\ref{lemma:Wstar-zeros}). Then, choose $C_\delta$ possibly larger than before according to a compactness argument. Control in $V_\delta$ is a result of $v \mapsto F(v) - \gamma v$ being Lipschitz and the connected components of $V_\delta$ having diameter smaller than $\delta$:
	\[
	\forall\, v \in V_\delta \qquad \abs{F(v) - \gamma v} \le C\delta.
	\]
	Combining gives 
	\begin{align*}
		\norm{\psi_n -  \gamma \phi_n}_{L^q(\Omega)}^q \le C\abs{\Omega}\delta^q +  C_\delta \int_{\set{\phi_n \not\in V_\delta}} {W_\ast}(\phi_n)\,dx 
		\le C \abs{\Omega}\delta^q +  C_\delta \varepsilon_n \scF_{\varepsilon_n}[\phi_n] \le C\abs{\Omega}\delta^q+ C_\delta \ol{\scG}\varepsilon_n,  
	\end{align*}
	and because $\varepsilon_n \to 0$,
	\[
	\limsup_{n\to\infty} \norm{\psi_n - \gamma \phi_n}_{L^q(\Omega)} \le C\abs{\Omega}^{1/q} \delta.
	\]
	Since this holds for all $\delta>0$ and the left-hand side is independent of $\delta$, sending $\delta \to 0^+$ gives $\limsup_{n\to\infty} \norm{\psi_n - \gamma \phi_n}_{L^q(\Omega)} = 0$, and then the triangle inequality and embedding $L^q(\Omega) \hookrightarrow L^1(\Omega)$ give $\phi_n \to \phi$ in $L^1(\Omega)$, which shows one of the modes of convergence in \ref{prop:compactness-phiConvergence}.
	
	\medskip
	
	By passing to a further subsequence, we also have $\phi_n \to \phi$ a.e. Then continuity of ${W_\ast}$ (Lemma \ref{lemma:potentialWstar}-\ref{lemma:Wstar-nonneg}) gives ${W_\ast}\circ \phi_n \to {W_\ast}\circ \phi$ a.e., so Fatou's lemma and well-preparedness provide
	\[
	\int_\Omega {W_\ast}(\phi)\,dx \le \liminf_{n\to\infty} \int_\Omega {W_\ast}(\phi_n)\,dx \le \liminf_{n\to\infty}\big( \varepsilon_n \scF_{\varepsilon_n}[\phi_n]\big) \le \ol\scG \liminf_{n\to\infty} \varepsilon_n = 0.
	\]
	Since ${W_\ast}$ is nonnegative, we deduce ${W_\ast}(\phi) = 0$ a.e., so there exists a measurable set $E \subset \Omega$ such that $\phi = \phi_+\chi_E$ since $\set{{W_\ast}=0}=\set{0,\phi_+}$.  Continuity of $g$,  Fatou's lemma, and well-preparedness similarly give
	\[
	\int_\Omega (\sfc - {\ubar\sfc})g(\phi)\,dx \le \liminf_{n\to\infty} \int_\Omega (\sfc(x) - {\ubar\sfc})g(\phi_n)\,dx \le \liminf_{n\to\infty}\big( \varepsilon_n \scG_{\varepsilon_n}[\phi_n]\big) \le  \ol{\scG}\liminf_{n\to\infty} \varepsilon_n = 0,
	\]
	and since $(\sfc - {\ubar\sfc})g(\phi)$ is nonnegative we must have $0 = (\sfc - {\ubar\sfc})g(\phi) = (\sfc - {\ubar\sfc})g(\phi_+)\chi_E$ a.e. Since $g(\phi_+)>0$, this provides the complementarity condition in \ref{prop:compactness-domainOfLimit}.
	
	\medskip
	At this stage, we may have $\abs{E} = 0$, but this is resolved by the two remaining equalities in \ref{prop:compactness-domainOfLimit}. 
	
	As mentioned in the remark preceding this proof, the well-preparedness of $\seq{(\rho_n,\phi_n)}$ and strong $L^1$-convergence of $\seq{\phi_n}$ to $\phi = \phi_+\chi_E$ imply $\seq{\rho_n}$ converges strongly in $L^1$ to $\rho \coloneq \rho_+ \chi_E$ \emph{without passing to a further subsequence}. To establish this, we will show every subsequence of the current subsequence $\seq{(\rho_n,\phi_n)}$ admits a further subsequence converging strongly in $L^1$ to $(\rho,\phi)$. For this purpose, we now pass to an arbitrary (not relabeled) subsequence of $\seq{(\rho_n,\phi_n)}$. 
	
	Lemma \ref{lemma:coercivity}-\ref{lemma:coercivity-aprioriEst} implies $\seq{\rho_n}$ is bounded in $L^m(\Omega)$, so by passing to a subsequence, $\rho_n \rightharpoonup \tilde\rho$ weakly in $L^m(\Omega)$. Testing this convergence with nonnegative functions and constant functions implies $\tilde\rho \ge 0$ a.e.\ and $\int_\Omega \tilde\rho\,dx=1$, resp. 
	
	We show $\phi = \frac{\phi_+}{\rho_+}\tilde\rho$, which shows $\tilde\rho = \rho$ a.e.\ as well as $\abs{E}>0$ in view of $\int_\Omega \tilde\rho\,dx=1$; this will complete the proof of \ref{prop:compactness-domainOfLimit}. To establish this identity, we show $[{\ubar\sfc}\,g]^\ast(\tilde\rho) - \tilde\rho\phi + {\ubar\sfc}\,g(\phi) = 0$ a.e. Well-preparedness provides
	\begin{equation} \label{eq:FenchelVanishing}
		0\le\limsup_{n\to\infty} \int_\Omega [{\ubar\sfc}\,g]^\ast(\rho_n) - \rho_n\phi_n + {\ubar\sfc}\,g(\phi_n) \,dx \le \limsup_{n\to\infty}\big(\varepsilon_n \scG_{\varepsilon_n}[\rho_n,\phi_n]\big) \le \ol{\scG}\lim_{n\to\infty}\varepsilon_n = 0,
	\end{equation}
	and since the Fenchel-Young inequality implies $[{\ubar\sfc}\,g]^\ast(\rho_n) - \rho_n\phi_n + {\ubar\sfc}\,g(\phi_n) \ge 0$ a.e., we have $[{\ubar\sfc}\,g]^\ast(\rho_n) - \rho_n\phi_n + {\ubar\sfc}\,g(\phi_n) \to 0$ in $L^1(\Omega)$. It suffices to show $\int_\Omega[{\ubar\sfc}\,g]^\ast(\rho_n) - \rho_n\phi_n + {\ubar\sfc}\,g(\phi_n)\,dx$ is asymptotically no smaller than $\int_\Omega [{\ubar\sfc}\,g]^\ast(\tilde\rho) - \tilde\rho\phi + {\ubar\sfc}\,g(\phi)\,dx$.  
	
	Combining the uniform $L^q(\Omega)$ bound on $\seq{\phi_n}$ from Lemma \ref{lemma:coercivity}-\ref{lemma:coercivity-aprioriEst} and the pointwise a.e.\ convergence from earlier implies the sequence converges weakly in $L^q(\Omega)$. We improve the strong convergence by interpolating $L^p$-norms and the uniform bound in $L^q$,	which completes the proof of \ref{prop:compactness-phiConvergence}. 
	
	Since $m > q'$, we have $m'<q$. Hence, $\phi_n \to \phi$ strongly in $L^{m'}(\Omega)$, and since $\rho_n \rightharpoonup \tilde\rho$ weakly in $L^m(\Omega)$ we can conclude 
	\begin{equation} \label{eq:weakStrongConvergence}
		\lim_{n\to\infty}\int_\Omega \rho_n\phi_n \,dx = \int_\Omega \tilde\rho\phi\,dx.
	\end{equation}

	The strict convexity and smoothness of $g$ imply the same properties of $g^\ast$, which then hold for the rescaled function $[{\ubar\sfc}\,g]^\ast(u) = {\ubar\sfc}\,g^\ast(u/{\ubar\sfc})$. Hence, each of the functionals $\varphi \mapsto \int_\Omega {\ubar\sfc}\,g(\varphi)\,dx$ and $\mu \mapsto \int_\Omega \big[{\ubar\sfc}\,g\big]^\ast(\mu)\,dx$ is weakly sequentially lower semicontinuous for $L^1(\Omega)$, and thus also for $L^p(\Omega)$ with $1 \le p < +\infty$ since $\Omega$ is bounded. Therefore, 
	\begin{equation}\label{eq:lscReactionNonlinearities}
		\int_\Omega [{\ubar\sfc}\,g]^\ast(\tilde\rho) \,dx \le \liminf_{n\to\infty} \int_\Omega [{\ubar\sfc}\,g]^\ast(\rho_n)\,dx	\quad\text{and}\quad \int_\Omega {\ubar\sfc}\,g(\phi) \,dx \le \liminf_{n\to\infty} \int_\Omega {\ubar\sfc}\,g(\phi_n)\,dx.
	\end{equation}
	Combining \eqref{eq:FenchelVanishing}, \eqref{eq:weakStrongConvergence}, and \eqref{eq:lscReactionNonlinearities} shows 
	\begin{align*}
		0 &= \lim_{n\to\infty} \int_\Omega [{\ubar\sfc}\,g]^\ast(\rho_n) - \rho_n\phi_n + {\ubar\sfc}\,g(\phi_n)\,dx \\
		&\ge  \liminf_{n\to\infty} \int_\Omega [{\ubar\sfc}\,g]^\ast(\rho_n)\,dx - \lim_{n\to\infty}\int_\Omega \rho_n\phi_n\,dx +  \liminf_{n\to\infty} \int_\Omega {\ubar\sfc}\,g(\phi_n)\,dx 
		\ge \int_\Omega [{\ubar\sfc}\,g]^\ast(\tilde\rho) - \tilde\rho\phi + {\ubar\sfc}\,g(\phi)\,dx 
		\ge 0.
	\end{align*}
	Thus, the inequalities are all equalities, and since the final integrand is nonnegative (again by the Fenchel-Young inequality), we conclude 
	$
	[{\ubar\sfc}\,g]^\ast(\tilde\rho) - \tilde\rho\phi + {\ubar\sfc}\,g(\phi) = 0 \text{ a.e.\ in } \Omega.
	$
	Recall $\phi = \phi_+\chi_E$, so for $x\in E$ we deduce $\tilde\rho(x) = {\ubar\sfc}\,g'(\phi_+) = \rho_+$ from the definition of $\phi_+$ (Lemma \ref{lemma:potentialWstar}-\ref{lemma:Wstar-zeros}) and the identity ${[{\ubar\sfc}\,g]^\ast}'= ({\ubar\sfc}\,g')^{-1}$. Similarly, if $x \in \Omega \setminus E$, we have $\tilde\rho(x) = {\ubar\sfc}\,g'(0)=0$. Hence, $\tilde\rho = \rho_+\chi_E \eqcolon \rho = \frac{\rho_+}{\phi_+}\phi$ a.e., and since $\int_\Omega\tilde\rho\,dx=1$, we see $\abs{E}=1/\rho_+$. The complementarity condition $\chi_{\cl\Omega \setminus \Omega_0}\chi_E = 0$ a.e.\ from \ref{prop:compactness-domainOfLimit} then implies $\rho_+\abs{E \cap \Omega_0} = 1$, which is only possible by \ref{hyp-c:measure}. This completes the proof of \ref{prop:compactness-domainOfLimit}. Now, since $\phi \in BV(\Omega)$ and $\phi,\rho$ are scalar multiples of one another, $\rho\in BV(\Omega)$.
	
	\medskip
	Having established that a sub-subsequential limit of the density must be $\rho = \rho_+\chi_E$, where the set $E$ was determined by the limit $\phi = \phi_+ \chi_E$ of $\seq{\phi_n}$, our next aim is to show that the density converges strongly. In previous works on phase separation for PKS systems where $g(v) = v^2/2$, the penalty $\varepsilon^{-1}\int_\Omega[{\ubar\sfc}\,g]^\ast(\rho) - \rho\phi + {\ubar\sfc}\,g(\phi)\,dx$ in $\scG_\varepsilon$ evaluates to $\varepsilon^{-1}\int_\Omega \frac{1}{2{\ubar\sfc}}\rho^2 - \rho\phi + \frac{{\ubar\sfc}}{2}\phi^2\,dx = \varepsilon^{-1}\int_\Omega \frac{1}{2{\ubar\sfc}}(\rho - {\ubar\sfc}\phi)^2\,dx$. Thus, the $L^2(\Omega)$-distance between $\rho$ and ${\ubar\sfc}\,\phi$ is a term in $\scG_\varepsilon$, allowing to easily deduce  $L^1(\Omega)$-convergence of $\seq{\rho_n}$ from that of $\seq{\phi_n}$. In our setting, this penalty and the convergence \eqref{eq:FenchelVanishing} will still allow us to deduce the strong convergence of $\seq{\rho_n}$. Since we already know $\seq{\rho_n}$ and $\seq{\phi_n}$ converge to the appropriate (multiples of) characteristic functions, the left-hand side of \eqref{eq:FenchelVanishing} measures simultaneously how much each of $\rho_n$ and $\phi_n$ deviate from their limits. We would like to use the integral $\int_\Omega [{\ubar\sfc}\,g]^\ast(\rho_n) - \rho_n\phi + {\ubar\sfc}\,g(\phi)\,dx$ to control how much only the sequence $\seq{\rho_n}$ deviates from $\rho$.  Consequently, we first show the convergence of $\int_\Omega [{\ubar\sfc}\,g]^\ast(\rho_n)\,dx$. We use that $[{\ubar\sfc}\,g]^\ast(\rho) - \rho\phi + {\ubar\sfc}\,g(\phi) = 0$ a.e.\ to write
	\begin{align*}
		0 \le \int_\Omega [{\ubar\sfc}\,g]^\ast(\rho_n) - \rho_n\phi_n + {\ubar\sfc}\,g(\phi_n)\,dx 
		= \int_\Omega [{\ubar\sfc}\,g]^\ast(\rho_n) - [{\ubar\sfc}\,g]^\ast(\rho)\,dx + \int_\Omega {\ubar\sfc}\, g(\phi_n) - {\ubar\sfc}\,g(\phi) \,dx + \int_\Omega \rho\phi - \rho_n\phi_n\,dx. 
	\end{align*}
	By \eqref{eq:FenchelVanishing}, the integral on the first line vanishes as $n\to\infty$, so the second line vanishes, too. Moreover, since $\rho_n \rightharpoonup \rho$ weakly in $L^m$ and $\phi_n \to \phi$ in $L^{m'}$, the third integral on the second line vanishes. Rearranging, taking the limit $n\to\infty$, and using lower semicontinuity \eqref{eq:lscReactionNonlinearities} shows
	\begin{align*}
		0 &= \lim_{n\to\infty} \int_\Omega [{\ubar\sfc}\,g]^\ast(\rho_n) - \rho_n\phi_n + {\ubar\sfc}\,g(\phi_n)\,dx + \lim_{n\to\infty} \int_\Omega \rho_n\phi_n - \rho\phi\,dx \\
		&= \lim_{n\to\infty} \Big[ \int_\Omega [{\ubar\sfc}\,g]^\ast(\rho_n) - [{\ubar\sfc}\,g]^\ast(\rho)\,dx + \int_\Omega {\ubar\sfc}\, g(\phi_n) - {\ubar\sfc}\,g(\phi) \,dx \Big] \\
		&\ge \liminf_{n\to\infty} \int_\Omega [{\ubar\sfc}\,g]^\ast(\rho_n) - [{\ubar\sfc}\,g]^\ast(\rho)\,dx + \liminf_{n\to\infty}\int_\Omega {\ubar\sfc}\, g(\phi_n) - {\ubar\sfc}\,g(\phi) \,dx 
		\ge 0,
	\end{align*}
	so the inequalities are equalities. Because each $\liminf$ is  nonnegative \eqref{eq:lscReactionNonlinearities}, we deduce each is zero. Passing to a subsequence that achieves the first  $\liminf$ proves what was aimed. 
	
	\medskip
	We now show the strong $L^1(\Omega)$-convergence of the density. We first define $G(u)\coloneq F(\frac{\phi_+}{\rho_+}u)$, which has the property $G(\rho) = F(\phi)$ a.e., and we will rewrite the limiting density using the identities $\rho = \frac{\rho_+}{\phi_+}\phi = \frac{\rho_+}{\phi_+\gamma}F(\phi)$. (Recall, the definitions of $\gamma$, \eqref{eq:surfaceTensionCoeff}, and $F$, \eqref{eq:auxF}, imply $\gamma = \frac{1}{\phi_+}F(\phi_+)$.)  We split the distance
	\[
	\abs{\rho_n - \rho} \le 
	\abs{ \rho_n - \frac{\rho_+}{\gamma\phi_+}G(\rho_n) } 
	+ \frac{\rho_+}{\gamma\phi_+} \abs{ G(\rho_n) - F(\phi) }.
	\] 
	
	\begin{itemize}
		\item For the first term, notice the Lipschitz function $\left[0,+\infty\right[ \ni u \mapsto u - \frac{\rho_+}{\gamma\phi_+}G(u)$ vanishes exactly at $\set{0,\rho_+}$, which are precisely the wells of $W$. Arguing as before with a small neighborhood of these wells $U_\delta \coloneq \cset{u \in \left[0,+\infty\right[}{\dist(u,\set{0,\rho_+}) < \frac{\delta}{2}}$, the Lipschitz continuity of $u - \frac{\rho_+}{\gamma\phi_+}G(u)$ and superlinearity of $W$ (Lemma \ref{lemma:potentialW}-\ref{lemma:Wcoercive}) provide a constant $C_\delta>0$ (depending only on $\delta,\frac{\rho_+}{\gamma\phi_+}G,W$) such that
		\begin{align*}
			\norm{\rho_n - \frac{\rho_+}{\gamma\phi_+}G(\rho_n)}_{L^1(\Omega)} \le C\abs{\Omega}\delta + C_\delta \int_{\set{\rho_n \not\in U_\delta}} W(\rho_n)\,dx 
			\le C\delta + C_\delta \varepsilon_n\scG_{\varepsilon_n}[\rho_n, \phi_n]  \le C\delta + C_\delta\ol{\scG}\varepsilon_n.
		\end{align*}
		Sending $n\to\infty$ and then $\delta\to0^+$ shows this term vanishes.
		
		\item  For the term $G(\rho_n) - F(\phi)$, we recall $\phi = \phi_+\chi_E$. If $x\in E$, this function evaluates to $G(\rho_n(x)) - F(\phi_+)$, and we are lead to consider the function $\left[0,+\infty\right[\ni u \mapsto G(u)-F(\phi_+)$. It is Lipschitz and vanishes precisely at $\rho_+$, which is the unique zero of the continuous, nonnegative function $\left[0,+\infty\right[ \ni u \mapsto [{\ubar\sfc}\,g]^\ast(u) - u\phi_+ + {\ubar\sfc}\,g(\phi_+)$. This latter function is superlinear at $u\to+\infty$ since $g^\ast$ is, and consequently for each small $\delta>0$ there exists a constant $C_\delta'$ such that $\abs{G(u) - F(\phi_+)} \le C_\delta' \big([{\ubar\sfc}\,g]^\ast(u) - u\phi_+ + {\ubar\sfc}\,g(\phi_+) \big)$ for all $u \not\in U^+_\delta \coloneq \left]\rho_+ - \frac{\delta}{2},\rho_+ + \frac{\delta}{2}\right[$. The Lipschitz continuity provides the estimate $\abs{G(u) - F(\phi_+)} \le C\delta$ for $u \in U_\delta^+$, so
		\[
		\int_E \abs{G(\rho_n) - F(\phi)} \,dx \le C \abs{\Omega}\delta + C_\delta'\int_E [{\ubar\sfc}\,g]^\ast(\rho_n) - \rho_n \phi + {\ubar\sfc}\,g(\phi)\,dx. 
		\]
		The previous paragraph gave $\int_\Omega [{\ubar\sfc}\,g]^\ast(\rho_n)\,dx \to \int_\Omega [{\ubar\sfc}\,g]^\ast(\rho)\,dx$ and \ref{prop:compactness-rhoConvergence} gives $\rho_n\rightharpoonup \rho$ weakly in $L^m(\Omega)$, so
		\[
		\limsup_{n\to\infty} \int_E \abs{G(\rho_n) - F(\phi)}\,dx \le C\delta + C_\delta'\int_E  [{\ubar\sfc}\,g]^\ast(\rho) - \rho \phi + {\ubar\sfc}\,g(\phi)\,dx = C\delta,
		\] 
		where the last equality follows since $[{\ubar\sfc}\,g]^\ast(\rho) - \rho \phi + {\ubar\sfc}\,g(\phi)=0$ a.e.
		
		A completely analogous argument on the region $\cl\Omega \setminus E$ is used to control $G(\rho_n) - F(0)$ by $[{\ubar\sfc}\,g]^\ast(\rho_n) -\rho_n\cdot0 + {\ubar\sfc}\,g(0)$, so combining provides
		$
		\limsup_{n\to\infty} \norm{G(\rho_n) - F(\phi)}_{L^1(\Omega)} \le C\delta 
		$, 
		and we send $\delta \to 0^+$.
	\end{itemize}
	Thus, $\rho_n \to \rho$ in $L^1(\Omega)$, which completes our proof that $\seq{\rho_n}$ converges strongly along the same subsequence as $\seq{\phi_n}$. By interpolating $L^p(\Omega)$-norms and using the uniform bound in $L^m(\Omega)$, we obtain the strong convergence in $L^s(\Omega)$ for all $1\le s < m$, which completes the proof of \ref{prop:compactness-rhoConvergence}. 	
	
	\medskip
	Finally, we show \ref{prop:compactness-lsc}. Because $(x,p)\mapsto \abs{p}_{\sfA(x)}$ is continuous (a consequence of \ref{hyp-on-A:regularity}) as well as positively 1-homogeneous and convex in $p$ (consequences of its definition \eqref{eq:Aip} and \ref{hyp-on-A:ellipticity}), the Reshetnyak lower semicontinuity theorem \cite[Theorem 2.38]{Ambrosio.Fusco.ea_2000_FunctionsBoundedVariation} provides that the functional 
	\[
	BV(\Omega) \ni \mu \mapsto \int_\Omega \abs{\frac{\nabla\mu}{\abs{\nabla\mu}}(x)}_{\sfA(x)}\,d\vert\nabla\mu\vert(x)
	\]
	is lower semicontinuous for weak-$\ast$ convergence in $BV(\Omega)$. Above, the notation $\nabla\mu/{\abs{\nabla\mu}}$ denotes the Radon-Nikodym derivative of the measure $\nabla\mu$ w.r.t.\ its total variation $\abs{\nabla\mu}$.  Our work has shown that the auxiliary functions $\seq{\psi_n}$ converge weakly-$\ast$ in $BV(\Omega)$ to $\gamma\phi$, so the 
	vector-valued Radon measures $\seq{\nabla\psi_n}$ converge weakly-$\ast$ to $\nabla(F \circ \phi) = \gamma\nabla\phi$. As a consequence of  \eqref{eq:GFineq}, \eqref{eq:ModicaMortolaF}, \eqref{eq:MMtrick:psibar}, and Reshetnyak lower semicontinuity, we obtain
	\begin{align*}
		\liminf_{n\to\infty} \scG_{\varepsilon_n}[\rho_n, \phi_n] \ge \liminf_{n\to\infty} \scF_{\varepsilon_n}[\phi_n] 
		\ge \liminf_{n\to\infty} \int_\Omega \abs{\nabla\psi_n(x)}_{\sfA(x)} \,dx  
		\ge  \gamma\int_\Omega \abs{\vec{\nu}(x)}_{\sfA(x)}d\vert\nabla\phi\vert(x) = \scG_0[\rho,\phi],
	\end{align*}
	which concludes the proof of \ref{prop:compactness-lsc}. 
\end{proof}

\subsection{Proof of the \texorpdfstring{$\Gamma$-$\liminf$}{\textGamma-liminf} inequality (Theorem \ref{thm:gammaConvergence}-\ref{thm:liminfIneq})}

Let $(\rho,\phi) \in L^1(\Omega)^2$, $\seq{\varepsilon_n}$ be a positive sequence converging to $0$, and $\seq{(\rho_n, \phi_n)}$ be a sequence in $L^1(\Omega)^2$ converging in $L^1(\Omega)^2$ to $(\rho,\phi)$. If $\liminf_{n \to \infty} \scG_{\varepsilon_n}[\rho_n, \phi_n] = +\infty$, then we are done, so we suppose this limit inferior is finite and pass to a (not relabeled) subsequence that achieves it: $\liminf_{n \to \infty} \scG_{\varepsilon_n}[\rho_n, \phi_n] = \lim_{n \to \infty} \scG_{\varepsilon_n}[\rho_n, \phi_n]$. By perhaps passing to a further subsequence, we can suppose $\ol{\scG} \coloneq \sup_{n\in\bbN} \scG_{\varepsilon_n}[\rho_n, \phi_n] < +\infty$. Now we are in the setting of Proposition \ref{prop:compactness}. By passing to another subsequence, Proposition \ref{prop:compactness}-\ref{prop:compactness-lsc} allows us to conclude.

	\subsection{Proof of the \texorpdfstring{$\Gamma$-$\limsup$}{\textGamma-limsup} inequality (Theorem \ref{thm:gammaConvergence}-\ref{thm:limsupIneq})}
	The construction of a recovery sequence for $\seq{\scG_\varepsilon}$ will rely on the construction of a recovery sequence for the Modica-Mortola functionals
	\[
	\scF_{\sfc,\varepsilon}[\phi] \coloneq 
	\begin{cases}
		\displaystyle\frac{1}{\varepsilon}\int_\Omega W_{\ast\sfc}(x,\phi) \,dx + \frac{\varepsilon}{2} \int_\Omega \abs{\nabla \phi}_{\sfA(x)}^2\,dx & \text{if } \phi \in H^1(\Omega), \\
		+\infty &\text{otherwise in } L^1(\Omega).
	\end{cases}
	\] 
	We recall $\scG_\varepsilon[\rho,\phi] \ge \scF_{\sfc,\varepsilon}[\phi]$ in general (cf.\ \eqref{eq:GFineq}), and we have equality provided  $\varepsilon^{-1}\int_\Omega f(\rho) - \rho(\phi-\sfa) + f^\ast(\phi-\sfa)\,dx=0$, i.e., when $\rho$ is defined pointwise to satisfy \eqref{eq:defRhoGammaConvergence}. This observation facilitates the construction of a recovery sequence for the $\rho$ variable of $\seq{\scG_\varepsilon}$ from a recovery sequence for $\seq{\scF_{\sfc,\varepsilon}}$. In order to satisfy the mass constraint $\int_\Omega \rho \,dx =1$ on the density, we will also show some stability of our recovery sequences w.r.t.\ small term-by-term translations of the transition layer.
	
	\medskip
	\subsubsection{A recovery sequence for $\scF_{\sfc,\varepsilon}$}
	First, we construct an optimal profile $\omega$ for the functionals $\seq{\scF_{\sfc,\varepsilon}}$. 
	\begin{lemma}[optimal profile] \label{lemma:optimalProfile}
		Make the standard assumptions \ref{hyp-standardAssumptions}, suppose $f$ and $g$ are compatible \ref{hyp:compatibilityConditionF&G}, and suppose $g$ satisfies \ref{hyp-g:growthNear0-gamma}. Fix $0<\delta<1$ and for each $p \in \cl{B_1(0)}\setminus B_\delta(0)$ and $x\in\cl\Omega$ consider the Cauchy problem
		\begin{equation} \label{eq:CauchyPb-optimalProfile}
			\text{Seek } \omega(\cdot;x,p) \text{ such that } \qquad
			\begin{cases}
				\displaystyle \partial_z\omega(z;x,p) = -\frac{\sqrt{2W_{\ast\sfc}(x,\omega(z;x,p))}}{\abs{p}_{\sfA(x)}}	& \text{for all } z \in \bbR, \\
				\displaystyle \omega(0;x,p) = \frac{\phi_+}{2}.
			\end{cases} 
		\end{equation}
		\begin{enumerate}[label=(\roman*)]
			\item \label{lemma:OP-existence} There exists a unique solution $\omega$ to \eqref{eq:CauchyPb-optimalProfile}. This solution is $C^{1,1}$ on $\bbR \times \cl\Omega \times (\cl{B_1(0)} \setminus B_\delta(0))$ and $\omega(\cdot;x,p)$ is strictly decreasing.
			
			\item (convergence to $0$) \label{lemma:OP-conv0} Let $0<\omega_0\le \min\set{1,a,1/R_1}$, where $\sfa$ is from \ref{hyp:compatibilityConditionF&G} and $R_1$ is from \ref{hyp-g:growthCondition}. There exists $Z_0=Z_0(\omega_0)>0$ (large) such that
			\[
			\forall \, z \ge Z_0 \qquad 
			0 \le \omega(z;x,p) \le 
			\begin{cases}
				\omega_0\exp\left(-\sqrt{\frac{2{\ubar\sfc}}{\ol{\sfA}K_1}}(z-Z_0)\right) & \text{if } q_0=2, \\
				\omega_0\left[ 1 + \omega_0^{\frac{q_0-2}{2}}(\frac{q_0-2}{2})\sqrt{\frac{2{\ubar\sfc}}{\ol{\sfA}K_1}} (z-Z_0) \right]^{-2/(q_0-2)}	& \text{if } q_0 > 2,
			\end{cases}
			\]
			where $q_0$ is from \ref{hyp-g:growthNear0-gamma}.
			
			\item (convergence to $\phi_+$) \label{lemma:OP-convPhiPlus} If $x\in\Omega_0$, then 
			\[
			\forall \, z \le 0 \qquad 0 \le \phi_+ - \omega(z;x,p) \le \frac{\phi_+}{2} \exp\left( \frac{1}{\delta}\sqrt{\frac{2{\ubar\sfc}\,\norm{g''}_{L^\infty([\phi_+/2,\phi_+])}}{\ubar{\sfA}}} z\right).
			\]
		\end{enumerate}
	\end{lemma}
	\begin{proof}
		\ref{lemma:OP-existence}: We claim $v\mapsto\sqrt{W_{\ast\sfc}(x,v)}/\abs{p}_{\sfA(x)}$ is Lipschitz in $[0,\phi_+]$ uniformly in $(x,p) \in \cl\Omega \times (\cl{B_1(0)}\setminus B_\delta(0))$, so the existence and uniqueness of $\omega(\cdot;x,p)$ will follow from Cauchy-Lipschitz. Since $W_{\ast\sfc}(x,\cdot)$ is continuously differentiable (Lemma \ref{lemma:potentialWstarC}-\ref{lemma:WstarC-regularity}), we only need to address its growth near its zeros, i.e., $v=0$ (if $x\in\cl\Omega$) and $v=\phi_+$ (if $x\in\Omega_0$). By definition \eqref{eq:WstarC}, 
		\[
		W_{\ast\sfc}(x,v) = {W_\ast}(v) + (\sfc(x) - {\ubar\sfc})g(v).
		\]

		If $x\in\Omega_0$, then $W_{\ast\sfc}(x,v)={W_\ast}(v)$, so for the growth near $v=0$, Lemma \ref{lemma:potentialWstar}-\ref{lemma:Wstar-near0} gives 
		\[
		\forall\,\abs{v} \le \min\set{1,a,1/R_1} \qquad 	\abs{\sqrt{2W_{\ast\sfc}(x,v)} - \sqrt{2W_{\ast\sfc}(x,0)}} \le \sqrt{2{\ubar\sfc}\,K_1} \abs{v}^{q_1/2}.
		\]	
		Since $q_1/2 \ge (1+(1/2)q_0)/2 \ge 1$ (recall from  \ref{hyp-g:growthNear0-gamma} that $q_0\ge2$) and $\abs{v} \le 1$, we have $\abs{v}^{q_1/2} \le \abs{v}$. Otherwise, for $x\in\cl\Omega \setminus \Omega_0$, the double-well potential is perturbed $W_{\ast\sfc}(x,v) = {W_\ast}(v) + (\sfc(x) - {\ubar\sfc})g(v)$, so we combine Lemma \ref{lemma:potentialWstar}-\ref{lemma:Wstar-near0} with  \ref{hyp-g:growthNear0-gamma} to obtain for these small $v$
		\[
		\abs{\sqrt{2W_{\ast\sfc}(x,v)} - \sqrt{2W_{\ast\sfc}(x,0)}} \le \sqrt{2\sfc(x)K_1}\abs{v}^{q_1/2} \le \sqrt{2\max_{x\in\cl\Omega}\sfc(x)K_1}\abs{v}.
		\]
		
		If $x\in\Omega_0$, then $W_{\ast\sfc}(x,v) = {W_\ast}(v)$, so the growth near $v=\phi_+$ is described by Lemma \ref{lemma:potentialWstar}-\ref{lemma:Wstar-nearPhip} and we can conclude.
		
		It is straightforward, in view of \ref{hyp-on-A:regularity} and \ref{hyp-c:regularity}, to check that 
		$
		\sqrt{2\big(\sfc(x)g(v) - f^\ast(v-\sfa)\big)}/\abs{\sqrt{p\cdot \sfA(x)p}}
		$ 
		depends in a continuously differentiable manner on the parameters $(x,p) \in \cl\Omega \times (\cl{B_1(0)}\setminus B_\delta(0))$.
		
		\medskip
		
		\ref{lemma:OP-conv0}: Since $W_{\ast\sfc}(x,v) = {W_\ast}(v) + (\sfc(x) - {\ubar\sfc})g(v) \ge {W_\ast}(v)$, we estimate 
		\[
		\partial_z\omega = -\frac{\sqrt{2W_{\ast\sfc}(x,\omega)}}{\sqrt{p\cdot \sfA(x)p}} \le -\frac{\sqrt{2{W_\ast}(\omega)}}{\sqrt{\ol{\sfA}}}.
		\]
		Since $\omega(\cdot;x,p)$ is monotone decreasing, the above bound shows there exists $Z_0>0$, independent of $x,p$, large enough such that for all $z\ge Z_0$ we have $\omega(z;x,p) \le \omega_0 \le \min\set{1,a,1/R_1}$. For these large $z$,  Lemma \ref{lemma:potentialWstar}-\ref{lemma:Wstar-near0} gives ${W_\ast}(\omega) \ge {\ubar\sfc}\,\abs{\omega}^{q_0}/K_1$, so $\partial_z \omega \le -\sqrt{\frac{2{\ubar\sfc}}{\ol{\sfA}K_1}} \omega^{q_0/2}$, and the remainder of the proof is elementary.

		\medskip 
		
		\ref{lemma:OP-convPhiPlus}: For $x\in\Omega_0$, we have $W_{\ast\sfc}(x,v) = {W_\ast}(v)$, which combined with Lemma \ref{lemma:potentialWstar}-\ref{lemma:Wstar-nearPhip} gives 
		\[
		\partial_z(\phi_+ - \omega) = \frac{\sqrt{2{W_\ast}(\omega)}}{\sqrt{p \cdot \sfA(x)p}} \le  \frac{1}{\delta}\sqrt{\frac{2{\ubar\sfc}\,\norm{g''}_{L^\infty([\phi_+/2,\phi_+])}}{\ubar{\sfA}}} (\phi_+-\omega).
		\]
		We dropped the absolute value since $\omega(z) < \phi_+$ for all $z$. Gronwall's lemma gives the estimate.
	\end{proof}

	We first construct a recovery sequence for  $\seq{\scF_{\sfc,\varepsilon}}$ when $\phi = \phi_+\chi_E$ with $E$ having $C^2$ boundary. 
	
	\begin{lemma} \label{lemma:limsupFSmoothBdry}
		Suppose \ref{hyp-standardAssumptions},  \ref{hyp:compatibilityConditionF&G}, \ref{hyp-c:measure}, and  \ref{hyp-g:growthNear0-gamma}. Let $\seq{\varepsilon_n}$ be a positive sequence converging to zero and $\phi = \phi_+\chi_E \in BV(\Omega;\set{0,\phi_+})$ satisfy $\int_\Omega \phi \,dx = \phi_+/\rho_+$ and $\phi \chi_{\cl\Omega \setminus \Omega_0} = 0$ a.e. If $E$ is open and has $C^2$ boundary, then there exists a sequence $\seq{\phi_n} \subset H^1(\Omega)$ such that
		\[
		\lim_{n\to\infty}\phi_n = \phi \text{ in } L^1(\Omega) 
		\qquad \text{and} \qquad	
		\limsup_{n\to\infty}\scF_{\sfc,\varepsilon_n}[\phi_n] \le \gamma \int_\Omega \abs{\vec{\nu}(x)}_{\sfA(x)}\,d\vert\nabla\phi\vert(x),
		\]
		where $\gamma$ is the constant defined by \eqref{eq:surfaceTensionCoeff} and $\vec{\nu}$ is the outer unit normal vector to $\partial E$. 
	\end{lemma}
	
	\begin{proof}
		The complementarity condition $\phi \chi_{\cl\Omega \setminus \Omega_0} = 0$ implies $\cl{E} \subset \Omega_0$. We will denote by $d_{\partial E}$ the signed distance function on $\cl\Omega$ to $\partial E$ with the sign convention that $\set{d_{\partial E}\le 0} = \cl{E}$. 
		We introduce 
		\begin{equation} \label{eq:PhiProfile}
			\Phi_n(z;x) \coloneq
			\begin{cases}
				\phi_+		&	\text{if } z < -2,\\
				\phi_+ - \Big(\phi_+ - \omega\big(-\frac{1}{\sqrt{\varepsilon_n}};x,\nabla d_{\partial E}(x) \big)\Big)\big(2+z\big) & \text{if } -2 \le z  < -1, \\
				\omega\big(\frac{z}{\sqrt{\varepsilon_n}};x,\nabla d_{\partial E}(x) \big) 	& \text{if } -1 \le z < 1, \\
				\omega\big(\frac{1}{\sqrt{\varepsilon_n}};x,\nabla d_{\partial E}(x) \big)\big(2 - z\big) & \text{if } 1 \le z  < 2, \\
				0			&	\text{if } z \ge 2,
			\end{cases}
		\end{equation}
		where $\omega(\cdot;x,\nabla d_{\partial E}(x))$ is the optimal profile from Lemma \ref{lemma:optimalProfile}. 
		
		We show the sequence $\seq{\phi_n}$ defined by 
		$
		\phi_n(x) \coloneq \Phi_n\big(\frac{d_{\partial E}(x)}{\sqrt{\varepsilon_n}};x\big)
		$ 
		has the desired properties.
	
		We recall some basic facts about the (signed) distance function; see, e.g., \cite[Chapters 6 and 7]{Delfour.Zolesio_2011_ShapesGeometriesMetrics}. The signed distance function is Lipschitz, so it is differentiable a.e.\ by Rademacher. Also, at points $x\in\Omega$ where $d_{\partial E}$ is differentiable we have $\abs{\nabla d_{\partial E}(x)}=1$, which is compatible with the cut-off scale $0<\delta<1$ for the parameter $p \in \cl{B_1(0)}\setminus B_\delta(0)$ in the definition of $\omega$. In fact, for $\varepsilon_n$ sufficiently small, $d_{\partial E} \in C^2(\set{\abs{d_{\partial E}}<3\sqrt{\varepsilon_n}})$ because $\partial E$ is $C^2$. Together, for large $n$, these imply $\phi_n \in W^{1,\infty}(\Omega) \hookrightarrow H^1(\Omega)$, with the embedding following since $\Omega$ is bounded. Using $0<\omega <\phi_+$, it is straightforward to verify that $\phi_n \to \phi$ in $L^1(\Omega)$. 
		
		We split the integral:
		\begin{align*}
			\scF_{\sfc,\varepsilon_n}[\phi_n] &= \int_\Omega \frac{1}{\varepsilon_n}W_{\ast\sfc}(x,\phi_n(x)) + \frac{\varepsilon_n}{2}\abs{\nabla\phi_n(x)}_{\sfA(x)}^2\,dx 
			= \int_{\set{\abs{d_{\partial E}}/\sqrt{\varepsilon_n} \le 1} } + \int_{\set{1 \le \abs{d_{\partial E}}/\sqrt{\varepsilon_n} \le 2}} + \int_{\set{\abs{d_{\partial E}}/\sqrt{\varepsilon_n} > 2}}.
		\end{align*}
		The last integral is zero since $\phi_n$ is constant on each of $\set{d_{\partial E}/\sqrt{\varepsilon_n} < -2}$ and $\set{d_{\partial E}/\sqrt{\varepsilon_n} \ge 2}$, which causes the interface penalty in $\scF_{\sfc,\varepsilon_n}$ to vanish. Because on this region $\phi_n$ is, in particular, $0$ or $\phi_+$, i.e., one of the well locations, the integral of $W_{\ast\sfc}(\cdot,\phi_n)$  vanishes. 
		
		On the right-hand side, the second integral is negligible as $n\to\infty$. We sketch the argument for the region $\set{1 \le d_{\partial E}/\sqrt{\varepsilon_n} < 2}$, where the rate of convergence of the profile $\omega$ to $0$ may not be exponential (Lemma \ref{lemma:optimalProfile}-\ref{lemma:OP-conv0}). 
		We start with the contribution from $W_{\ast\sfc}(\cdot,\phi_n)$. Let $Z_0>0$ be from Lemma \ref{lemma:optimalProfile}-\ref{lemma:OP-conv0} and $N$ be so large that $\varepsilon_n < 1/Z_0^2$ for all $n\ge N$, so Lemma \ref{lemma:optimalProfile}-\ref{lemma:OP-conv0} gives for any $q_0\ge2$
		\[
		\text{for all } n \ge N \text{ and a.e.\ } x \in \set{1 \le d_{\partial E}/\sqrt{\varepsilon_n} < 2} \qquad 0 < \phi_n(x) \le \omega_0 \le \min\set{1,a,1/R_1}. 
		\]
		Then Lemma \ref{lemma:potentialWstar}-\ref{lemma:Wstar-near0} and \ref{hyp-g:growthNear0-gamma} together give for these $x$ and large $n$ 
		\begin{gather*}
			W_{\ast\sfc}(x,\phi_n(x)) = {W_\ast}(\phi_n(x)) + (\sfc(x) - {\ubar\sfc})g(\phi_n(x)) \le \big(\max_{x\in\cl\Omega}\sfc(x) \big) K_1\phi_n(x)^{q_1}
			\lesssim \left[ c +  \frac{1}{\sqrt{\varepsilon_n}} \right]^{-\frac{2q_1}{q_0-2}},  
		\end{gather*}
		where the last inequality follows from Lemma \ref{lemma:optimalProfile}-\ref{lemma:OP-conv0} in the pessimistic case of $q_0>2$ and $c>0$ is a constant independent of $x$ and $n$. Consequently, 
		\[
		\frac{1}{\varepsilon_n}\int_{\set{\sqrt{\varepsilon_n} \le d_{\partial E} < 2\sqrt{\varepsilon_n}}} W_{\ast\sfc}(x,\phi_n(x)) \,dx \lesssim \left[ c\varepsilon_n^{\frac{q_0 -2}{2q_1}} +  \varepsilon_n^{\frac{q_0 -2 - q_1}{2q_1}} \right]^{-\frac{2q_1}{q_0-2}} \abs{\set{\sqrt{\varepsilon_n} \le d_{\partial E} < 2\sqrt{\varepsilon_n}}},
		\]
		and since the measure of this region is $O(\sqrt{\varepsilon_n})$ it is straightforward to check that the right-hand side vanishes as $n\to\infty$ under the condition $1 + \frac{q_0}{2} \le q_1 \le q_0$ from \ref{hyp-g:growthNear0-gamma}. For the contribution from the interface penalty, we have $\varepsilon_n \nabla \phi_n(x)\cdot \sfA(x)\nabla\phi_n(x) = O(\varepsilon_n)$ for $x\in\set{\sqrt{\varepsilon_n} \le d_{\partial E} < 2\sqrt{\varepsilon_n}}$. After integration over $\set{\sqrt{\varepsilon_n} \le d_{\partial E} < 2\sqrt{\varepsilon_n}}$, its contribution to $\scF_{\sfc,\varepsilon_n}[\phi_n]$ is $O(\varepsilon_n^{3/2})$ and hence negligible as $n\to\infty$, too. 
		
		We examine the region near the interface: $\set{\abs{d_{\partial E}} \le \sqrt{\varepsilon_n}}$. Here, the density of the interface penalty is, by the Cauchy problem \eqref{eq:CauchyPb-optimalProfile} satisfied by $\omega$, 
		\[
		\nabla \phi_n(x) \cdot \sfA(x) \nabla \phi_n(x) =  \frac{2}{\varepsilon_n^2}W_{\ast\sfc}\big(x,\omega\big(\frac{d_{\partial E}(x)}{\varepsilon_n}; x, \nabla d_{\partial E}(x)\big)\big) + O(\frac{1}{\varepsilon_n}).
		\]
		The last term is negligible after multiplication by $\varepsilon_n$ since it is integrated over a domain of measure $O(\sqrt{\varepsilon_n})$. Overall, 
		\[
		\scF_{\sfc,\varepsilon_n}[\phi_n] = \frac{2}{\varepsilon_n}\int_{\set{\abs{d_{\partial E}} \le \sqrt{\varepsilon_n}}} W_{\ast\sfc}\big(x,\omega\big(\frac{d_{\partial E}(x)}{\varepsilon_n}; x, \nabla d_{\partial E}(x)\big)\big)\,dx + o_{n\to\infty}(1). 
		\]
		As is standard, the remaining integral is computed by the  coarea formula \cite[Theorem 2.93]{Ambrosio.Fusco.ea_2000_FunctionsBoundedVariation}. We make two changes of variable. First, we let $z \coloneq \lambda/\varepsilon_n$, so 
		\begin{align*}
			&\frac{2}{\varepsilon_n}\int_{\set{\abs{d_{\partial E}} \le \sqrt{\varepsilon_n}}} W_{\ast\sfc}\big(x,\omega\big(\frac{d_{\partial E}(x)}{\varepsilon_n}; x, \nabla d_{\partial E}(x)\big)\big) \abs{\nabla d_{\partial E}(x)}\,dx  \\
			&= 2\int_{-\frac{1}{\sqrt{\varepsilon_n}}}^{\frac{1}{\sqrt{\varepsilon_n}}}\int_{\set{d_{\partial E} = \varepsilon_n z}} W_{\ast\sfc}(x,\omega(z; x, \nabla d_{\partial E}(x))) \,d\cH^{d-1}(x)\,dz,
		\end{align*}
		Consequently, 
		\[
		\limsup_{n\to\infty} \scF_{\sfc,\varepsilon_n}[\phi_n] \le \limsup_{n\to\infty} 2\int_{-\infty}^\infty \int_{\set{d_{\partial E} = \varepsilon_nz}} W_{\ast\sfc}(x, \omega(z; x, \nabla d_{\partial E}(x)))\,d\cH^{d-1}(x)\,dz.
		\]
		Second, Lemma \ref{lemma:optimalProfile}-\ref{lemma:OP-existence} gives for each $x$ that the optimal profile $z \mapsto \omega(z;x,\nabla d_{\partial E}(x))$ is monotone (decreasing), so we change variables from signed distance to concentration $v \coloneq \omega(z;x,\nabla d_{\partial E}(x))$. Then,  \eqref{eq:CauchyPb-optimalProfile} provides $dv = \partial_z\omega(\cdot;x,\nabla d_{\partial E}(x))\,dz = -\sqrt{2W_{\ast\sfc}(x,v)}/\abs{\nabla d_{\partial E}(x)}_{\sfA(x)}\,dz$. Since 
		$\lim_{z\to+\infty}\omega(z) = 0$ (Lemma \ref{lemma:optimalProfile}-\ref{lemma:OP-conv0}) and 
		$\lim_{z\to-\infty}\omega(z;x,\nabla d_{\partial E}(x)) = \phi_+$  for  $x\in \set{d_{\partial E} \le0} \subset \Omega_0$ (Lemma \ref{lemma:optimalProfile}-\ref{lemma:OP-convPhiPlus}),
	this change of variable yields  
	\begin{align*}
		= \limsup_{n\to\infty}\int_0^{\phi_+} \int_{\set{\omega(d_{\partial E}(\cdot)/\varepsilon_n; \cdot, \nabla d_{\partial E}(\cdot)) = v}} \sqrt{2W_{\ast\sfc}(x,v)} \abs{\nabla d_{\partial E}(x)}_{\sfA(x)}\,d\cH^{d-1}(x) \,dv.
	\end{align*}
	
	The remainder of the proof consists of computing the above limit superior. By computations analogous to those in \cite[Section 3]{Owen.Sternberg_1991_NonconvexVariationalProblems}, we obtain
	\begin{gather*}
		\limsup_{n\to\infty} \scF_{\sfc,\varepsilon_n}[\phi_n] \le \limsup_{n\to\infty}\int_0^{\phi_+} \int_{\set{\omega(d_{\partial E}(\cdot)/\varepsilon_n; \cdot, \nabla d_{\partial E}(\cdot)) = v}} \sqrt{2W_{\ast\sfc}(x,v)} \abs{\nabla d_{\partial E}}_{\sfA(x)} \,d\cH^{d-1}(x)\,dv \\
		= \int_0^{\phi_+} \int_{\partial E} \sqrt{2W_{\ast\sfc}(x,v)} \abs{\vec{\nu}}_{\sfA(x)}\,d\cH^{d-1}(x) \,dv = \gamma \int_\Omega \abs{\vec{\nu}}_{\sfA(x)}\,d\vert\nabla\phi\vert(x).
	\end{gather*}
	The last equality follows from recalling $\abs{\nabla\phi} = \phi_+ \cH^{d-1}\lfloor\partial E$ and $\partial E\subset\Omega_0$, the latter of which yields $W_{\ast\sfc}(x,v) = {W_\ast}(v)$. 
\end{proof}

\subsubsection{Stability of the recovery sequence}

We recall that $\rho$, and not $\phi$, in $\seq{\scG_\varepsilon}$ satisfies a mass constraint. The recovery sequences we constructed for $\seq{\scF_{\sfc,\varepsilon}}$ are stable to small term-by-term translations of the location of the transition layer in the profile $\Phi_n$ (cf. \eqref{eq:PhiProfile}) away or towards $\partial E$. 
\begin{lemma}[stability] \label{lemma:limsup-stability}
	Suppose \ref{hyp-standardAssumptions},  \ref{hyp:compatibilityConditionF&G}, \ref{hyp-c:measure}, and  \ref{hyp-g:growthNear0-gamma}. Let $\seq{\varepsilon_n}$ be a positive sequence converging to zero, $E \subset \cl\Omega$ open with $C^2$ boundary, and $\phi \coloneq \phi_+ \chi_E$ satisfy $\int_\Omega \phi\,dx = \phi_+/\rho_+$ and $\phi \chi_{\cl\Omega \setminus \Omega_0} = 0$ $\cL^d$-a.e. Let $\seq{\phi_n}$ be the sequence $\phi_n(x) \coloneq \Phi_n(\frac{d_{\partial E}(x)} {\sqrt{\varepsilon_n}}; x )$ constructed in Lemma \ref{lemma:limsupFSmoothBdry}, with $\Phi_n$ defined in \eqref{eq:PhiProfile}. For any sequence $\seq{\tau_n} \subset [-2,2]$,
	the sequence $\seq{\phi_n^{\tau_n}}$ defined by $\phi_n^{\tau_n}(x) \coloneq \Phi_n(\frac{d_{\partial E}(x)}{\sqrt{\varepsilon_n}} - \sqrt{\varepsilon_n}\tau_n; x)$ satisfies 
	\[
	\text{for } n \text{ large } \norm{\phi^{\tau_n}_n - \phi_n}_{L^1(\Omega)} = O(\sqrt{\varepsilon_n}\abs{\tau_n}),  \quad \text{and} \quad \limsup_{n\to\infty}\scF_{\sfc,\varepsilon_n}[\phi^{\tau_n}_n] \le \gamma \int_\Omega \abs{\vec{\nu}(x)}_{\sfA(x)} \,d\vert\nabla\phi(x)\vert.
	\]
\end{lemma}
\begin{proof}
	By taking $n$ large enough, we may assume $\sqrt{\varepsilon_n} \abs{\tau_n} < 1/2$. Suppose $\tau_n > 0$, so $\phi^{\tau_n}_n \ge \phi_n$ for all $x$ since, for fixed $x$, the profile $z \mapsto \Phi_n(z;x)$ is nonincreasing (cf.\ \eqref{eq:PhiProfile}). We sketch why the $L^1$-distance between $\phi_n^{\tau_n}$ and $\phi_n$ is $O(\sqrt{\varepsilon_n}\tau_n)$; the case $\tau_n < 0$ is analogous. 
	
	For a fixed $x$ (that we suppress), 
	\begin{equation} \label{eq:translatedProfileDifference}
		\footnotesize
		\begin{gathered}
			0 \le \Phi_n(z - \sqrt{\varepsilon_n}\tau_n) - \Phi_n(z) \\
			= 
			\begin{cases}
				\big(\phi_+ - \omega(-\frac{1}{\sqrt{\varepsilon_n}})\big)(2+z) = O(e^{-1/\sqrt{\varepsilon_n}}\sqrt{\varepsilon_n}\tau_n) & \text{if } -2 \le z < -2 + \sqrt{\varepsilon_n}\tau_n, \\
				\big( \phi_+ - \omega(-\frac{1}{\sqrt{\varepsilon_n}})\big) \sqrt{\varepsilon_n}\tau = O(e^{-1/\sqrt{\varepsilon_n}}\sqrt{\varepsilon_n}\tau_n) & \text{if } -2 + \sqrt{\varepsilon_n}\tau_n \le z < -1, \\
				\phi_+ - (\phi_+ - \omega(-\frac{1}{\sqrt{\varepsilon_n}}))(2 - \sqrt{\varepsilon_n}\tau_n + z) - \omega(\frac{z}{\sqrt{\varepsilon_n}}) = O\big(e^{-1/\sqrt{\varepsilon_n}}( e^{\tau_n} - 1) \big) & \text{if } -1 \le z < -1 + \sqrt{\varepsilon_n}\tau_n, \\
				\omega(\frac{z - \sqrt{\varepsilon_n}\tau_n}{\sqrt{\varepsilon_n}}) - \omega(\frac{z}{\sqrt{\varepsilon_n}}) = O(1) & \text{if } -1 + \sqrt{\varepsilon_n}\tau_n \le z < 1, \\
				\omega(\frac{z-\sqrt{\varepsilon_n}\tau_n}{\sqrt{\varepsilon_n}}) - \omega(\frac{1}{\sqrt{\varepsilon_n}})(2-z) = O(\varepsilon_n^{q_1/(2(q_0-2))}\tau_n) & \text{if } 1 \le z < 1 + \sqrt{\varepsilon_n}\tau_n, \\
				\omega(\frac{1}{\sqrt{\varepsilon_n}})\sqrt{\varepsilon_n}\tau_n = O(\varepsilon_n^{1/(q_0-2)}\sqrt{\varepsilon_n}\tau_n) & \text{if } 1 + \sqrt{\varepsilon_n}\tau_n \le z < 2, \\
				\omega(\frac{1}{\sqrt{\varepsilon_n}}) (2 + \sqrt{\varepsilon_n}\tau_n - z) = O(\varepsilon_n^{1/(q_0-2)}\sqrt{\varepsilon_n}\tau_n) & \text{if } 2 \le z < 2 + \sqrt{\varepsilon_n} \tau_n.
			\end{cases}
		\end{gathered}
	\end{equation}
	\normalsize
	We sketch how to estimate the difference for each of the cases above. For the first two cases, both profiles are close to $\phi_+$ while in the last two cases both profiles are close to $0$, so our estimates on the rate of convergence of $\omega$ provide that the profile and its translation are, respectively, exponentially close to one another (Lemma \ref{lemma:optimalProfile}-\ref{lemma:OP-convPhiPlus}) or (at worst) algebraically close to one another (Lemma \ref{lemma:optimalProfile}-\ref{lemma:OP-conv0}). For the third and fifth cases, the estimates are more delicate, but essentially the same conclusion holds. In the fourth case,  where the profiles are simultaneously near their transition layers, the integral of the difference is $O(\sqrt{\varepsilon_n} \tau_n)$ after changing to the macroscopic variable $z \coloneq d_{\partial E}(x)/\sqrt{\varepsilon_n}$.
		
	We show how to handle the fifth case, which is the most delicate. Let $Z_0$ be from Lemma \ref{lemma:optimalProfile}-\ref{lemma:OP-conv0}, and by  taking $n$ sufficiently large, we may suppose $1/\sqrt{\varepsilon_n} \ge  2Z_0$, which gives us access to the rate of convergence of $\omega$ to $0$ (Lemma \ref{lemma:optimalProfile}-\ref{lemma:OP-conv0}). 
	
	Since $z -\sqrt{\varepsilon_n}\tau_n \ge 1-\sqrt{\varepsilon_n}\tau_n$, the profile $z \mapsto \omega(z)$ is nonincreasing, and $2-z > 1 - \sqrt{\varepsilon_n}\tau$, 
	\[
	\omega(\frac{z - \sqrt{\varepsilon_n}\tau_n}{\sqrt{\varepsilon_n}}) - \omega(\frac{1}{\sqrt{\varepsilon_n}})(2-z) \le \omega(\frac{1-\sqrt{\varepsilon_n}\tau_n}{\sqrt{\varepsilon_n}}) - \omega(\frac{1}{\sqrt{\varepsilon_n}}) + \sqrt{\varepsilon_n}\tau_n\omega(\frac{1}{\sqrt{\varepsilon_n}}).
	\]
	The last term is (at worst) $O(\varepsilon_n^{1/(q_0-2)} \sqrt{\varepsilon_n}\tau_n)$. We estimate the difference using the Cauchy problem for $\omega$. Because $\sqrt{\varepsilon_n} \tau_n < 1/2$ and $1/\sqrt{\varepsilon_n} \ge 2Z_0$, we have $(1-\sqrt{\varepsilon_n} \tau_n)/\sqrt{\varepsilon_n} > Z_0$, so the rate of convergence of $\omega$ to $0$ gives for $\zeta \ge (1-\sqrt{\varepsilon_n}\tau_n)/\sqrt{\varepsilon_n}$ that  $\omega(\zeta)\le\omega_0\le \min\set{1,a,1/R_1}$. Then the growth of ${W_\ast}$ near $0$ (Lemma \ref{lemma:potentialWstar}-\ref{lemma:Wstar-near0}) shows $\sqrt{{W_\ast}(\omega(\zeta))} \lesssim \omega(\zeta)^{q_1/2}$, so the rate of convergence to $0$ provides
	\begin{gather*}
		\omega(\frac{1-\sqrt{\varepsilon_n}\tau_n}{\sqrt{\varepsilon_n}}) - \omega(\frac{1}{\sqrt{\varepsilon_n}}) \lesssim \int_{(1-\sqrt{\varepsilon_n}\tau_n)/\sqrt{\varepsilon_n}}^{1/\sqrt{\varepsilon_n}} \sqrt{{W_\ast}(\omega(\zeta))}\,d\zeta \lesssim \int_{(1-\sqrt{\varepsilon_n}\tau_n)/\sqrt{\varepsilon_n}}^{1/\sqrt{\varepsilon_n}} ( c + \zeta )^{-\frac{q_1}{q_0-2}}\,d\zeta \\
		\le \tau_n (c + \frac{1}{\sqrt{\varepsilon_n}} - \tau_n)^{-q_1/(q_0-2)} \le \tau_n (c + \frac{1}{2\sqrt{\varepsilon_n}})^{-q_1/(q_0-2)}.
	\end{gather*}
	The constant $c>0$ arises from Lemma \ref{lemma:optimalProfile}-\ref{lemma:OP-conv0}; and is uniform over $\zeta$ and large $n$. 
	
	Combining the pointwise estimates we derived for the first three or last three regions in \eqref{eq:translatedProfileDifference} with an analogous calculation using the coarea formula shows the differences $\phi_n^{\tau_n} - \phi_n$ in $L^1$ are  $O(\sqrt{\varepsilon_n}\tau_n)$. Since $\phi_n \to \phi$ in $L^1(\Omega)$, we thus deduce $\phi_n^{\tau_n} \to \phi$ in $L^1(\Omega)$, too.
\end{proof}

\subsubsection{A recovery sequence for $\scG_\varepsilon$} \label{subsubsec:Geps-recoverySequence}
Let $\seq{\varepsilon_n}$ be a positive sequence converging to zero and $(\rho,\phi) \in L^1(\Omega)^2$. If $\scG_0[\rho,\phi] = +\infty$, then we may arbitrarily choose any sequence $\seq{(\rho_n,\phi_n)} \subset L^1(\Omega)^2$ converging to $(\rho,\phi)$ in $L^1(\Omega)$ and be done. Now, suppose $\scG_0[\rho,\phi] <+\infty$, so $\rho \in BV(\Omega; \set{0,\rho_+})$ and $\phi \in BV(\Omega; \set{0,\phi_+})$ and they satisfy $\rho = \frac{\rho_+}{\phi_+}\phi$ a.e., $\phi \chi_{\cl\Omega \setminus \Omega_0} = 0$ a.e., and $\int_\Omega \rho\,dx = 1$. Then, $\phi = \phi_+ \chi_E$ for some set of finite perimeter $E$ with $\abs{E} = 1/\rho_+$. We further assume $E$ is open with $C^2$ boundary. We relax this later.

\medskip
Consider a recovery sequence $\seq{\phi_n}$ for $\phi$ and the functionals $\seq{\scF_{\sfc,\varepsilon_n}}$, let $\seq{\tau_n} \subset [-2,2]$ be a sequence of translations that is to be determined, and let $\seq{\phi_n^{\tau_n}}$ be the term-by-term translations of $\seq{\phi_n}$ as in Lemma \ref{lemma:limsup-stability}. In particular, $\phi_n^{\tau_n} \to \phi$ in $L^1(\Omega)$. Since $f^\ast$ is continuously differentiable, we define 
\begin{equation} \label{eq:rhoLimsup}
	\text{a.e.\ } x \in \Omega \qquad \rho_n(x) \coloneq {f^\ast}'(\phi_n^{\tau_n}(x) - \sfa), 
\end{equation}
where $\sfa$ is the constant from \ref{hyp:compatibilityConditionF&G}. 

\medskip
Each $\rho_n$ is nonnegative as a consequence of the nonnegativity of ${f^\ast}'$. To see the latter, note ${f^\ast}'$ is nondecreasing (as the derivative of a convex function), and because $f^\ast\equiv0$ on $\left]-\infty,0\right]$ (since $f \equiv +\infty$ there) we have ${f^\ast}' \equiv 0$ on $\left]-\infty,0\right]$, too. 

\medskip
We now choose the translations $\seq{\tau_n}$ via a continuity argument so as to enforce the mass constraint $\int_\Omega \rho_n \,dx = 1$. For a fixed $n$, we note if $\tau_n = -2$ then the translation satisfies $\phi_n^{-2} \le \phi$ a.e. Because ${f^\ast}'$ is nondecreasing, we deduce ${f^\ast}'(\phi_n^{-2} - \sfa) \le {f^\ast}'(\phi - \sfa) = \rho$ a.e.\ 
\[
\int_\Omega {f^\ast}'(\phi_n^{-2} - \sfa) \,dx \le \int_\Omega {f^\ast}'(\phi - \sfa) \,dx = \int\rho\,dx = 1.
\]
If $\tau_n = 2$, then $\phi_n^2 \ge \phi$ a.e., so analogous reasoning shows 
\[
\int_\Omega {f^\ast}'(\phi_n^2 - \sfa) \,dx \ge \int_\Omega {f^\ast}'(\phi - \sfa) \,dx = \int_\Omega \rho \,dx = 1. 
\]
For large enough $n$, the map $[-2,2]\ni \tau \mapsto \phi_n^\tau \in L^1(\Omega)$ is Lipschitz and nondecreasing, and the monotonicity and continuity of ${f^\ast}'$ allow to show $[-2,2] \ni \tau \mapsto \int_\Omega {f^\ast}'(\phi_n^\tau - \sfa)\,dx \in [1-\delta,1+\delta]$ is continuous. Intermediate value theorem gives $\tau(\varepsilon_n) \in [-2,2]$ such that $\int_\Omega {f^\ast}'(\phi_n^{\tau(\varepsilon_n)} - \sfa) \,dx = 1$, so we choose $\tau_n \coloneq \tau(\varepsilon_n)$. 

\medskip
We show the $\Gamma$-$\limsup$ inequality. Suppose $n$ is large enough that the mass constraint $\int_\Omega \rho_n\,dx=1$ is satisfied and thus $\scG_{\varepsilon_n}[\rho_n,\phi_n^{\tau_n}]<+\infty$. We appeal to the formula \eqref{eq:GepsModicaMortola} for $\scG_{\varepsilon_n}$:
\[
\scG_{\varepsilon_n}[\rho_n,\phi_n^{\tau_n}] = \scF_{\sfc,\varepsilon_n}[\phi_n^{\tau_n}] + \frac{1}{\varepsilon_n} \int_\Omega f(\rho_n) - \rho_n(\phi_n^{\tau_n} - \sfa) + f^\ast(\phi_n^{\tau_n} - \sfa)\,dx.
\]
Our definition of $\seq{\rho_n}$ in \eqref{eq:rhoLimsup} was chosen such that the integrand on the right-hand side is zero a.e. Consequently, $\scG_{\varepsilon_n}[\rho_n,\phi_n^{\tau_n}] = \scF_{\sfc,\varepsilon_n}[\phi_n^{\tau_n}]$, so by using that $\seq{\phi_n^{\tau_n}}$ is a recovery sequence for $\seq{\scF_{\sfc,\varepsilon_n}}$ (Lemma \ref{lemma:limsup-stability}), we deduce 
\[
\limsup_{n\to\infty} \scG_{\varepsilon_n}[\rho_n,\phi_n^{\tau_n}] = \limsup_{n\to\infty} \scF_{\sfc,\varepsilon_n}[\phi_n^{\tau_n}] \le \scG_0[\rho,\phi].
\]

\medskip
By Remark \ref{rmk:strongConvergence-oneFromTheOther}, 
$\rho_n \to \rho$ in $L^1$ since, after finitely many terms, $\seq{(\rho_n,\phi_n)}$ \ref{hyp:well-preparedG} and $\phi_n \to \phi$ in $L^1$.

\medskip
We show there was no loss of generality to assume $E$ was open with $C^2$ boundary. Suppose $E \subset \Omega_0$ is a set of finite perimeter in $\Omega$. Because $\Omega,\Omega_0$ have  Lipschitz boundaries, there exists \cite[Lemma 1]{Sternberg_1988_EffectSingularPerturbation}, \cite[Remark 3.43]{Ambrosio.Fusco.ea_2000_FunctionsBoundedVariation} a sequence $\seq{E_j}_{j\in\bbN}$ of open sets with smooth boundary in $\bbR^d$ such that for each $j \in \bbN$ we have $\abs{E_j \cap \Omega_0} = \abs{E}$ and 
$
\lim_{j\to\infty}\abs{(E_j \cap \Omega_0) \Delta E} = 0$  and  $\lim_{j\to\infty}\Per(E_j, \Omega) = \Per(E,\Omega).
$
Set $\phi_j \coloneq \phi_+ \chi_{E_j}$ and $\rho_j \coloneq (\rho_+/\phi_+)\phi_j$, so the above means each of $\seq{\rho_j}$, $\seq{\phi_j}$ converge strictly in $BV(\Omega)$ to $\rho,\phi$, resp. We then apply our previous work term-by-term to $\seq{(\rho_j, \phi_j)}$ and obtain for each $j$ a sequence $\seq{(\rho_{j,n}, \phi_{j,n})}_n$ such that 
\[
\lim_{n\to\infty}(\rho_{j,n}, \phi_{j,n}) = (\rho_j, \phi_j)  \text{ in } L^1(\Omega)^2 \quad \text{and} \quad \limsup_{n\to\infty} \scG_{\varepsilon_n}[\rho_{j,n}, \phi_{j,n}] \le  \gamma  \int_\Omega \abs{\vec{\nu}_j}_{\sfA(x)} \,d\vert\nabla\phi_j\vert \eqcolon\scG_0[\rho_j, \phi_j].
\]
The Reshetnyak continuity theorem gives 
$\lim_{j\to\infty} \scG_0[\rho_j,\phi_j] = \scG_0[\rho,\phi]$, so \\
$\limsup_{j\to\infty} \limsup_{n\to\infty} \scG_{\varepsilon_n}[\rho_{j,n}, \phi_{j,n}] \le \scG_0[\rho,\phi]$.  
A diagonalization argument gives the conclusion. 

\section{The limiting continuity equation and phase separation} \label{sec:continuityEq&PhaseSeparation}
The aim of this section is twofold. First, by viewing the first equation in \eqref{eq:PP-PKSeps} as a continuity equation with density $\rho_\varepsilon$ and velocity $\vec{v}_\varepsilon$, we show a subsequence of these solutions converge as $\varepsilon \to 0^+$ to a distributional solution $(\rho,\vec{v})$ of the continuity equation (Proposition \ref{prop:continuityEq}). Second, we show perhaps along a further subsequence that the chemoattractant becomes phase separated (Proposition \ref{prop:phaseSeparation}), and along the same subsequence the density converges to a scalar multiple of limiting chemoattractant, and thus also undergoes phase separation. This section complete the proofs of Theorem \ref{thm:main}-\ref{thm:continuityEq},\ref{thm:phaseSeparation}. We emphasize that this section does \emph{not} make use of the energy convergence hypothesis \ref{hyp:energyConvergenceG}.

\subsection{Convergence of the continuity equation (Proof of Theorem \ref{thm:main}-\ref{thm:continuityEq})} \label{subsec:continuityEq}
We derive, using the energy dissipation inequality, well-preparedness of the initial data \eqref{hyp:well-preparedG},  and coercivity, some classical a priori estimates for gradient flows and distributional solutions of the continuity equation. The convergence of a subsequence of $\seq{(\rho_{\varepsilon_n},  \vec{v}_{\varepsilon_n})}$ is obtained by exhibiting some precompactness for the family of densities and the family of velocities. These are essentially those in \cite[Lemma 4.1]{Mellet_2024_HeleShawFlowSingular}. 

\begin{lemma}
	\label{lemma:aPrioriEst-continuityEq}
	If $(\rho_{\varepsilon_n}, \phi_{\varepsilon_n})$ solves  \eqref{eq:PP-PKSeps} with well-prepared \ref{hyp:well-preparedG} initial data $(\rho_{\varepsilon_n}^{\init}, \phi_{\varepsilon_n}^{\init})$, then 
	\begin{enumerate}[label = (\roman*)]
		\item \label{lemma:boundedEnergies} for all $t \in [0,T]$ and $n \in \bbN$ we have $\scF_{\varepsilon_n}[\phi_{\varepsilon_n}(t)] \le \scG_{\varepsilon_n}[\rho_{\varepsilon_n}(t), \phi_{\varepsilon_n}(t)] \le {\ol\scG}$ and $\int_0^t \scD_{\varepsilon_n}(s)\,ds \le {\ol\scG}$;

		\item \label{lemma:massFluxL2L1} denoting $\vec{j}_{\varepsilon} \coloneq \rho_{\varepsilon}\vec{v}_{\varepsilon}$, we have  $\sup_{n\in\bbN}\norm{\vec{j}_{\varepsilon_n}}_{L^2(0,T; L^1(\Omega;\bbR^d))} \le {\ol \scG}$;

		\item \label{lemma:densityAndFluxBounds} there exists a constant $C({\ol\scG}) > 0$ such that 
		\begin{gather}
			\sup_{n\in\bbN}\norm{\rho}_{L^\infty(0,T; L^m(\Omega))} < C({\ol\scG}), \qquad  \sup_{n\in\bbN}\norm{\vec{j}_{\varepsilon_n}}_{L^2(0,T;L^{\frac{2m}{m+1}}(\Omega))} < C({\ol\scG}), \label{eq:aPEst-MassFlux-weakConv} \\  \norm{\rho_{\varepsilon_n}(t) - \rho_{\varepsilon_n}(s)}_{(W^{1, \frac{2m}{m-1}}(\Omega))^\ast} \le C({\ol\scG})\abs{t-s}^{1/2} \quad \forall \, t,s \in[0,T]. \label{eq:equicontinuityRho}
		\end{gather}
	\end{enumerate}
\end{lemma}
\begin{proof}
	\ref{lemma:boundedEnergies}: The first inequality is a direct consequence of the definition of $\scF_\varepsilon$ (recall \eqref{eq:GepsModicaMortola} and \eqref{eq:ModicaMortolaF}). The second inequality follows from the energy dissipation inequality \eqref{eq:GDissipationIneq}, the nonnegativity of the dissipation term $\int_0^t \scD_\varepsilon(s)\,ds$, and the well-preparedness of the initial data. The bound on the dissipation term follows again by the energy dissipation inequality and the bound on $\scG_\varepsilon$. 
	
	\medskip
	\ref{lemma:massFluxL2L1}: Rewriting the dissipation from the continuity equation using the mass flux  $\vec{j}_\varepsilon \coloneq \rho_\varepsilon \vec{v}_\varepsilon$ gives 
	\[
	\int_0^T\!\!\!\!\int_\Omega \rho_{\varepsilon_n} \abs{\vec{v}_{\varepsilon_n}}^2 \,dx\,dt = \int_0^T\!\!\!\!\int_\Omega \frac{\abs{\vec{j}_{\varepsilon_n}}^2}{\rho_{\varepsilon_n}}\,dx \,dt.
	\]
	Then Cauchy-Schwarz and that $\rho_\varepsilon(t) \in \cP_\ac(\Omega)$ for all $t \ge 0$ give
	\[
	\int_\Omega \abs{\vec{j}_{\varepsilon_n}(t)} \,dx \le \left( \int_\Omega \rho_{\varepsilon_n}(t) \,dx\right)^{1/2} \left( \int_\Omega \frac{\abs{\vec{j}_{\varepsilon_n}(t)}^2}{\rho_{\varepsilon_n}(t)} \,dx \right)^{1/2} = \left( \int_\Omega \frac{\abs{\vec{j}_{\varepsilon_n}(t)}^2}{\rho_{\varepsilon_n}(t)} \,dx \right)^{1/2},
	\]
	so squaring, integrating over $t \in [0,T]$, and the energy dissipation inequality \eqref{eq:GDissipationIneq} provide
	\[
	\int_0^T \norm{\vec{j}_{\varepsilon_n}(t)}_{L^1(\Omega)}^2 \,dt \le \int_0^T\!\!\!\!\int_\Omega \frac{\abs{\vec{j}_{\varepsilon_n}}^2}{\rho_{\varepsilon_n}} \,dx\,dt \le \int_0^T \scD_{\varepsilon_n}(t) \,dt \le  \ol\scG.
	\]
	
	\medskip 
	\ref{lemma:densityAndFluxBounds}: For the bound on the density, fix $t \in[0,T]$ and then \ref{lemma:boundedEnergies} shows the sequence $\scG_{\varepsilon_n}[\rho_{\varepsilon_n}(t), \phi_{\varepsilon_n}(t)]$ is bounded uniformly-in $t$-and-$n$, so Lemma \ref{lemma:coercivity}-\ref{lemma:coercivity-aprioriEst} gives 
	\[
	\norm{\rho_{\varepsilon_n}(t)}_{L^m(\Omega)} \le C(\ol\scG\varepsilon_n)^{1/m} + C\qquad	\text{for all } \, n \in \bbN \text{ and } t \in [0,T].
	\]
	Since $\seq{\varepsilon_n}_n$ is bounded and the upper bound is independent of $t$, the estimate follows.
	
	For the bound on the mass flux, fix $t \ge 0$, so H\"older's inequality with $p = m+1$ and $p' = \frac{m+1}{m}$ gives
	\begin{align*}
		\int_\Omega \abs{\vec{j}_{\varepsilon_n}(t)}^{2/p'} \,dx = \int_\Omega \abs{\rho_{\varepsilon_n}(t)}^{1/p'} \left(\frac{\abs{\vec{j}_{\varepsilon_n}(t)}^2}{\rho_{\varepsilon_n}(t)}\right)^{1/p'} \,dx 
		\le \left( \int_\Omega \abs{\rho_{\varepsilon_n}(t)}^m\,dx \right)^{1/p} \left(\int_\Omega \frac{\abs{\vec{j}_{\varepsilon_n}(t)}^2}{\rho_{\varepsilon_n}(t)}\,dx \right)^{1/p'},
	\end{align*}
	and then the uniform $L^\infty_tL^m_x$-bound on the density gives
	\[
	\norm{\vec{j}_{\varepsilon_n}(t)}_{L^{\frac{2m}{m+1}}(\Omega)}^2 \le C \int_\Omega \frac{\abs{\vec{j}_{\varepsilon_n}(t)}^2}{\rho_{\varepsilon_n}(t)}\,dx.
	\]
	Integration over $t \in [0,T]$ and the energy dissipation inequality \eqref{eq:GDissipationIneq} provide the estimate.
	
	For the uniform equicontinuity \eqref{eq:equicontinuityRho}, we first note that by density of $C^1(\cl\Omega)$ in $W^{1,\frac{2m}{m-1}}(\Omega)$ and linearity, the weak formulation of the continuity equation is satisfied for test functions in $C^1_c(\left[0,T\right[; W^{1,\frac{2m}{m-1}}(\Omega))$. Moreover, since $(\rho_{\varepsilon_n},\vec{v}_{\varepsilon_n})$ satisfy the weak formulation on the time interval $[0,T]$, they also satisfy it over the time intervals $[0,s] \subset [0,t] \subset [0,T]$. Thus, given a test function $\zeta \in  W^{1,\frac{2m}{m-1}}(\Omega)$, we take the difference of the weak formulations over the time intervals $[0,t]$ and $[0,s]$
	\[
	\int_\Omega (\rho_{\varepsilon_n}(t) - \rho_{\varepsilon_n}(s))\zeta \,dx = \int_s^t \int_\Omega \rho_{\varepsilon_n} \vec{v}_{\varepsilon_n} \cdot \nabla \zeta \,dx \,d\tau. 
	\]  
	Then H\"older's inequality with $p=\frac{2m}{m+1}$ and  $p'=\frac{2m}{m-1}$ and the uniform $L^2_tL^{\frac{2m}{m+1}}_x$-bound on the flux gives 
	\begin{gather*}
		\abs{\int_\Omega (\rho_{\varepsilon_n}(t) - \rho_{\varepsilon_n}(s))\zeta \,dx } \le \int_s^t \int_\Omega \abs{\vec{j}_{\varepsilon_n}} \abs{\nabla \zeta} \,dx \,d\tau \le \left( \int_s^t \norm{\vec{j}_{\varepsilon_n}(\tau)}_{L^{\frac{2m}{m+1}}_x} \,d\tau \right) \norm{\nabla\zeta}_{L^{\frac{2m}{m-1}}(\Omega)}   \\
		\le \big(\sqrt{t-s} \sup_{n\in\bbN}\norm{\vec{j}_{\varepsilon_n}}_{L^2(0,T; L^{\frac{2m}{m+1}}(\Omega))}\big)\norm{\zeta}_{W^{1,\frac{2m}{m-1}}(\Omega)},
	\end{gather*}
	which implies the estimate.
\end{proof}

We show there is a limit point $(\rho,\vec{j})$ of $\seq{(\rho_{\varepsilon_n}, \vec{j}_{\varepsilon_n})}$ that solves the continuity equation.
\begin{proposition}[limiting continuity equation] \label{prop:continuityEq}
	Assume the hypotheses of the previous lemma.  
	\begin{enumerate}[label=(\roman*)]
		\item \label{prop:continuityEq-convergence} (compactness) There exists a (not relabeled) subsequence $\seq{\varepsilon_n}$, 
		\begin{equation}
			\begin{gathered}
				\rho \in L^\infty(0,T; L^m(\Omega)) \cap C([0,T]; (W^{1,\frac{2m}{m-1}}(\Omega))^\ast) 
				\text{ and }   \vec{j} \in L^2(0,T; L^{\frac{2m}{m+1}}(\Omega; \bbR^d)) 
			\end{gathered}
		\end{equation}
		such that $\rho_{\varepsilon_n} \to \rho$ strongly in $C([0,T]; (W^{1,\frac{2m}{m-1}}(\Omega))^\ast)$ and weakly-$\ast$ in $L^\infty(0,T; L^m(\Omega))$, and $\vec{j}_{\varepsilon_n} \rightharpoonup \vec{j}$ weakly in $L^2(0,T; L^{\frac{2m}{m+1}}(\Omega)^d)$.
		
		\item \label{prop:continuityEq-velocity} There exists $\vec{v} \in L^2(\left]0,T\right[ \times \Omega, \rho\,dx dt; \bbR^d)$ such that $\vec{j} = \rho \vec{v}$ a.e.\ in $\left]0,T\right[ \times \Omega$ and the pair $(\rho,\vec{v})$ satisfies the distributional formulation of the continuity equation \eqref{eq:limitContinuityEquation}.
	\end{enumerate}
\end{proposition}

\begin{proof}
	(i): The weak convergences follows immediately from the uniform bounds in Lemma \ref{lemma:aPrioriEst-continuityEq}-\ref{lemma:densityAndFluxBounds}. The strong convergence is a consequence of Arzel{\`a}-Ascoli. Indeed, since we work on a bounded domain, we have the compact embedding $W^{1,\frac{2m}{m-1}}(\Omega)\hookrightarrow\hookrightarrow L^\frac{m}{m-1}(\Omega)$, so Schauder provides the compact embedding $L^m(\Omega) \hookrightarrow\hookrightarrow (W^{1,\frac{2m}{m-1}}(\Omega))^\ast$, and then the uniform $L^\infty_t L^m_x$-bound on $\seq{\rho_{\varepsilon_n}}$ implies the sequence is bounded in $C([0,T]; (W^{1,\frac{2m}{m-1}}(\Omega))^\ast)$ and for each $t$ that $\seq{\rho_{\varepsilon_n}(t)}_n$ is precompact in $(W^{1,\frac{2m}{m-1}}(\Omega))^\ast$. The requisite uniform equicontinuity also from Lemma \ref{lemma:aPrioriEst-continuityEq}-\ref{lemma:densityAndFluxBounds} allows us to conclude.
	
	\medskip
	(ii): After replacing $\rho_{\varepsilon_n}\vec{v}_{\varepsilon_n}$ with $\vec{j}_{\varepsilon_n}$, the distributional formulation of the continuity equation reads 
	\[
	\int_\Omega \rho_{\varepsilon_n}^{\init}\zeta(0) \,dx + \int_0^T\!\!\!\!\int_\Omega \rho_{\varepsilon_n} \partial_t\zeta \,dx\,d\tau + \int_0^T\!\!\!\!\int_\Omega \vec{j}_{\varepsilon_n} \cdot \nabla \zeta \,dx\,d\tau = 0 \qquad \forall \, \zeta \in C^1_c(\left[0,T\right[ \times \cl\Omega).
	\]
	The strong convergence of $\seq{\rho_{\varepsilon_n}}$ in $C([0,T]; (W^{1,\frac{2m}{m-1}}(\Omega))^\ast)$ permits us to pass to the limit in the first term, and the fact that the limit is in $L^\infty(0,T; L^m(\Omega))$ allows us to represent this term again as an integral. The weak convergences from (i) allow to pass to the limit in the second and third terms. 
	
	That the limiting mass flux $\vec{j}$ can be written as $\vec{j} = \rho \vec{v}$ is a classical consequence of the functional 
	\begin{gather*}
		E \colon \cM(\left]0,T\right[\times\Omega; \bbR^d) \times \cM(\left]0,T\right[\times\Omega) \to [0,+\infty], \\ E(\vec{\nu},\mu) \coloneq \begin{cases}
			\displaystyle \int_0^T\!\!\!\!\int_\Omega \abs{\frac{d\vec{\nu}}{d\mu dt}}^2 \,d\mu\,dt & \text{if } d\vec{\nu} \ll d\mu\,dt, \\
			+\infty & \text{otherwise}
		\end{cases}
	\end{gather*}
	being lower semicontinuous for joint weak-$\ast$ convergence of Radon measures \cite[Theorem 2.34]{Ambrosio.Fusco.ea_2000_FunctionsBoundedVariation}, which follows from the convexity and superlinearity of  $\bbR^d \ni p \mapsto \abs{p}^2$. The weak convergences from (i) imply the joint weak-$\ast$ convergence, so combining this lower semicontinuity with the fact that the contribution to the dissipation from the continuity equation is bounded
	\[
	E(\vec{j},\rho) \le \liminf_{n\to\infty} \int_0^T\!\!\!\!\int_\Omega \abs{\vec{v}_{\varepsilon_n}}^2\,d\rho_{\varepsilon_n}\,dt \le\liminf_{n\to\infty}\int_0^T \scD_{\varepsilon_n}(t)\,dt \le \sup_{n\in\bbN} \scG_{\varepsilon_n}[\rho_{\varepsilon_n}^{\init}, \phi_{\varepsilon_n}^{\init}] < +\infty 
	\]
	we deduce the existence of $\vec{v} \in L^2(\left]0,T\right[ \times \Omega, d\rho\,dt; \bbR^d)$ such that $\vec{j}\,dx\,dt = d\vec{j} = \vec{v}\, d\rho\, dt = \vec{v} \rho \,dx\,dt$ as $\bbR^d$-valued Radon measures on $\left]0,T\right[ \times \Omega$. Thus, $\vec{j} = \vec{v}\rho$ a.e.\ on $\left]0,T\right[\times\Omega$. 

\end{proof}

\subsection{Phase separation of the density \texorpdfstring{$\rho_\varepsilon$}{\textrho\_\textepsilon} and chemoattractant \texorpdfstring{$\phi_\varepsilon$}{ϕ\_\textepsilon} (Proof of Theorem \ref{thm:main}-\ref{thm:phaseSeparation})} \label{subsec:phaseSeparation}
The convergence of the continuity equation did not directly depend on the contribution of the chemoattractant $\phi_{\varepsilon}$ to the velocity $\vec{v}_\varepsilon$, but only on the uniform-in-$\varepsilon$ bound on the flux $\vec{j}_\varepsilon$ through the dissipation. For phase separation, we will need a further  estimate on the chemoattractant.
\begin{lemma}
	\label{lemma:aPrioriEst-phaseSeparation}
	Let $\seq{(\rho_{\varepsilon_n}^{\init}, \phi_{\varepsilon_n}^{\init})}$ be well-prepared \ref{hyp:well-preparedG} and let $(\rho_{\varepsilon_n}, \phi_{\varepsilon_n})$ solve \eqref{eq:PP-PKSeps} with initial data $(\rho_{\varepsilon_n}^{\init}, \phi_{\varepsilon_n}^{\init})$. There exist a constant $C({\ol\scG}) > 0$ that  $\sup_{n\in\bbN}\norm{\phi_{\varepsilon_n}}_{L^\infty(0,T; L^q(\Omega))} \le C({\ubar\sfc}, {\ol\scG})$.
\end{lemma}

\begin{proof}

	The energy dissipation inequality and well-preparedness give $\scG_{\varepsilon_n}[\rho_{\varepsilon_n}(t),\phi_{\varepsilon_n}(t)] \le {\ol\scG}$ uniformly in $n$ and $t$, so Lemma \ref{lemma:coercivity}-\ref{lemma:coercivity-aprioriEst} provides
	\[
	\norm{\phi_{\varepsilon_n}(t)}_{L^q(\Omega)} \le C \big({\ol\scG}\sup_n \varepsilon_n\big)^{1/q} + C \qquad \text{for all } n\in\bbN \text{ and } t \in [0,T].
	\] 
	The upper bound is independent of $t$, so the estimate follows.
	
\end{proof}

We show the limiting density $\rho$ and chemoattractant $\phi$ are phase separated. 
Demonstrating this requires to show some strong $L^1$-convergence of either the density or chemoattractant. The strategy is similar to that in proof of the compactness result (Proposition \ref{prop:compactness}). The Modica-Mortola functional $\scF_\varepsilon$ controlled by $\scG_\varepsilon$ naturally bounds in $BV(\Omega)$ the auxiliary functions $\seq{\psi_{\varepsilon_n}}$ at each fixed time, where $\psi_{\varepsilon_n}(t,x) \coloneq F(\phi_{\varepsilon_n}(t,x))$ and $F$ was defined in \eqref{eq:auxF} (see also Lemma \ref{lemma:transformationF}). The energy dissipation inequality satisfied by $\scG_\varepsilon$ \eqref{eq:GDissipationIneq} is then responsible for the control over time, which allows to derive strong convergence of $\seq{\phi_{\varepsilon_n}}$. The coupling of $(\rho_{\varepsilon_n}, \phi_{\varepsilon_n})$ through the penalty $\varepsilon^{-1}\int_\Omega [{\ubar\sfc} g]^\ast(\rho_{\varepsilon}) - \rho_{\varepsilon_n}\phi_{\varepsilon_n} + [{\ubar\sfc} g](\phi_{\varepsilon_n}) \,dx$ in $\scG_{\varepsilon_n}$ allows to deduce the $L^1$ convergence of $\seq{\rho_{\varepsilon_n}}$ from that of $\seq{\phi_{\varepsilon_n}}$.

\begin{proposition}[phase separation] \label{prop:phaseSeparation}
	Let $\seq{\varepsilon_n}$ be a positive sequence converging to zero and $\seq{(\rho_{\varepsilon_n}, \phi_{\varepsilon_n})}$ be weak solutions of \eqref{eq:PP-PKSeps} with well-prepared initial data $\seq{(\rho_{\varepsilon_n}^\init, \phi_{\varepsilon_n}^\init)}$ \ref{hyp:well-preparedG} and $\sfc$ satisfying \ref{hyp-c:measure}. Let $\psi_{\varepsilon_n}(t,x) \coloneq F(\phi_{\varepsilon_n}(t,x))$ with $F \colon \bbR \to \left[0,+\infty\right[$ from Lemma \ref{lemma:transformationF}. Then there exists 
	\[
	\phi \in C^{\frac{1}{2}}([0,T]; L^1(\Omega)) \cap C^{\frac{1}{2}-\theta}([0,T];L^p(\Omega)) \cap BV(\left]0,T\right[\times\Omega) \cap L^\infty(0,T; BV(\Omega))
	\]
	for all arbitrarily small $\theta>0$ and any $1 \le p < +\infty$ such that, up to a subsequence,  
	\begin{gather}
		\psi_{\varepsilon_n} \xrightarrow{n\to\infty} \gamma\phi \text{ strongly in } C([0,T]; L^r(\Omega)) \text{ for all } 1 \le r < q; \\
		\rho_{\varepsilon_n} \xrightarrow{n\to\infty} \rho \text{ strongly in } L^\infty(0,T; L^s(\Omega)) \text{ for all } 1 \le s < m \text{ and weakly-}\ast \text{ in } L^\infty(0,T;L^m(\Omega));\\
		\text{and } \phi_{\varepsilon_n} \xrightarrow{n\to\infty} \phi \text{ strongly in } C([0,T]; L^r(\Omega)) \text{ for all } 1 \le r < q \text{ and weakly-}\ast \text{ in } L^\infty(0,T;L^q(\Omega)).
	\end{gather}
	Furthermore, there exists a family $\seq{E(t)}_{t \in [0,T]}$ of measurable subsets of $\cl\Omega$ satisfying
	\begin{gather}
		\forall \, t\in[0,T] \qquad \rho(t) = \frac{\rho_+}{\phi_+}\phi(t) = \rho_+ \chi_{E(t)} \text{ a.e.\ in } \Omega, \quad \abs{E(t)} = \frac{1}{\rho_+}, \label{eq:identityAndVolume}\\
		\chi_{\cl\Omega \setminus \Omega_0} \chi_{E(t)} = 0\text{ a.e.\ in }\Omega, \quad\text{and} \quad  \sup_{t\in[0,T]}\Per(E(t), \Omega) < +\infty. \label{eq:complementarityAndPerimeter}
	\end{gather}
\end{proposition}
We recall \cite[Proposition 12.35]{Maggi_2012_SetsFinitePetimeter}  that sets of finite perimeter with finite volume satisfy the isoperimetric inequality $\Per(E) \ge d\omega_d^{1/d} \abs{E}^{(d-1)/d}$, where $\omega_d$ is the volume of the Euclidean unit ball in $\bbR^d$ ($d\ge2$). Since our sets $\seq{E(t)}_t$ have a finite, positive, constant volume \eqref{eq:identityAndVolume}, their perimeter is also bounded away from zero uniformly in time.

\begin{proof}
	We show some precompactness for $\seq{\psi_{\varepsilon_n}}$ in $C([0,T];L^1(\Omega))$ using Arzel\`{a}-Ascoli. As part of this aim, we show that $\seq{\psi_{\varepsilon_n}}$ is bounded in $L^\infty(0,T; BV(\Omega))$ and $BV(\left]0,T\right[ \times \Omega)$. 
	
	At a fixed $t\in[0,T]$, we use that $\psi_{\varepsilon_n} \coloneq F \circ \phi_{\varepsilon_n}$ for Lipschitz $F$ (Lemma \ref{lemma:transformationF}), so the uniform-in-$\varepsilon_n$ $L^\infty_tL^q_x$-bound on the chemoattractant (Lemma \ref{lemma:aPrioriEst-phaseSeparation}) gives
	\[
	\norm{\psi_{\varepsilon_n}(t)}_{L^q(\Omega)} \le \norm{F'}_{L^\infty(\bbR)}\norm{\phi_{\varepsilon_n}(t)}_{L^q(\Omega)} \le \norm{F'}_{L^\infty(\bbR)}\sup_{n\in\bbN}\norm{\phi_{\varepsilon_n}}_{L^\infty(0,T;L^q(\Omega))}.
	\]
	Then $\seq{\psi_{\varepsilon_n}}$ is bounded in $L^\infty(0,T;L^q(\Omega))$ and thus also in $L^\infty(0,T;L^1(\Omega))$.
	
	Keeping time suppressed, we estimate $\seq{\nabla\psi_{\varepsilon_n}}$ in $L^1(\Omega)^d$. The uniform ellipticity of $\sfA$ \ref{hyp-on-A:ellipticity}, the chain rule, and Young's inequality, give 
	\begin{align*}
		\abs{\nabla \psi_{\varepsilon_n}(x)} \le \frac{1}{\sqrt{\ubar{\sfA}}}\abs{\nabla\psi_{\varepsilon_n}(x)}_{\sfA(x)} = \abs{\sqrt{2W_\ast(\phi_{\varepsilon_n}(x))} \nabla\phi_{\varepsilon_n}(x)}_{\sfA(x)} 
		\le  \frac{1}{\varepsilon_n}W_\ast(\phi_{\varepsilon_n}(x)) + \frac{\varepsilon_n}{2}\abs{\nabla\phi_{\varepsilon_n}(x)}_{\sfA(x)}^2.
	\end{align*}
	After an integration over $x \in \Omega$, the second term is $\scF_{\varepsilon_n}[\phi_{\varepsilon_n}(t)]$, which is bounded uniformly in $t$ and $n$ by ${\ol\scG}$ due to Lemma \ref{lemma:aPrioriEst-continuityEq}-\ref{lemma:boundedEnergies}. Thus, $\seq{\nabla\psi_{\varepsilon_n}}$ is bounded in $L^\infty(0,T; L^1(\Omega)^d)$, which completes the proof that $\seq{\psi_{\varepsilon_n}}$ is bounded in  $L^\infty(0,T; BV(\Omega))$.
	
	To additionally show $\seq{\psi_{\varepsilon_n}}$ is bounded in $BV(\left]0,T\right[\times\Omega)$, it remains to estimate $\seq{\partial_t\psi_{\varepsilon_n}}$. 
	Cauchy-Schwarz, the inequality $\scF_{\varepsilon_n}[\cdot] \le \scG_{\varepsilon_n}[\rho_{\varepsilon_n}(t),\cdot]$, and the energy dissipation inequality \eqref{eq:GDissipationIneq} give an $L^2(0,T; L^1(\Omega))$-bound on $\seq{\partial_t\psi_{\varepsilon_n} }$:
	\begin{gather*}
		\int_0^T \left(\int_\Omega \abs{\partial_t \psi_{\varepsilon_n}} \,dx\right)^2dt = \int_0^T \left(\int_\Omega \sqrt{2W_\ast(\phi_{\varepsilon_n})} \abs{\partial_t\phi_{\varepsilon_n}}\,dx\right)^2 dt \\
		\le \int_0^T \left(\frac{2}{\varepsilon_n}\int_\Omega W_\ast(\phi_{\varepsilon_n}) \,dx\right) \left(\int_\Omega \varepsilon_n \big(\partial_t \phi_{\varepsilon_n} \big)^2\,dx\right) \,dt \le 2 \int_0^T \scF_{\varepsilon_n}[\phi_{\varepsilon_n}(t)] \scD_{\varepsilon_n}(t)\,dt \le 2{\ol\scG}^2.
	\end{gather*}
	In view of the previous $L^\infty_tL^1_x$-bounds, we conclude $\seq{\psi_{\varepsilon_n}}$ is bounded in  $BV(\left]0,T\right[ \times \Omega)$. 
	
	The above bounds also show $\seq{\psi_{\varepsilon_n}}$ is bounded in $H^1(0,T;L^1(\Omega)) \hookrightarrow C^{\frac{1}{2}}([0,T]; L^1(\Omega))$. Identifying each term with its H\"older continuous representative, the $L^\infty(0,T;BV(\Omega))$-bound implies  $\seq{\psi_{\varepsilon_n}(t)}_{(t,n)\in[0,T]\times\bbN}$ is bounded in $BV(\Omega)$, hence precompact in $L^1(\Omega)$. Then Arzel\`{a}-Ascoli provides $\seq{\psi_{\varepsilon_n}}$ is precompact in $C([0,T]; L^1(\Omega))$. We thus extract a subsequence that converges to, say, $\gamma \phi \in  C([0,T]; L^1(\Omega))$, where $\gamma$ is defined in \eqref{eq:surfaceTensionCoeff}. The uniform bound on the H\"older constants implies the H\"older continuity passes to the limit, so $\phi \in C^{\frac{1}{2}}([0,T]; L^1(\Omega))$. The lower semicontinuity of the $BV(\Omega)$-seminorm and $L^\infty(0,T;BV(\Omega))$-bound (resp.\ $BV(\left]0,T\right[\times\Omega)$-seminorm and $BV(\left]0,T\right[\times\Omega)$-bound) implies $\phi \in L^\infty(0,T; BV(\Omega))$ (resp. $\phi \in BV(\left]0,T\right[\times\Omega)$).
	
	\medskip	
	As in the proof of Proposition \ref{prop:compactness}, we can transfer this convergence to $\seq{\phi_{\varepsilon_n}}$. By fixing $t\in[0,T]$ and applying the same argument, for any small $\delta>0$ there exists a constant $C_\delta > 0$ (independent of $n,t$ and depending only on  $\delta,F,W_\ast$) for which
	\[
	\norm{\psi_{\varepsilon_n}(t) - \gamma\phi_{\varepsilon_n}(t)}_{L^q(\Omega)} \le C\abs{\Omega}^{1/q'}\delta + C_\delta \varepsilon_n \scF_{\varepsilon_n}[\phi_{\varepsilon_n}(t)] \le C\delta + C_\delta {\ol\scG} \varepsilon_n.
	\]
	The right-hand side is independent of $t$, so taking the  supremum over $t$, sending $n\to\infty$, and then $\delta\to0^+$, we obtain
	\[
	\limsup_{n\to\infty}\norm{\psi_{\varepsilon_n} - \gamma\phi_{\varepsilon_n} }_{L^\infty(0,T; L^q(\Omega))} = 0.
	\]
	Thus, the strong convergence of $\seq{\psi_{\varepsilon_n}}$ implies $\phi_{\varepsilon_n} \to \phi$ in $C([0,T]; L^1(\Omega))$. This implies for each $t \in [0,T]$ that $\phi_{\varepsilon_n}(t) \to \phi(t)$ strongly in $L^1(\Omega)$ and (up to another subsequence) a.e.\ in $\left]0,T\right[\times\Omega$.
	
	\medskip
	
	Using this strong convergence, we can establish phase separation of the chemoattractant. By continuity of $ W_\ast$, we note for each $t \in [0,T]$ that $W_\ast\circ\phi_{\varepsilon_n}(t) \to W_\ast\circ\phi(t)$ a.e.\ in $\Omega$. Then nonnegativity of $W_\ast$, Fatou's lemma, and that $\scF_{\varepsilon_n}[\phi_{\varepsilon_n}(t)] \le \scG_{\varepsilon_n}[\rho_{\varepsilon_n}(t), \phi_{\varepsilon_n}(t)] \le {\ol\scG}$ uniformly in $n$ and $t$ give
	\[
	\int_\Omega W_\ast(\phi(t))\,dx \le \liminf_{n\to\infty} \int_\Omega W_\ast(\phi_{\varepsilon_n}(t)) \,dx \le \liminf_{n\to\infty} \big( \varepsilon_n \scF_{\varepsilon_n}[\rho_{\varepsilon_n}(t), \phi_{\varepsilon_n}(t)] \big)\le {\ol\scG} \liminf_{n\to\infty} \varepsilon_n = 0.
	\]
	Thus, $W_\ast(\phi(t)) = 0$ a.e.\ in $\Omega$, and since $W_\ast$ vanishes only on $\set{0,\phi_+}$, there exists a measurable set $E(t) \subset \Omega$ such that $\phi(t)=\phi_+\chi_{E(t)}$ a.e.\ in $\Omega$. Since $\phi(t) \in BV(\Omega)$, this means $E(t)$ has finite perimeter in $\Omega$. Since $\phi \in L^\infty(0,T;BV(\Omega))$, this means the perimeter is uniformly bounded: $\sup_{t\in[0,T]} \Per(E(t),\Omega)<+\infty$.
	
	Now that we know for each $t\in[0,T]$ that $\phi = \phi_+\chi_{E(t)}$, we have $\phi \in L^\infty(0,T; L^\infty(\Omega))$, so by interpolating $L^p(\Omega)$-norms between $1$ and arbitrarily large $p_2$, we have for each $t_2,t_1\in[0,T]$ and $p = (1-\theta)\cdot1 + \theta p_2$ with $\theta \in \left]0,1\right[$
	\begin{align*}
		\norm{\phi(t_2) - \phi(t_1)}_{L^p(\Omega)} \le \norm{\phi(t_2)-\phi(t_1)}_{L^1(\Omega)}^{1-\theta}\norm{\phi(t_2) - \phi(t_1)}_{L^{p_2}(\Omega)}^\theta 
		\le \abs{\phi}_{C^{\frac{1}{2}}}^{1-\theta} \abs{\Omega}^{\frac{\theta}{p_2}} (2\phi_+)^{\theta} \abs{t_2-t_1}^{\frac{1-\theta}{2}}.
	\end{align*}
	The second inequality follows from $\phi \in C^{\frac{1}{2}}([0,T];L^1(\Omega))$ and the $L^\infty(0,T;L^\infty(\Omega))$-bound on $\phi$.
	
	\medskip
	
	That the limit is concentrated on $\Omega_0$ uses analogous reasoning. For fixed $t\in[0,T]$, continuity of $g$ gives $g(\phi_{\varepsilon_n}(t)) \to g(\phi(t))$ a.e.\ in $\Omega$, so nonnegativity of $(\sfc - {\ubar\sfc})g(\phi(t))$,  Fatou's lemma, energy dissipation, and well-preparedness give
	\begin{align*}
		0 \le \int_\Omega (\sfc - {\ubar\sfc})g(\phi(t))\,dx \le \liminf_{n\to\infty} \int_\Omega (\sfc - {\ubar\sfc}) g(\phi_{\varepsilon_n}(t))\,dx 
		\le \liminf_{n\to\infty}\big( \varepsilon_n\scG_{\varepsilon_n}[\rho_{\varepsilon_n}(t),\phi_{\varepsilon_n}]\big) \le {\ol\scG}\liminf_{n\to\infty}\varepsilon_n = 0.
	\end{align*}
	Then nonnegativity of the integrand on the left-hand side provides $(\sfc - {\ubar\sfc})\chi_{E(t)} = 0$ a.e.\ since $g(\phi(t)) = g(\phi_+)\chi_{E(t)}$ with $g(\phi_+)>0$. Hence, for each $t\in[0,T]$ we have  $\chi_{\cl\Omega \setminus \Omega_0}\chi_{E(t)} = 0$ a.e.\ in $\Omega$.
	
	\medskip
	From Lemma \ref{lemma:aPrioriEst-continuityEq}-\ref{lemma:densityAndFluxBounds}, up to a subsequence we have 
	\begin{equation} \label{eq:densityWeakStarConvergence}
		\rho_{\varepsilon_n} \to \rho \text{ weakly-}\ast \text{ in } L^\infty(0,T;L^m(\Omega)),
	\end{equation} 
	so the nonnegativity and unit mass conditions of $\seq{\rho_{\varepsilon_n}}$ pass to the limit.

	\medskip
	We now turn to showing $\rho = (\rho_+/\phi_+)\phi$ a.e.\ in $\left]0,T\right[\times\Omega$. It suffices to show
	\begin{equation} \label{eq:FenchelInequality-Vanishing-timeDep}
		\int_0^T\!\!\!\!\int_\Omega[{\ubar\sfc}\,g]^\ast(\rho) - \rho\phi + {\ubar\sfc}\,g(\phi) \,dx \,dt = 0.
	\end{equation}
	since the integrand is nonnegative.  The energy dissipation inequality and well-preparedness provide
	\begin{equation} \label{eq:FenchelVanishing-LinftyL1}
		\lim_{n\to\infty} \big( [{\ubar\sfc}\,g]^\ast(\rho_n) - \rho_n\phi_n + {\ubar\sfc}\,g(\phi_n)  \big) = 0 \text{ strongly in } L^\infty(0,T;L^1(\Omega)).
	\end{equation}
	Let $r = (1-\theta)\cdot 1 + \theta q$ with $\theta \in \left]0,1\right[$. Then, for a.e.\ $t\in[0,T]$, we have by interpolating $L^p(\Omega)$-norms 
	\begin{gather*}
		\norm{\phi_{\varepsilon_n}(t,\cdot) - \phi(t,\cdot)}_{L^r(\Omega)} 
		\le \norm{\phi_{\varepsilon_n}(t,\cdot) - \phi(t,\cdot)}_{L^1(\Omega)}^{1-\theta} \norm{\phi_{\varepsilon_n}(t,\cdot) - \phi(t,\cdot)}_{L^q(\Omega)}^\theta
	\end{gather*}
	Then the uniform-in-$n$ $L^\infty(0,T;L^q(\Omega))$-bound (Lemma \ref{lemma:aPrioriEst-phaseSeparation}) and the $L^\infty(0,T;L^1(\Omega))$-convergence imply
	\begin{equation} \label{eq:chemoattractantLinftyLrConvergence}
		\lim_{n\to\infty}\norm{\phi_{\varepsilon_n} - \phi}_{L^\infty(0,T;L^r(\Omega))} = 0 \qquad \forall r \in \left[1,q\right[.
	\end{equation}
	Since $m>q'>1$, we have $m'<q$, so combining the above strong convergence with the weak-$\ast$ convergence in \eqref{eq:densityWeakStarConvergence} gives
	$
	\lim_{n\to\infty} \int_0^T\!\!\!\!\int_\Omega \rho_{\varepsilon_n} \phi_{\varepsilon_n}\,dx\,dt = \int_0^T\!\!\!\!\int_\Omega \rho \phi \,dx\,dt.
	$
	The functionals $\varphi \mapsto \int_0^T\!\!\int_\Omega{\ubar\sfc}\,g(\varphi) \,dx\,dt$ and $\mu\mapsto \int_0^T\!\!\int_\Omega [{\ubar\sfc}\,g]^\ast(\mu)\,dx\,dt$ are weakly sequentially lower semicontinuous for, say, $L^q(\left]0,T\right[\times\Omega)$ and $L^m(\left]0,T\right[\times\Omega)$, resp., so
	\begin{gather*}
		\liminf_{n\to\infty} \int_0^T\!\!\!\!\int_\Omega {\ubar\sfc}\,g(\phi_{\varepsilon_n})\,dx\,dt \ge \int_0^T\!\!\!\!\int_\Omega {\ubar\sfc}\,g(\phi)\,dx\,dt,  \quad \liminf_{n\to\infty} \int_0^T\!\!\!\!\int_\Omega [{\ubar\sfc}\,g]^\ast(\rho_{\varepsilon_n})\,dx\,dt \ge \int_0^T\!\!\!\!\int_\Omega [{\ubar\sfc}\,g]^\ast(\rho) \,dx\,dt.
	\end{gather*}
	Then using the same argument as in Proposition \ref{prop:compactness} provides \eqref{eq:FenchelInequality-Vanishing-timeDep}, which provides \eqref{eq:identityAndVolume} (we omit the details since they are the same as in Proposition \ref{prop:compactness}). In particular, we have proven $\rho = (\rho_+/\phi_+)\phi$ a.e.

	\medskip
	We show the density converges in $L^\infty(0,T;L^1(\Omega))$. This can be done by arguing pointwise in time as in the time-independent case of the compactness result (Proposition \ref{prop:compactness})  provided the convergence of the contributions to $\scG_\varepsilon$ from the reaction terms is essentially uniform-in-time. We point out that the following argument uses 1) some uniform-in-time convergence we have already established for $\seq{\phi_n}$ (cf.\ \eqref{eq:chemoattractantLinftyLrConvergence}) and $\seq{\rho_n}$ (Proposition \ref{prop:continuityEq}-\ref{prop:continuityEq-convergence}) and 2) that $m > q'$ (cf.\ \ref{hyp-g:growthCondition}, \ref{hyp-f:growthCondition}).
	
	\begin{claim} \label{claim:aggregationConvergence-Linfty} $\displaystyle\lim_{n\to\infty} \norm{\int_\Omega\rho_{\varepsilon_n}(\cdot,x)\phi_{\varepsilon_n}(\cdot,x) - \rho(\cdot,x) \phi(\cdot,x)\,dx}_{L^\infty(0,T)}=0$.
	\end{claim}
	\begin{proof}[Proof of claim.]

		Fix $t \in [0,T]$ and let $\phi_\delta(t,x) \coloneq  (\eta_\delta \ast \phi(t))(x)$, where $\eta_\delta$ is the standard mollifier for the spatial variable. Then it is easy to show, using that $\phi \in C^{1/4}([0,T]; L^{m'}(\Omega))$, that $\phi_\delta \in C^{1/4}([0,T];L^{m'}(\Omega)) \cap L^\infty(0,T; W^{\frac{2m}{m-1}}(\Omega))$, and $\phi_\delta \to \phi$ in $C([0,T]; L^{m'}(\Omega))$. Then we split
		\begin{gather*}
			\begin{aligned}
				\abs{\int_\Omega \rho_{\varepsilon_n}(t)\phi_{\varepsilon_n}(t) - \rho(t)\phi(t)\,dx} 
				&\le  \abs{\int_\Omega (\phi_{\varepsilon_n}(t) - \phi(t))(\rho_{\varepsilon_n}(t) - \rho(t))\,dx} 
				+ \abs{\int_\Omega (\phi(t) - \phi_\delta(t)) (\rho_{\varepsilon_n}(t) - \rho(t))\,dx} \\
				&\qquad+ \abs{\int_\Omega \phi_\delta(t)(\rho_{\varepsilon_n}(t) - \rho(t))\,dx} 
				\qquad  +  \abs{\int_\Omega (\phi_{\varepsilon_n}(t) - \phi(t))\rho(t)\,dx}
			\end{aligned} \\
			\begin{aligned}
				&\le \norm{\phi_{\varepsilon_n} - \phi}_{L^\infty(0,T; L^{m'}(\Omega))}(\norm{\rho_{\varepsilon_n} - \rho}_{L^\infty(0,T;L^m(\Omega))} + \norm{\rho}_{L^\infty(0,T; L^m(\Omega))}) \\	
				&	\qquad +\norm{\phi - \phi_\delta}_{C([0,T];L^{m'}(\Omega))}\norm{\rho_{\varepsilon_n} - \rho}_{L^\infty(0,T;L^m(\Omega))} 
				+ \norm{\phi_\delta}_{L^\infty(0,T; W^{1,\frac{2m}{m-1}}(\Omega))}\norm{\rho_{\varepsilon_n} - \rho}_{C([0,T]; (W^{1,\frac{2m}{m-1}}(\Omega))^\ast)}.
			\end{aligned}
		\end{gather*}
		Recall $\seq{\rho_{\varepsilon_n}}$ is bounded in $L^\infty(0,T;L^m(\Omega))$. Then, we take the essential supremum over $t$, send $n\to\infty$ using  \eqref{eq:chemoattractantLinftyLrConvergence} (recall $m'<q$) and Proposition \ref{prop:continuityEq}-\ref{prop:continuityEq-convergence}, and then send $\delta \to 0^+$, which proves the claim.
	\end{proof}	
	
	\begin{claim} \label{claim:gAstConvergence-Linfty} $\displaystyle\lim_{n\to\infty} \norm{\int_\Omega [{\ubar\sfc}\,g]^\ast(\rho_{\varepsilon_n}(\cdot,x)) - [\ubar\sfc\,g]^\ast(\rho(\cdot,x))\,dx}_{L^\infty(0,T)} = 0$.
	\end{claim}
	\begin{proof}[Proof of claim.]
		We show
		\begin{equation}\label{eq:gAstConvergence-goal}
			\begin{gathered}
				0 \le \liminf_{n\to\infty} \left[ \essinf_{t\in[0,T]} \int_\Omega [{\ubar\sfc}\,g]^\ast(\rho_{\varepsilon_n}(t)) - [{\ubar\sfc}\,g]^\ast(\rho(t))\,dx \right] \le \limsup_{n\to\infty} \left[ \esssup_{t\in[0,T]} \int_\Omega [{\ubar\sfc}\,g]^\ast(\rho_{\varepsilon_n}(t)) - [{\ubar\sfc}\,g]^\ast(\rho(t))\,dx \right] \le 0.
			\end{gathered}
		\end{equation}
		The second inequality is obvious, so the task is to show the first and third inequalities.
		\medskip
		
		\textbf{Step 1:} At a.e.\ fixed time $t\in[0,T]$ we have
		\begin{gather}
			\int_\Omega [{\ubar\sfc}\,g]^\ast(\rho_{\varepsilon_n}(t)) - \rho_{\varepsilon_n}(t)\phi_{\varepsilon_n}(t) + {\ubar\sfc}\,g(\phi_{\varepsilon_n}(t)) \,dx + \int_\Omega \rho_{\varepsilon_n}(t)\phi_{\varepsilon_n}(t) - \rho(t)\phi(t)\,dx \nonumber \\
			=  \int_\Omega [{\ubar\sfc}\,g]^\ast(\rho_{\varepsilon_n}(t)) - [{\ubar\sfc}\,g]^\ast(\rho(t))\,dx + \int_\Omega {\ubar\sfc}\,g(\phi_{\varepsilon_n}(t)) - {\ubar\sfc}\,g(\phi(t))\,dx, \label{eq:reactionTerms-identity}
		\end{gather}
		where the second line follows from the equality   $[{\ubar\sfc}\,g]^\ast(\rho) -\rho\phi + {\ubar\sfc}\,g(\phi) = 0$ a.e.\ in $\left]0,T\right[\times\Omega$. The first line converges to $0$ in $L^\infty(0,T)$ by \eqref{eq:FenchelVanishing-LinftyL1} and Claim \ref{claim:aggregationConvergence-Linfty}.
		
		\medskip
		
		\textbf{Step 2:} In \eqref{eq:reactionTerms-identity}, if we obtain an essentially-uniform-in-$t$ asymptotically negligible lower bound on one integral, then we obtain an essentially-uniform-in-$t$ asymptotically negligible upper bound on the other integral. This allows to show the last inequality in \eqref{eq:gAstConvergence-goal}. 
		
		The first-order convexity inequality for $g$ gives, for a.e.\ $t\in[0,T]$,
		\[
		\int_\Omega {\ubar\sfc}\,g(\phi_{\varepsilon_n}(t)) - {\ubar\sfc}\,g(\phi(t))\,dx \ge \int_\Omega {\ubar\sfc}\,g'(\phi(t))(\phi_{\varepsilon_n}(t) - \phi(t)) \,dx = \int_\Omega \rho(t)(\phi_{\varepsilon_n}(t) - \phi(t)).
		\]	
		We use H\"older's inequality, take the essential infimum over $t\in[0,T]$, and use \eqref{eq:chemoattractantLinftyLrConvergence} to deduce  
		\begin{equation} \label{eq:reactionTerm-phi-liminf}
			\begin{gathered}
				\liminf_{n\to\infty}\left[ \essinf_{t\in[0,T]} \int_\Omega g(\phi_{\varepsilon_n}(t)) - g(\phi(t)) \,dx \right] 
				\ge -\norm{\rho}_{L^\infty(0,T;L^m(\Omega))} \lim_{n\to\infty} \norm{\phi_{\varepsilon_n} - \phi}_{L^\infty(0,T;L^{m'}(\Omega))} = 0,
			\end{gathered}
		\end{equation}
		so combining with the identity in Step 1, \eqref{eq:FenchelVanishing-LinftyL1}, and Claim \ref{claim:aggregationConvergence-Linfty} gives the last inequality in \eqref{eq:gAstConvergence-goal}.
		
		\medskip
		\textbf{Step 3:} By the first-order convexity inequality, we obtain for a.e.\ $t\in[0,T]$
		\begin{align*}
			\int_\Omega [{\ubar\sfc}\,g]^\ast(\rho_{\varepsilon_n}(t)) - [{\ubar\sfc}\,g]^\ast(\rho(t))\,dx \ge \int_\Omega {g^\ast}'(\rho(t)/{\ubar\sfc})(\rho_{\varepsilon_n}(t) - \rho(t)) \,dx 
			= \int_\Omega \phi(t)(\rho_{\varepsilon_n}(t) - \rho(t))\,dx.
		\end{align*}
		As in the proof of Claim \ref{claim:aggregationConvergence-Linfty}, we regularize $\phi \in C^{1/4}([0,T]; L^{m'}(\Omega))$ to some \\ $\phi_\delta \in C^{1/4}([0,T]; L^{m'}(\Omega)) \cap L^\infty(0,T; W^{1,\frac{2m}{m-1}}(\Omega))$ such that $\phi_\delta \to \phi$ in $C([0,T]; L^{m'}(\Omega))$. Then we can bound the right-hand side from below by
		\[
		\begin{gathered}
			\ge -\norm{\phi_\delta}_{L^\infty(0,T;  W^{1,\frac{2m}{m-1}}(\Omega))}\norm{\rho_{\varepsilon_n} - \rho}_{C([0,T]; W^{1,\frac{2m}{m-1}}(\Omega)^\ast)} 
			- \norm{\phi - \phi_\delta}_{C([0,T];L^{m'}(\Omega))}\norm{\rho_{\varepsilon_n} - \rho}_{L^\infty(0,T;L^m(\Omega))}.
		\end{gathered}
		\]
		Using the $L^\infty(0,T;L^m(\Omega))$ bound on $\seq{\rho_{\varepsilon_n}}$, we obtain a lower bound on the second term. We then send $n \to \infty$, apply Proposition \ref{prop:continuityEq}-\ref{prop:continuityEq-convergence},  and then send $\delta \to 0^+$, which gives
		\[
		\liminf_{n\to\infty} \left[ \essinf_{t \in [0,T]} \int_\Omega [{\ubar\sfc}\,g]^\ast(\rho_{\varepsilon_n}(t)) - [{\ubar\sfc}\,g]^\ast(\rho(t))\,dx \right] \ge 0.
		\]
		This proves the first inequality in \eqref{eq:gAstConvergence-goal} and thus proves the claim. 

	\end{proof}
	
	The same argument for the strong convergence of the density in Proposition \ref{prop:compactness} can be carried out  pointwise a.e.\ $t\in[0,T]$, and the resulting upper bounds vanish in $L^\infty(0,T)$ due to Claims \ref{claim:aggregationConvergence-Linfty} and \ref{claim:gAstConvergence-Linfty}. This provides $\rho_{\varepsilon_n} \to \rho$ in $L^\infty(0,T;L^1(\Omega))$. By interpolating $L^p(\Omega)$-norms and combining this convergence with the $L^\infty(0,T;L^m(\Omega))$-bound on $\seq{\rho_{\varepsilon_n}}$, we see $\rho_{\varepsilon_n} \to \rho$ in $L^\infty(0,T;L^s(\Omega))$ for all $1 \le s < m$.
	
\end{proof}

\section{Normal velocity of the free boundary \texorpdfstring{$V$}{V}} \label{sec:normalVelocity}

\subsection{Equipartition of the energies}
We derive some classical consequences of the assumption that $(\rho,\phi)$ and $\seq{(\rho_{\varepsilon_n}, \phi_{\varepsilon_n})}$ satisfy the energy convergence condition \ref{hyp:energyConvergenceG}; see \cite[Lemma 1]{Luckhaus.Modica_1989_GibbsThompsonRelationGradient}, \cite[proof of Theorem 3.3]{Cicalese.Nagase.ea_2010_GibbsThomsonRelation} \cite[Lemma 2.11]{Laux.Simon_2018_ConvergenceAllenCahnEquation}, \cite[Section 5.2]{Kim.Mellet.ea_2023_DensityconstrainedChemotaxisHeleShaw}, \cite[proof of Proposition 5.1]{Mellet_2024_HeleShawFlowSingular}, \cite[Section 4.2]{Laux.Stinson.ea_2024_DiffuseinterfaceApproximationWeak}, or \cite[Lemma 7.1]{Mellet.Rozowski_2024_VolumepreservingMeancurvatureFlow}. Namely, we show, in the limit $\varepsilon \to 0^+$, the potential and interface penalty in $\scF_{\sfc,\varepsilon}$ are equal asymptotically, the sum of the potential $\varepsilon^{-1} W(\rho_\varepsilon)$ and penalty $\varepsilon^{-1}\big([{\ubar\sfc}\,g]^\ast(\rho_\varepsilon) - \rho_\varepsilon\phi_\varepsilon + {\ubar\sfc}\,g(\phi_\varepsilon) \big)$ coupling $(\rho_\varepsilon,\phi_\varepsilon)$ is asymptotically equal to the interface penalty, and the gaps between $\scF_\varepsilon[\phi_\varepsilon]$, $\scF_{\sfc,\varepsilon}[\phi_\varepsilon]$, $\scG_\varepsilon[\rho_\varepsilon,\phi_\varepsilon]$, and $\scG_0[\rho,\phi]$ vanish. These allows us to show, at a.e.\ time $t\in[0,T]$, the contributions to the energies converge weakly-$\ast$ as Radon measures on $\cl\Omega$ to (an equal proportion of) the limiting contributor to the limiting energy $\scG_0$ (Lemma \ref{lemma:equipartitionOfEnergy-timeIndependent}-\ref{lemma:weakStarConvergence-Energy}).

\begin{lemma}[equipartition of the energies] \label{lemma:equipartitionOfEnergy-timeIndependent}
	Let $\seq{(\rho_{\varepsilon_n}, \phi_{\varepsilon_n})}$ be weak solutions of \eqref{eq:PP-PKSeps} with well-prepared initial data $\seq{(\rho_{\varepsilon_n}^\init, \phi_{\varepsilon_n}^\init)}$, i.e., satisfying \ref{hyp:well-preparedG}, for which the conclusion of Proposition \ref{prop:phaseSeparation} holds with limit $(\rho,\phi)$. If $(\rho,\phi)$, $\seq{(\rho_{\varepsilon_n}, \phi_{\varepsilon_n})}$ satisfy energy convergence  \ref{hyp:energyConvergenceG}, then up to another subsequence:
	\begin{enumerate}[label = (\roman*)]
		\item \label{lemma:equalEnergies} for a.e.\ $t \in [0,T]$ and strongly in $L^1(0,T)$, 
		\begin{equation*}
			\begin{aligned}
				\scG_0[\rho,\phi] = \lim_{n\to\infty} \scG_{\varepsilon_n}[\rho_{\varepsilon_n}, \phi_{\varepsilon_n}] &= \lim_{n\to\infty} \scF_{\sfc,\varepsilon}[\phi_{\varepsilon_n}] =
				\scF_{\varepsilon_n}[\phi_{\varepsilon_n}]  = \lim_{n\to\infty} \int_\Omega \abs{\nabla\psi_{\varepsilon_n}}_\sfA\,dx;
			\end{aligned} 
		\end{equation*}

		\item \label{lemma:energyObstacle} for a.e.\ $t \in [0,T]$ and strongly in $L^1(0,T)$, $\displaystyle\lim_{n\to\infty} \frac{1}{\varepsilon_n}\int_\Omega (\sfc(x) - {\ubar\sfc}) g(\phi_{\varepsilon_n}) \,dx = 0$;
		
		\item \label{lemma:equalPotentialEnergies} $\displaystyle\lim_{n\to\infty}\frac{1}{\varepsilon_n}\Big( W(\rho_{\varepsilon_n}) + \big( [{\ubar\sfc}\,g]^\ast(\rho_{\varepsilon_n}) - \rho_{\varepsilon_n}\phi_{\varepsilon_n} + {\ubar\sfc}\,g(\phi_{\varepsilon_n}) \big) - W_\ast(\phi_{\varepsilon_n}) \Big) = 0$ strongly in $L^1(\left]0,T\right[\times\Omega)$;
		
		\item \label{lemma:equipartitionFeps} $\displaystyle \lim_{n\to\infty} \Big(  \frac{1}{\sqrt{\varepsilon_n}}\sqrt{W_\ast(\phi_{\varepsilon_n})} - \sqrt{\frac{\varepsilon_n}{2}}\abs{\nabla\phi_{\varepsilon_n}}_\sfA \Big) = 0$ strongly in $L^2(0,T;L^2(\Omega))$ and \\
		$\displaystyle \lim_{n\to\infty}\Big( \frac{1}{\varepsilon_n} W_\ast(\phi_{\varepsilon_n}) - \frac{1}{2}\abs{\psi_{\varepsilon_n}}_\sfA\Big) = \lim_{n\to\infty}\Big( \frac{\varepsilon_n}{2} \abs{\nabla\phi_{\varepsilon_n}}_\sfA^2 - \frac{1}{2}\abs{\psi_{\varepsilon_n}}_\sfA\Big) = 0$ strongly in $L^1(\left]0,T\right[\times\Omega)$;
		
		\item \label{lemma:weakStarConvergence-Energy} for any test function $\zeta \in C(\cl\Omega)$ and a.e.\ $t \in [0,T]$
		\begin{gather*} 
			\frac{\gamma}{2} \int_\Omega \zeta(x) \abs{\vec{\nu}(t,x)}_{\sfA(x)} \,d\vert\nabla\phi(t)\vert(x) = \lim_{n\to\infty} \int_\Omega  \zeta(x)\frac{1}{2}\abs{\nabla\psi_{\varepsilon_n}(t,x)}_{\sfA(x)} \,dx \\
			= \lim_{n\to\infty} \int_\Omega \zeta(x) \frac{1}{\varepsilon_n}W_\ast(\phi_{\varepsilon_n}(t,x)) \,dx 			
			= \lim_{n\to\infty} \int_\Omega \zeta(x) \frac{\varepsilon_n}{2} \abs{\nabla\phi_{\varepsilon_n}(t,x)}_{\sfA(x)}^2 \,dx.
		\end{gather*}
	\end{enumerate}
\end{lemma}

\begin{remark}[asymptotic equality]
	\begin{enumerate}
		\item The limit in \ref{lemma:equalPotentialEnergies} expresses asymptotic equality in the Fenchel-Young inequality $f(\rho_\varepsilon) + f^\ast(\phi_\varepsilon - a) - \rho_\varepsilon(\phi_\varepsilon-a) \ge 0$. The formula \eqref{eq:GepsModicaMortola} for $\scG_\varepsilon$ shows this difference in  \ref{lemma:equalPotentialEnergies} is the density of the gap between $\scG_\varepsilon$ and the Modica-Mortola functional $\scF_{{\sfc},\varepsilon}$ (cf.\ \eqref{eq:ModicaMortolaF}).
		
		\item The first limit in \ref{lemma:equipartitionFeps} is classical and represents asymptotic equality in Young's inequality. Indeed, expanding the square in the $L^2(\Omega)$-norm gives $\frac{1}{\varepsilon_n}W_\ast(\phi_{\varepsilon_n}) + \frac{\varepsilon_n}{2}\abs{\nabla\phi}_\sfA^2 - \sqrt{2W_\ast(\phi_{\varepsilon_n})} \abs{\nabla\phi_{\varepsilon_n}}_\sfA$, which is nonnegative by Young's inequality.
	\end{enumerate}
\end{remark}

\begin{proof}
	Throughout the proof, we will abbreviate the energies as
	\begin{gather*}
		\scG_n(t) \coloneq \scG_{\varepsilon_n}[\rho_{\varepsilon_n}(t), \phi_{\varepsilon_n}(t)], \quad \scF_{\sfc,n}(t) \coloneq \scF_{\sfc,\varepsilon}[\phi_{\varepsilon_n}(t)], \quad \scF_n(t) \coloneq \scF_{\varepsilon_n}[ \phi_{\varepsilon_n}(t)], \quad
		\scG_0(t) \coloneq \scG_0[\rho(t), \phi(t)].
	\end{gather*}
	In this notation, the energy dissipation inequality (Lemma \ref{lemma:aPrioriEst-continuityEq}-\ref{lemma:boundedEnergies}) and \ref{hyp:well-preparedG} provide $\scG_n(t) \le {\ol\scG}$ for a.e.\ $t\in[0,T]$ and all $n\in\bbN$. We recall the ordering (see \eqref{eq:GFineq},  \eqref{eq:ModicaMortolaF}) and that applying the Modica-Mortola trick for $\scF_n(t)$ gives
	\begin{equation} \label{eq:equipartition-energyOrdering}
		\begin{aligned}
			\text{a.e.\ } t \in[0,T] \text{ and each } n \in \bbN \qquad	{\ol\scG} \ge \scG_n(t) 
			\ge \scF_{\sfc,n}(t) \ge \scF_n(t) \ge \int_\Omega \abs{\nabla\psi_{\varepsilon_n}(t)}_{\sfA(x)}\,dx \ge 0, 
		\end{aligned}
	\end{equation}
	where $\psi_{\varepsilon_n}(t,x) \coloneq F(\phi_{\varepsilon_n}(t,x))$ with $F$ defined by \eqref{eq:auxF} (see also Lemma \ref{lemma:transformationF}).

	\ref{lemma:equalEnergies}: For the first equality, we write $\abs{\scG_n(t) - \scG_0(t)} = 2(\scG_0(t) - \scG_n(t))_+ - (\scG_0(t) - \scG_n(t))$. Then after an integration over $t \in [0,T]$, the second term vanishes as $n \to \infty$ by the energy convergence assumption \ref{hyp:energyConvergenceG}, so it only remains to show the first term vanishes in $L^1(0,T)$. The $\liminf$ inequality of $\Gamma$-convergence (Theorem \ref{thm:gammaConvergence}-\ref{thm:liminfIneq}) implies for a.e.\ $t\in[0,T]$ that $\lim_{n\to\infty}(\scG_0(t) - \scG_n(t))_+ = 0$, and the $\liminf$ inequality combined with Fatou's lemma implies $\scG_0 \in L^1(0,T)$. By nonnegativity of $\scG_n$, we have the upper bound $(\scG_0(t) - \scG_n(t))_+ \le \scG_0(t)$ for a.e.\ $t \in [0,T]$. Then Lebesgue dominated convergence gives convergence in $L^1(0,T)$, so after extracting a subsequence we have  \begin{equation}\label{eq:energyConvergence-Time}
		\lim_{n\to\infty}\scG_n = \scG_0  \text{ a.e.\ }t \in [0,T] \text{ and strongly in } L^1(0,T).
	\end{equation}
	
	To show the remaining equalities in \ref{lemma:equalEnergies}, we first prove the convergence pointwise for a.e.\ $t\in[0,T]$ by proving a $\liminf$ and $\limsup$ inequality, which is made straightforward by \eqref{eq:energyConvergence-Time} and the ordering \eqref{eq:equipartition-energyOrdering}. 
	
	Since $\psi_{\varepsilon_n} \to \gamma\phi$ in $C([0,T]; L^1(\Omega))$ (Proposition \ref{prop:phaseSeparation}), $\phi \in L^\infty(0,T; BV(\Omega))$, and for each $t\in[0,T]$ the sequence $\seq{\psi_{\varepsilon_n}(t)}$ is bounded in $BV(\Omega)$, we have for each $t \in [0,T]$ that $\nabla\psi_{\varepsilon_n}(t) \to \gamma\nabla\phi(t)$ weakly-$\ast$ as Radon measures on $\cl\Omega$. Then \eqref{eq:energyConvergence-Time}, \eqref{eq:equipartition-energyOrdering}, and Reshetnyak lower semicontinuity \cite[Theorem 2.38]{Ambrosio.Fusco.ea_2000_FunctionsBoundedVariation} give for a.e.\ $t \in [0,T]$ 
	\begin{equation} \label{eq:psi-liminf-pointwise-in-t}
		\begin{gathered}
			\scG_0(t) = \lim_{n\to\infty} \scG_n(t) \ge	\liminf_{n\to\infty} \scF_{\sfc,n}(t) \ge \liminf_{n\to\infty} \scF_n(t) \\
			\ge	\liminf_{n\to\infty} \int_\Omega \abs{\frac{\nabla \psi_{\varepsilon_n}(t,x)}{\abs{\nabla\psi_{\varepsilon_n}(t,x)}}}_{\sfA(x)}  \abs{\nabla\psi_{\varepsilon_n}(t,x)}\,dx 
			\ge \gamma \int_\Omega \abs{\vec{\nu}(t,x)}_{\sfA(x)} d\vert\nabla\phi(t)\vert(x) = \scG_0(t).
		\end{gathered}
	\end{equation}
	Thus, all of the inequalities are equalities. The corresponding $\limsup$ inequalities are trivial; we have from \eqref{eq:energyConvergence-Time}, \eqref{eq:equipartition-energyOrdering}, and \eqref{eq:psi-liminf-pointwise-in-t} that for a.e.\ $t \in[0,T]$
	\begin{equation} \label{eq:psi-limsup-pointwise-in-t}
		\begin{gathered}
			\scG_0(t) = \lim_{n\to\infty} \scG_n(t) \ge \limsup_{n\to\infty} \scF_{\sfc,n}(t) \ge	\limsup_{n\to\infty} \scF_n(t) \\
			\ge \limsup_{n\to\infty} \int_\Omega \abs{\nabla\psi_{\varepsilon_n}(t)}_{\sfA(x)}\,dx \ge \liminf_{n\to\infty} \int_\Omega \abs{\nabla\psi_{\varepsilon_n}(t)}_{\sfA(x)}\,dx \ge \scG_0(t),
		\end{gathered}
	\end{equation}
	so the inequalities are equalities again. Then \eqref{eq:psi-liminf-pointwise-in-t}, \eqref{eq:psi-limsup-pointwise-in-t} show the pointwise a.e.\ convergence in \ref{lemma:equalEnergies}. On the other hand, the upper bound in \eqref{eq:equipartition-energyOrdering} provided by energy dissipation and well-preparedness shows each of the energies is dominated on $[0,T]$, so Lebesgue dominated convergence yields $L^1(0,T)$ convergence, which completes the proof of \ref{lemma:equalEnergies}.
	
	\medskip
	\ref{lemma:energyObstacle}: Since $(\sfc - {\ubar\sfc})g(\phi_{\varepsilon_n}) \ge 0$ a.e.\ in $\left]0,T\right[ \times \Omega$ (recall \ref{hyp-c:regularity}, \ref{hyp-g:regularity}), we note from \eqref{eq:ModicaMortolaF}, \eqref{eq:WstarC}
	\[
	0 \le \frac{1}{\varepsilon_n}\int_\Omega (\sfc - {\ubar\sfc})g(\phi_{\varepsilon_n})\,dx = \scF_{\sfc,n} - \scF_n \xrightarrow{n\to\infty} 0,
	\]
	where the convergence follows from the second and third limits of \ref{lemma:equalEnergies}.
	
	\medskip
	\ref{lemma:equalPotentialEnergies}: The definition of $W_\ast$ \eqref{eq:Wstar} implies $W(\rho_{\varepsilon_n}) + \big([{\ubar\sfc}\,g]^\ast(\rho_{\varepsilon_n}) - \rho_{\varepsilon_n}\phi_{\varepsilon_n} + {\ubar\sfc}\,g(\phi_{\varepsilon_n}) \big) - W_\ast(\phi_{\varepsilon_n}) \ge 0$ a.e.\ in $\left]0,T\right[\times\Omega$. Then the definition of $\scF_{\sfc,n}$ \eqref{eq:ModicaMortolaF}, that $\scG_n(t) \ge \scF_{\sfc,n}(t)$ \eqref{eq:GFineq}, and the first and third limits of \ref{lemma:equalEnergies} provide
	\[
	0 \le \frac{1}{\varepsilon_n}\int_\Omega W(\rho_{\varepsilon_n}) + \big([{\ubar\sfc}\,g]^\ast(\rho_{\varepsilon_n}) - \rho_{\varepsilon_n}\phi_{\varepsilon_n} + {\ubar\sfc}\,g(\phi_{\varepsilon_n}) \big) - W_\ast(\phi_{\varepsilon_n}) \,dx =  \scG_n - \scF_{\sfc,n} \xrightarrow{n\to\infty} 0.
	\]
	
	\medskip
	\ref{lemma:equipartitionFeps}: For clarity, let 
	\begin{gather*}
		u_n \coloneq \frac{\varepsilon_n}{2} \abs{\nabla \phi_{\varepsilon_n}}^2_\sfA,	\qquad 
		v_n \coloneq \frac{1}{\varepsilon_n}W_\ast(\phi_{\varepsilon_n}), \\
		o_n \coloneq \frac{1}{\varepsilon_n}(\sfc - {\ubar\sfc})g(\phi_{\varepsilon_n}), \qquad 
		w_n \coloneq \frac{1}{\varepsilon_n} \Big(W(\rho_{\varepsilon_n}) + \big([{\ubar\sfc}\,g]^\ast(\rho_{\varepsilon_n}) - \rho_{\varepsilon_n}\phi_{\varepsilon_n} + {\ubar\sfc}\,g(\phi_{\varepsilon_n}) \big)\Big),    
	\end{gather*}
	so $\scF_n(t) = \int_\Omega v_n(t) + u_n(t)\,dx$ and $\scG_n(t) = \int_\Omega w_n(t) + o_n(t) + u_n(t)\,dx$. Then the Modica-Mortola trick and some algebra give
	$
	0 \le \abs{\nabla \psi_{\varepsilon_n}}_\sfA = 2\sqrt{u_n v_n}  = u_n + v_n - (\sqrt{u_n} - \sqrt{v_n})^2.
	$
	Integrating over $x\in\Omega$ and rearranging then provides
	\[
	\norm{\sqrt{u_n} - \sqrt{v_n}}_{L^2(\Omega)}^2 = \scF_n - \int_\Omega \abs{\nabla\psi_{\varepsilon_n}}_\sfA \,dx  \xrightarrow{n\to\infty} 0 \text{ strongly in } L^1(0,T), 
	\]
	where the convergence is a consequence of the third and fourth equalities in \ref{lemma:equalEnergies}.
	
	By introducing a difference of squares, using Cauchy-Schwarz, recalling the uniform-in-$t$-and-$n$ bound on $
	\scF_n$ \eqref{eq:equipartition-energyOrdering}, and applying the above limit,
	\begin{equation}\label{eq:equipartitionOfEnergy-vanishingdifference}
		\begin{gathered}
			\norm{u_n - v_n}_{L^1(\Omega)} \le \norm{\sqrt{u_n} + \sqrt{v_n}}_{L^2(\Omega)} \norm{\sqrt{u_n} - \sqrt{v_n}}_{L^2(\Omega)} 
			\le C(\scF_n)^{1/2} \norm{\sqrt{u_n} - \sqrt{v_n}}_{L^2(\Omega)} \xrightarrow{n\to\infty} 0,
		\end{gathered}
	\end{equation}
	where the convergence is in $L^2(0,T)$-strong. 
	
	Now we observe, strongly in $L^1(\left]0,T\right[\times\Omega)$, 
	\begin{equation*}
		\lim_{n\to\infty}\big(v_n - \frac{1}{2}\abs{\nabla\psi_{\varepsilon_n}}_{A}\big) 
		=  \frac{1}{2}\lim_{n\to\infty}\big(v_n + u_n - \abs{\nabla\psi_{\varepsilon_n}}_\sfA\big) + \frac{1}{2} \lim_{n\to\infty}\big(v_n - u_n\big) = 0.
	\end{equation*}
	On the right-hand side, the first limit is a consequence of the third and fourth equalities in \ref{lemma:equalEnergies} and the second limit is a consequence of \eqref{eq:equipartitionOfEnergy-vanishingdifference}. This establishes the second limit in \ref{lemma:equipartitionFeps}. By interchanging the roles of $u_n$ and $v_n$ in the above argument, we obtain the third limit in \ref{lemma:equipartitionFeps}.
	
	\medskip
	\ref{lemma:weakStarConvergence-Energy}: We will prove the first equality for test functions $\zeta \in C(\cl\Omega)$ satisfying $0 \le \zeta(x) \le 1$ for all $x \in \Omega$. The extension to arbitrary $\zeta \in C(\cl\Omega)$ is obtained by linearity. 
	
	We find for a.e.\ $t\in[0,T]$,
	\[
	\int_\Omega \zeta\Big( \frac{1}{\varepsilon_n} W_\ast(\phi_{\varepsilon_n}(t)) + \frac{\varepsilon_n}{2} \abs{\nabla\phi_{\varepsilon_n}(t)}_\sfA^2 \Big) \,dx \ge \int_\Omega \zeta\abs{\nabla \psi_{\varepsilon_n}(t)}_{\sfA(x)}  \,dx,
	\]
	and since $\psi_{\varepsilon_n}(t) \to \gamma\phi(t)$ strongly in $L^1(\Omega)$ (Proposition \ref{prop:phaseSeparation}), the Reshetnyak lower semicontinuity theorem gives, a.e.\ $t \in [0,T]$,  
	\begin{equation*}
		\liminf_{n\to\infty} \int_\Omega \zeta\Big( \frac{1}{\varepsilon_n} W_\ast(\phi_{\varepsilon_n}(t)) + \frac{\varepsilon_n}{2} \abs{\nabla\phi_{\varepsilon_n}(t)}_{\sfA(x)}^2 \Big) \,dx 
		\ge \liminf_{n\to\infty} \int_\Omega \zeta\abs{\nabla\psi_{\varepsilon_n}(t)}_{\sfA(x)} \,dx \ge \gamma \int_\Omega \zeta\abs{\vec{\nu}(t)}_{\sfA(x)}d\vert\nabla\phi(t)\vert(x).
	\end{equation*}	
	To obtain the $\limsup$ inequalities, we replace $\zeta$ by $1-\zeta$ in the above inequality and apply the pointwise a.e.\ convergences $\scF_n(t), \, \int_\Omega\abs{\nabla\psi_{\varepsilon_n}(t)}_\sfA \,dx \xrightarrow{n\to\infty} \scG_0(t)$  from \ref{lemma:equalEnergies}. 
	
	The second and third limits in \ref{lemma:weakStarConvergence-Energy} are obtained by combining the first limit in  \ref{lemma:weakStarConvergence-Energy} with 
	the second and third limits in \ref{lemma:equipartitionFeps},  resp. 
\end{proof}

Fixing a test function $\zeta \in C(\cl\Omega)$, part  \ref{lemma:weakStarConvergence-Energy} gives convergence for pointwise a.e.\ $t\in[0,T]$. The uniform-in-$t$-and-$\varepsilon$ bound on the energy $\scF_{\varepsilon_n}[\phi_{\varepsilon_n}]$ \eqref{eq:equipartition-energyOrdering} and Lebesgue dominated convergence give the following time-dependent version of the weak-$\ast$ convergence in \ref{lemma:weakStarConvergence-Energy}.  We will use it to deduce the convergence of the various terms in the weak formulation of \eqref{eq:PP-PKSeps} to those in our weak formulation of \eqref{eq:HS-STKU}.
\begin{corollary}\label{cor:equipartition-timeDependent}
	Let $\seq{(\rho_{\varepsilon_n}, \phi_{\varepsilon_n})}$ be as in the previous lemma. For any test function $\zeta \in C([0,T]\times\cl\Omega)$ 
	\begin{equation}\label{eq:equipartition-timeDependent-convergences}
		\begin{gathered}
			\frac{\gamma}{2} \int_0^T\int_\Omega \zeta(t,x) \abs{\vec{\nu}(t,x)}_{\sfA(x)} \,d\vert\nabla\phi(t)\vert(x)\,dt = \lim_{n\to\infty} \int_0^T\int_\Omega  \zeta(t,x)\frac{1}{2}\abs{\nabla\psi_{\varepsilon_n}(t,x)}_{\sfA(x)} \,dx\,dt \\
			= \lim_{n\to\infty} \int_0^T\int_\Omega \zeta(t,x) \frac{1}{\varepsilon_n}W_\ast(\phi_{\varepsilon_n}(t,x)) \,dx \,dt 	
			= \lim_{n\to\infty} \int_0^T\int_\Omega \zeta(t,x) \frac{\varepsilon_n}{2} \abs{\nabla\phi_{\varepsilon_n}(t,x)}_{\sfA(x)}^2 \,dx\,dt.
		\end{gathered}
	\end{equation}
\end{corollary}

\subsection{The normal component of the velocity of the free boundary (Proof of identity \texorpdfstring{\eqref{eq:weakNormalComponent}}{2.16} in Theorem \texorpdfstring{\ref{thm:main}-\ref{thm:conditionalConvergence}}{2.3-(iii)})}
The aim of this section is to demonstrate that the limiting velocity $\vec{v}$ derived in Proposition \ref{prop:continuityEq}-\ref{prop:continuityEq-velocity} admits, in a weak sense, a normal trace on the free boundary. This is handled through methods that are classical for the derivation of $BV$ solutions to mean-curvature flow as a singular limit of strong solutions of Allen-Cahn-type equations (see, e.g., \cite[Proposition 2.10]{Laux.Simon_2018_ConvergenceAllenCahnEquation}, \cite[Section 3.1.3]{Hensel.Laux_2024_BVSolutionsMean}, \cite[Lemma 4.4]{Laux.Stinson.ea_2024_DiffuseinterfaceApproximationWeak}, or \cite[Lemma 7.2]{Mellet.Rozowski_2024_VolumepreservingMeancurvatureFlow}). 

The phase separation result (Proposition \ref{prop:phaseSeparation}) shows there is a limit point $(\rho,\phi)$ of $\seq{(\rho_{\varepsilon_n}, \phi_{\varepsilon_n})}$ that can be written, for each $t \in [0,T]$, as  $\frac{\rho_+}{\phi_+} \phi(t) = \rho(t) = \rho_+ \chi_{E(t)}$, where $E(t) \subset \cl\Omega$ is a set of finite perimeter in $\Omega$. The following result shows $\phi$ satisfies a distributional formulation of a level set equation with normal velocity $V$. 

\begin{lemma}[existence of the normal velocity $V$]
	There exists a constant $C({\ol\scG}) > 0$ such that, if the sequence $\seq{(\rho_{\varepsilon_n}, \phi_{\varepsilon_n})}$ from the phase separation result and its limit $(\rho,\phi)$ satisfy the energy convergence assumption \ref{hyp:energyConvergenceG}, then for all test functions $\zeta \in C^1_c(\left]0,T\right[\times\Omega)$ we have 
	\[
	\abs{\int_0^T \int_\Omega \phi(t,x)\partial_t \zeta(t,x) \,dx\,dt} \le C \left( \int_0^T \int_\Omega \zeta(t,x)\,d\vert\nabla\phi(t)\vert(x)\,dt\right)^{1/2}.
	\]
	In particular, there exists $V \in L^2(\left]0,T\right[\times\Omega,\abs{\nabla\phi}dt)$ such that $\partial_t\phi = V \abs{\nabla \phi}\,dt$ as Radon measures on $[0,T]\times\cl\Omega$ and thus as distributions in the following sense:
	\begin{equation} \label{eq:levelSet-phi}
		\forall \, \zeta \in C^1_c(\left[0,T\right[ \times \cl\Omega) \qquad \int_0^T \int_\Omega \phi \partial_t\zeta \,dx\,dt + \int_\Omega \phi(0,x)\zeta(0,x)\,dx = -\int_0^T \int_\Omega V \zeta \,d\vert\nabla\phi(t)\vert(x)\,dt.
	\end{equation}
	Moreover, for any $\zeta \in C^1_c(\left]0,T\right[ \times \cl\Omega)$, we have 
	$	\lim_{n\to\infty} \int_0^T\int_\Omega \partial_t \psi_{\varepsilon_n} \zeta \,dx\,dt = \int_0^T \int_\Omega \gamma_0 \zeta V \,d\vert\nabla\phi\vert\,dt$.
\end{lemma}
\begin{proof}
	We refer to \cite[Lemma 7.2]{Mellet.Rozowski_2024_VolumepreservingMeancurvatureFlow} for a proof.
\end{proof}

To derive \eqref{eq:weakNormalComponent}, the second equation in our weak formulation of \eqref{eq:HS-STKU}, we recall $\rho = \frac{\rho_+}{\phi_+}\phi$. Because $(\rho,\vec{v})$ satisfies  \eqref{eq:limitContinuityEquation} and $(\phi,V)$ satisfies \eqref{eq:levelSet-phi}, a linear combination of these equations gives \eqref{eq:weakNormalComponent}. Thus, $V(t)$ is equal (in a weak sense) to the normal component of $\vec{v}(t)$ on the free boundary $\partial E(t)$.

\section{Derivation of Hele-Shaw flow with surface tension and kinetic undercooling} 
In this section, we show that a potential for the organism velocity $\vec{v}_\varepsilon = -\nabla\big(f'(\rho_\varepsilon) + \sfa - \phi_\varepsilon \big)/\varepsilon$ leads to the free boundary condition $p(t) + \gamma_0 \nabla p(t) \cdot \vec{\nu}(t) = \gamma_0 \kappa_\sfA(t)$ on $\partial E(t)$, where $\kappa_\sfA$ is the anisotropic mean-curvature of the free boundary. In view of the velocity law $V(t) = -\nabla p(t) \cdot \vec{\nu}(t)$ on $\partial E(t)$, it is equivalent to show $p(t) = \gamma_0 V(t) + \gamma_0\kappa_\sfA(t)$ on $\partial E(t)$. We first show how the undercooling and curvature terms arise and then we apply this to derive \eqref{eq:weakFormFBCs}, the final part of our weak formulation of \eqref{eq:HS-STKU}.

\subsection{Derivation of the undercooling and curvature terms} \label{subsec:derivationUndercoolingCurvature}
Theorem \ref{thm:gammaConvergence} shows the energy $\scG_\varepsilon$ \eqref{eq:GepsModicaMortola} dissipated by the system $\Gamma$-converges for $L^1(\Omega)^2$-strong to a weighted perimeter $\scG_0$ \eqref{eq:G0}. The phase separation result (Proposition \ref{prop:phaseSeparation}) provides, essentially uniformly in time, the strong $L^1(\Omega)^2$-convergence of a subsequence of solutions to the system:  $\lim_{\varepsilon \to 0^+}(\rho_\varepsilon,\phi_\varepsilon) = (\rho,\phi)$, where $\rho(t) = (\rho_+/\phi_+)\phi(t) = \rho_+\chi_{E(t)}$. The energy convergence assumption \ref{hyp:energyConvergenceG} provides that the time-integrated energies (evaluated along this converging subsequence) also converge: $\int_0^T \scG_\varepsilon[\rho_\varepsilon(t), \phi_\varepsilon(t)] \,dt \xrightarrow{\varepsilon\to0^+} \int_0^T \scG_0[\rho(t),\phi(t)]\,dt$. We would then like to show, as a consequence of these limits, first variations of the energies $\seq{\scG_\varepsilon}$ evaluated along $\seq{(\rho_\varepsilon,\phi_\varepsilon)}$ converges to a first variation of the limit energy $\scG_0$ evaluated at the limit $(\rho,\phi)$. 

\begin{proposition}[convergence of first variations]\label{prop:Reshetnyak-timeDep}
	Suppose $\seq{(\rho_{\varepsilon_n},\phi_{\varepsilon_n})}$ satisfies the conclusion of the phase separation result (Proposition \ref{prop:phaseSeparation}) with limit $(\rho,\phi)$ and energy convergence \ref{hyp:energyConvergenceG} is satisfied. Let $\sfc$ satisfy \ref{hyp-c:grad-nondegeneracy}, and let $\vec{n}$ and $\vec{n}_0$ denote the outer unit normals to $\partial\Omega$ and $(\partial\Omega_0)\cap \Omega$, resp. Then for every test vector field $\xi \in C^1([0,T] \times \cl\Omega; \bbR^d)$ satisfying $\xi \cdot \vec{n} = 0$ on $\partial\Omega$ and $\xi \cdot \vec{n}_0 = 0$ on $(\partial\Omega_0) \cap \Omega$: 
	\begin{equation} \label{eq:convergenceFirstVariation}
		\begin{aligned}	
			\lim_{n\to\infty} &\frac{1}{\varepsilon_n}\int_0^T\int_\Omega \big[ f(\rho_{\varepsilon_n}) + \sfa\rho_{\varepsilon_n} - \rho_{\varepsilon_n} \phi_{\varepsilon_n} \big] \div \xi - \rho_{\varepsilon_n} \nabla\phi_{\varepsilon_n} \cdot \xi \,dx \,dt \\
			&=  \gamma_0\int_0^T \int_\Omega V\vec{\nu}\cdot\xi\,d\vert\nabla\rho(t)\vert(x)\,dt \\
			&\quad+  \gamma_0 \int_0^T \int_\Omega \Big[ \div\xi - \frac{A\vec{\nu}}{\abs{\vec{\nu}}_\sfA} \otimes \frac{\vec{\nu}}{\abs{\vec{\nu}}_\sfA} \colon (\nabla \xi)^\trans + \frac{1}{2}\frac{\vec{\nu}}{\abs{\vec{\nu}}_\sfA} \otimes \frac{\vec{\nu}}{\abs{\vec{\nu}}_\sfA} \colon \big(\xi^k\partial_k A\big)  \Big] \abs{\vec{\nu}}_\sfA \,d\vert\nabla\rho(t)\vert(x)\,dt,
		\end{aligned}
	\end{equation} 
	where $\gamma_0 \coloneq \gamma\phi_+/\rho_+$ and  $-\vec{\nu}(t)$ is the Radon-Nikodym derivative of $\nabla\rho(t)$ w.r.t.\  $\abs{\nabla\rho(t)}$.
\end{proposition}

The quantity on the first line of \eqref{eq:convergenceFirstVariation} being integrated over time is, formally, $\frac{d}{ds}\big\vert_{s=0} \scG_\varepsilon[\rho_s,\phi_\varepsilon]$, where $\rho_s$ is a perturbation of $\rho_\varepsilon$ by the flow of the transport equation: $\partial_s\rho_s + \xi \cdot \nabla \rho_s = 0$ with $\rho_0 = \rho_\varepsilon$ for a prescribed velocity field $\xi$. This perturbation is motivated by the fact that the distributional anisotropic mean-curvature arises is the first variation of the weighted perimeter $\scG_0$ w.r.t.\ these perturbations (domain variations). We perturb the first variable of $\scG_\varepsilon$ as opposed to its second variable (or both variables simultaneously) since, for our scaling of \eqref{eq:PP-PKSeps}, the dynamics corresponding to the first equation of \eqref{eq:PP-PKSeps} are slower than those of the second equation, so we expect them to dominate for  $\varepsilon\ll 1$. 

We require that the test vector fields $\xi$ satisfy $\xi \cdot \vec{n}_0 = 0$ on $(\partial \Omega_0) \cap \Omega$. We impose this restriction since all functions in the domain of the limit $\scG_0$ are supported on $\Omega_0$, so if $\rho$ is supported on $\Omega_0$, then perturbations of $\rho$ by the flow of the transport equation with velocity field $\xi$ will also be supported on $\Omega_0$. We note that the domain of  $\scG_0$ is not invariant under such a large class of perturbations. For instance, these do not necessarily preserve the mass, $1 = \int_\Omega \rho\,dx \ne \int_\Omega \rho_s \,dx$, since they do not not necessarily preserve the volume, $\rho_+^{-1} = \abs{E} \ne \abs{E_s}$, unless we further restrict to divergence-free vector fields. 

The integral on the second line of \eqref{eq:convergenceFirstVariation} is a weak formulation of the so-called undercooling term in \eqref{eq:HS-STKU}, i.e., the second term in the fourth equation of \eqref{eq:HS-STKU}, while the third line of \eqref{eq:convergenceFirstVariation} is a weak formulation of the anisotropic mean-curvature $\kappa_\sfA$ that appears in the fourth equation of \eqref{eq:HS-STKU}. In the limit, we obtain more than just the first variation of the limit energy (namely, the undercooling term) because we evaluate the energy along solutions $(\rho_\varepsilon,\phi_\varepsilon)$ to  \eqref{eq:PP-PKSeps}, as opposed to, say, minimizers of $\scG_\varepsilon$. 

\medskip
The proof of Proposition \ref{prop:Reshetnyak-timeDep} is facilitated by rewriting the first line of \eqref{eq:convergenceFirstVariation} in a manner more amenable to passing to the limit. The aim is to introduce terms whose convergence is addressed by the equipartition of the energy or Reshetnyak continuity. 
\begin{lemma}[first variation] \label{lemma:limitIdentity}
	For $\seq{(\rho_{\varepsilon_n}, \phi_{\varepsilon_n})}$ as in the previous proposition (but not necessarily satisfying the energy convergence assumption), we have for a.e.\ $t\in[0,T]$, which we suppress, and all $\xi \in C^1(\cl\Omega;\bbR^d)$ for which $\xi\cdot \vec{n} = 0$ on $\partial\Omega$
	\begin{equation} \label{eq:limitIdentity}
		\begin{aligned}
			\frac{1}{\varepsilon_n} \int_\Omega \big[f(\rho_{\varepsilon_n}) &+ \sfa \rho_{\varepsilon_n} - \rho_{\varepsilon_n}\phi_{\varepsilon_n}\big] \div\xi - \rho_{\varepsilon_n}\nabla\phi_{\varepsilon_n} \cdot \xi \,dx \\
			&= - \varepsilon_n \int_\Omega \partial_t \phi_{\varepsilon_n} \nabla\phi_{\varepsilon_n} \cdot \xi \,dx \\
			&\qquad+ \int_\Omega \frac{1}{\varepsilon_n}\big[ W(\rho) + [{\ubar\sfc}\,g]^\ast(\rho) - \rho\phi + {\ubar\sfc}\,g(\phi) \big] \div\xi + \frac{\varepsilon_n}{2}\abs{\nabla\phi_{\varepsilon_n}}^2_\sfA \div\xi\,dx  \\
			&\qquad- \varepsilon_n\int_\Omega   ((\sfA\nabla\phi_{\varepsilon_n}) \otimes \nabla\phi_{\varepsilon_n}) : (\nabla\xi)^\trans \,dx   + \frac{\varepsilon_n}{2}\int_\Omega \nabla\phi_{\varepsilon_n} \otimes \nabla \phi_{\varepsilon_n} \colon (\xi^k \partial_k A)\,dx \\
			&\qquad + \frac{1}{\varepsilon_n}\int_\Omega (\sfc - {\ubar\sfc})g(\phi_{\varepsilon_n}) \div\xi \,dx  + \frac{1}{\varepsilon_n}\int_\Omega g(\phi_{\varepsilon_n})\nabla \sfc \cdot \xi \,dx.
		\end{aligned}
	\end{equation}
	Assume \ref{hyp-c:grad-nondegeneracy}. Then there exists a constant $C>0$ depending on $\ol{\scG} \coloneq \sup_{n\in\bbN}\scG_{\varepsilon_n}[\rho^\init_{\varepsilon_n}, \phi^\init_{\varepsilon_n}]$, $T$, and independent of $n$ such that for all test vector fields $\xi \in C^1([0,T]\times\cl\Omega;\bbR^d)$ satisfying $\xi \cdot \vec{n} = 0$ on $\partial\Omega$ and $\xi \cdot \vec{n}_0 = 0$ on $(\partial\Omega_0) \cap \Omega$
	\begin{equation}\label{eq:firstVariationEstimate}
		\sup_{n\in\bbN}\abs{\frac{1}{\varepsilon_n} \int_0^T\int_\Omega \big[f(\rho_{\varepsilon_n}) + \sfa \rho_{\varepsilon_n} - \rho_{\varepsilon_n}\phi_{\varepsilon_n}\big] \div\xi - \rho_{\varepsilon_n}\nabla\phi_{\varepsilon_n} \cdot \xi \,dx\,dt} \le C \norm{\xi}_{L^2(0,T;C^1(\cl\Omega))}.
	\end{equation}
\end{lemma}

\begin{proof}[Proof of Lemma \ref{lemma:limitIdentity}.]
	Since we work with a fixed $n$ when proving the identity \eqref{eq:limitIdentity}, we omit $\varepsilon_n$.
	
	We use the definition of $W$ \eqref{eq:potentialW}, add and subtract each of ${\ubar\sfc}\,g(\phi)$, $\sfc g(\phi)$ to the quantity in brackets in the first line of \eqref{eq:limitIdentity}, and integrate by parts to obtain
	\begin{align*}
		&= \frac{1}{\varepsilon} \int_\Omega \big[ W(\rho) + [{\ubar\sfc}\,g]^\ast(\rho) - \rho\phi + {\ubar\sfc}\,g(\phi) \big] \div\xi \,dx + \frac{1}{\varepsilon}\int_\Omega (\sfc - {\ubar\sfc})g(\phi) \div\xi  + g(\phi)\nabla \sfc \cdot \xi \,dx \\
		& \qquad +  \frac{1}{\varepsilon} \int_\Omega (\sfc g'(\phi) - \rho) \nabla\phi \cdot \xi \,dx.
	\end{align*}
	Since $\phi$ is a strong solution of $\varepsilon\partial_t \phi - \varepsilon \pdiv{\sfA\nabla\phi} = -\frac{1}{\varepsilon}(\sfc g'(\phi) - \rho)$, the integral on the second line is
	\begin{align*}
		&= \frac{1}{\varepsilon}\int_\Omega \big[ W(\rho) + [{\ubar\sfc}\,g]^\ast(\rho) - \rho\phi + {\ubar\sfc}\,g(\phi) \big] \div\xi \,dx + \frac{1}{\varepsilon}\int_\Omega (\sfc - {\ubar\sfc})g(\phi) \div\xi  + g(\phi)\nabla \sfc \cdot \xi \,dx \\
		&\qquad- \varepsilon \int_\Omega \partial_t \phi \nabla\phi \cdot \xi - \pdiv{\sfA\nabla\phi}\nabla\phi\cdot\xi \,dx.
	\end{align*}
	
	We manipulate the divergence term on the second line and show it gives rise to terms appearing in the distributional formulation of the anisotropic mean-curvature. In the following, we use that on $\partial\Omega$ both $\xi \cdot \vec{n} = 0$ and $\sfA\nabla\phi\cdot \vec{n} = 0$ (no-flux boundary condition). 

	Throughout, we employ the Einstein summation convention. By several integrations by parts,
	\begin{gather*}
		\int_\Omega \pdiv{\sfA\nabla\phi} \nabla\phi \cdot \xi \,dx =	\int_\Omega \partial_i\big(\sfA^{i,j}\partial_j\phi\big) \partial_k\phi\xi^k \,dx = -\int_\Omega \sfA^{i,j} \partial_j\phi \partial_i(\partial_k\phi \xi^k) \,dx \\
		= -\int_\Omega \sfA^{i,j} \partial_j\phi \partial_k\partial_i\phi \xi^k + \sfA^{i,j} \partial_j\phi \partial_k\phi \partial_i\xi^k\,dx \\
		= -\int_\Omega \sfA^{i,j} \partial_j\phi \big[ \partial_k(\partial_i\phi \xi^k) - \partial_i\phi \partial_k\xi^k \big] \,dx - \int_\Omega \Tr\big[ ((\sfA\nabla\phi) \otimes \nabla\phi) \nabla\xi \big] \,dx \\
		= \int_\Omega \abs{\nabla\phi}_\sfA^2 \div\xi - \Tr\big[ ((\sfA\nabla\phi) \otimes \nabla\phi) \nabla\xi \big] \,dx +  \int_\Omega \partial_i\phi \partial_k\sfA^{i,j} \partial_j\phi\xi^k  + \partial_i\phi \sfA^{i,j} \partial_k\partial_j\phi \xi^k \,dx.
	\end{gather*}
	We note that the symmetry of $\sfA$ implies that the term $\partial_i\phi \sfA^{i,j} \partial_k\partial_j\phi \xi^k$ is invariant under the interchange of indices $i \leftrightarrow j$, so $\partial_i\phi \sfA^{i,j} \partial_k\partial_j\phi \xi^k = \frac{1}{2} \sfA^{i,j}\big( \partial_i\phi \partial_k\partial_j\phi + \partial_j\phi \partial_k\partial_i\phi \big) \xi^k = \frac{1}{2} \sfA^{i,j} \partial_k(\partial_i\phi \partial_j\phi) \xi^k$. Substituting this and integrating by parts once more provides
	\[
	\int_\Omega \pdiv{\sfA\nabla\phi} \nabla\phi \cdot \xi \,dx = \int_\Omega \frac{1}{2} \abs{\nabla\phi}_\sfA^2 \div\xi  - \Tr\big[ ((\sfA\nabla\phi) \otimes \nabla\phi) \nabla\xi \big]  + \frac{1}{2} \partial_i\phi \partial_k\sfA^{i,j} \partial_j\phi \xi^k \,dx,
	\]
	which establishes the identity \eqref{eq:limitIdentity}.
	
	\medskip
	For the bound \eqref{eq:firstVariationEstimate}, we control the right-hand side of \eqref{eq:limitIdentity} line-by-line. To control the second line, Cauchy-Schwarz, the energy dissipation inequality, and the upper bound on the initial energies give
	\begin{gather*}
		\abs{\varepsilon_n\int_0^T\int_\Omega \partial_t\phi_{\varepsilon_n} \nabla\phi_{\varepsilon_n} \cdot \xi \,dx\,dt} \le \int_0^T \norm{\varepsilon_n\partial_t\phi_{\varepsilon_n} \abs{\nabla\phi_{\varepsilon_n}}}_{L^1(\Omega)} \norm{\xi}_{L^\infty(\Omega)} \,dt \\
		\le \left[\int_0^T \left(\int_\Omega \varepsilon_n \big(\partial_t\phi_{\varepsilon_n}\big)^2 \,dx\right) \left(\varepsilon_n\int_\Omega \abs{\nabla\phi_{\varepsilon_n}}^2\,dx\right) \,dt\right]^{1/2}\norm{\xi}_{L^2(0,T;L^\infty(\Omega))} \\
		\le \left[\int_0^T \scD_{\varepsilon_n}(t)\,dt \right]^{1/2} \big[2\ol{\sfA}\,\ol{\scG}\big]^{1/2} \norm{\xi}_{L^2(0,T;L^\infty(\Omega))} \le [2\ol{\sfA}]^{1/2}\,\ol{\scG}\norm{\xi}_{L^2(0,T;L^\infty(\Omega))}.
	\end{gather*}
	
	For the third line in \eqref{eq:limitIdentity}, we recognize the quantities as contributions to $\scG_{\varepsilon_n}$:
	\begin{gather*}
		\abs{\int_0^T \int_\Omega \frac{1}{\varepsilon_n} \big(W(\rho_{\varepsilon_n}) + [{\ubar\sfc}\,g]^\ast(\rho_{\varepsilon_n}) - \rho_{\varepsilon_n}\phi_{\varepsilon_n} + {\ubar\sfc}\,g(\phi_{\varepsilon_n}) + \frac{\varepsilon_n}{2}\abs{\nabla\phi_{\varepsilon_n}}_\sfA^2 \big) \div\xi \,dx \,dt} \\
		\le \int_0^T \scG_{\varepsilon_n}[\rho_{\varepsilon_n}(t),\phi_{\varepsilon_n}(t)] \norm{\nabla\xi(t)}_{L^\infty(\Omega)}\,dt 
		\le\sqrt{T}\,\ol{\scG}\norm{\nabla\xi}_{L^2(0,T;L^\infty(\Omega))}. 
	\end{gather*}
	
	For the first integral in the fourth line of \eqref{eq:limitIdentity}, we fix $t\in[0,T]$, multiply and divide by $\abs{\nabla\phi_{\varepsilon_n}}_\sfA^2$, and use Cauchy-Schwarz for the Frobenius inner product $\abs{A : B} \le \abs{A}_F\abs{B}_F$ to obtain
	\begin{gather*}
		\varepsilon_n\int_\Omega \abs{(\sfA\nabla\phi_{\varepsilon_n})\otimes (\nabla\phi_{\varepsilon_n}) : (\nabla\xi)^\trans} \,dx \le \int_\Omega \abs{ \frac{(\sfA\nabla\phi_{\varepsilon_n})}{\abs{\nabla\phi_{\varepsilon_n}}_\sfA} \otimes \frac{\nabla\phi_{\varepsilon_n}} {\abs{\nabla\phi_{\varepsilon_n}}_\sfA} }_F  \abs{\nabla\xi}_F \big(\varepsilon_n\abs{\nabla\phi_{\varepsilon_n}}_\sfA^2\big) \,dx.
	\end{gather*}
	Note, for $p\in\bbR^d$, we have (using \ref{hyp-on-A:ellipticity}, \ref{hyp-on-A:regularity})
	\begin{gather*}
		\abs{\frac{\sfA p}{\abs{p}_\sfA} \otimes \frac{p}{\abs{p}_\sfA}}_F = \frac{1}{\abs{p}_\sfA^2}\Big[\sum_{i=1}^d ((\sfA p)^i)^2 \sum_{j=1}^d (p^j)^2\Big]^{1/2} = \frac{\abs{\sfA p}}{\abs{p}_\sfA}\frac{\abs{p}}{\abs{p}_\sfA} \le \frac{\ol{\sfA}}{\ubar{\sfA}},
	\end{gather*}
	so we can bound the previous integral by $C\frac{\ol{\sfA}}{\ubar{\sfA}} \norm{\nabla \xi(t)}_{L^\infty(\Omega)} \scG_{\varepsilon_n}[\rho_{\varepsilon_n}(t),\phi_{\varepsilon_n}(t)]$. Then, after an integration over $t\in[0,T]$, we obtain the estimate
	\[
	\abs{\int_0^T \varepsilon_n\int_\Omega (\sfA\nabla\phi_{\varepsilon_n})\otimes (\nabla\phi_{\varepsilon_n}) : (\nabla\xi)^\trans \,dx\,dt} \le  C\sqrt{T} \, \ol{\scG} \norm{\nabla\xi}_{L^2(0,T;L^\infty(\Omega))}.
	\]
	Estimating the second integral on the fourth line of \eqref{eq:limitIdentity} is similar and uses that $A \in C^1(\cl\Omega)$ \ref{hyp-on-A:regularity}:
	\[
	\abs{\varepsilon_n \int_0^T \int_\Omega \nabla\phi_{\varepsilon_n}\otimes \nabla\phi_{\varepsilon_n} : (\xi^k \partial_kA) \,dx\,dt} \le C\sqrt{T}\,\ol{\scG} \norm{\xi}_{L^2(0,T; L^\infty(\Omega))}.
	\]
	
	On the fifth line of \eqref{eq:limitIdentity}, the first integral has a contribution to $\scG_{\varepsilon_n}$, so we control this integral by
	\begin{gather*}
		\abs{\int_0^T \int_\Omega \frac{\sfc - {\ubar\sfc}}{\varepsilon_n} g(\phi_{\varepsilon_n}) \div\xi \,dx\,dt} 
		\le \int_0^T \scG_{\varepsilon_n}[\rho_{\varepsilon_n}(t), \phi_{\varepsilon_n}(t)] \norm{\nabla\xi(t)}_{L^\infty(\Omega)}\,dt \le  \sqrt{T}\,\ol{\scG}\norm{\nabla\xi}_{L^2(0,T;L^\infty(\Omega))}. 
	\end{gather*}
	Finally, we control the second integral on the fifth line of \eqref{eq:limitIdentity}. For this purpose, we use \ref{hyp-c:grad-nondegeneracy} to bound it with  $\frac{1}{\varepsilon_n}\int_\Omega (\sfc - {\ubar\sfc}) g(\phi_{\varepsilon_n})\,dx$ and it is at this stage that we use $\xi\cdot\vec{n}_0 = 0$ on $(\partial\Omega_0)\cap\Omega$. Recall $\chi_{\Omega_0} \nabla\sfc = 0$ a.e., so what remains is an estimate over the region $\cl\Omega \setminus \Omega_0$. Let $U$ be the neighborhood of $(\partial\Omega_0) \cap \Omega$ from \ref{hyp-c:grad-nondegeneracy}, and then at a.e.\ $t\in[0,T]$
	\begin{equation} \label{eq:obstacleEstimate}
		\begin{aligned}
			\frac{1}{\varepsilon_n}\int_\Omega g(\phi_{\varepsilon_n})\abs{\nabla\sfc \cdot \xi} \,dx \le \frac{K_4\norm{\xi}_{C^1(\cl U)}}{\varepsilon_n} \int_U (\sfc - {\ubar\sfc})g(\phi_{\varepsilon_n})\,dx 
			\qquad+ \frac{1}{\varepsilon_n} \int_{(\Omega \setminus \Omega_0) \setminus U} g(\phi_{\varepsilon_n}) (\sfc - {\ubar\sfc}) \abs{\nabla\ln(\sfc-{\ubar\sfc})}\abs{\xi} \,dx.
		\end{aligned}
	\end{equation}
	The first term on the right-hand side is now controlled by $\scG_{\varepsilon_n}$. For the second term on the right-hand side, we note $\ln(\sfc - {\ubar\sfc})$ is Lipschitz on $(\Omega \setminus \Omega_0) \setminus U$ since $\sfc$ is bounded away from ${\ubar\sfc}$ on this region and $\sfc$ is Lipschitz, so $\abs{\nabla\ln(\sfc - {\ubar\sfc})} \in L^\infty((\Omega \setminus \Omega_0)\setminus U)$. Thus  \eqref{eq:obstacleEstimate} is bounded by
	$
	C \scG_{\varepsilon_n}[\rho_{\varepsilon_n}(t), \phi_{\varepsilon_n}(t)] \norm{\xi(t)}_{C^1(\cl\Omega)}, 
	$
	and then integrating over $t\in[0,T]$ provides
	$
	\abs{\frac{1}{\varepsilon_n}\int_0^T\int_\Omega g(\phi_{\varepsilon_n})\nabla\sfc \cdot \xi \,dx \,dt} \le C \sqrt{T}\,\ol{\scG} \norm{\xi}_{L^2(0,T;C^1(\cl\Omega))}.
	$
	This completes the proof of \eqref{eq:firstVariationEstimate}.
\end{proof}

\begin{proof}[Proof of Proposition \ref{prop:Reshetnyak-timeDep}]
	Using the previous lemma, we instead pass to the limit in the right-hand side of \eqref{eq:limitIdentity}. Throughout, we use that the limit $(\rho,\phi)$ satisfies $\phi = \frac{\phi_+}{\rho_+} \rho$.
	
	\medskip 
	For the second line in \eqref{eq:limitIdentity}, one uses the same reasoning as for the elliptic-parabolic system  \cite[Proposition 8.1]{Mellet.Rozowski_2024_VolumepreservingMeancurvatureFlow} or the convergence of strong solutions of Allen-Cahn-type equations to $BV$ solutions of mean-curvature flow. We omit the proof, which is the analogous to the one given in \cite[Proposition 8.1]{Mellet.Rozowski_2024_VolumepreservingMeancurvatureFlow}. This gives 
	\[
	\lim_{n\to\infty} \int_0^T \int_\Omega \varepsilon_n  \partial_t\phi_{\varepsilon_n} \nabla\phi_{\varepsilon_n} \cdot \xi \,dx\,dt = \gamma\int_0^T \int_\Omega V\xi \cdot d\nabla\phi(t)\,dt = -\gamma_0 \int_0^T \int_\Omega V \vec{\nu} \cdot \xi \,d\vert\nabla\rho(t)\vert \,dt.
	\]
	
	\medskip
	The quantities in the third line of \eqref{eq:limitIdentity} are the contributions to $\scG_{\varepsilon_n}$, and their convergence is handled by the equipartition of the energy. In particular, Lemma \ref{lemma:equipartitionOfEnergy-timeIndependent}-\ref{lemma:equalPotentialEnergies} and Lebesgue dominated convergence permits us to replace $W(\rho_{\varepsilon_n}) + [{\ubar\sfc}\,g]^\ast(\rho_{\varepsilon_n}) - \rho_{\varepsilon_n}\phi_{\varepsilon_n} + {\ubar\sfc}\,g(\phi_{\varepsilon_n})$ by $W_\ast(\phi_{\varepsilon_n})$ up to some asymptotically negligible error, and then the second equality of \eqref{eq:equipartition-timeDependent-convergences} in Corollary \ref{cor:equipartition-timeDependent} provides
	\begin{gather*}
		\lim_{n\to\infty} \frac{1}{\varepsilon_n}\int_0^T\int_\Omega \big[W(\rho_{\varepsilon_n}) + [{\ubar\sfc}\,g]^\ast(\rho_{\varepsilon_n}) - \rho_{\varepsilon_n}\phi_{\varepsilon_n} + {\ubar\sfc}\,g(\phi_{\varepsilon_n})\big] \div\xi \,dx \,dt 
		= \lim_{n\to\infty} \frac{1}{\varepsilon_n}\int_0^T\int_\Omega W_\ast(\phi_{\varepsilon_n}) \div\xi \,dx \,dt \\
		= \frac{\gamma}{2} \int_0^T \int_\Omega \abs{\vec{\nu}}_\sfA\div\xi\,d\vert\nabla\phi(t)\vert\,dt  
		= \frac{\gamma_0}{2} \int_0^T \int_\Omega \div\xi \abs{\vec{\nu}}_\sfA \,d\vert\nabla\rho(t)\vert \,dt.
	\end{gather*}
	
	The convergence of the Dirichlet-like energy is handled directly by the equipartition of the energy, the third equality in \eqref{eq:equipartition-timeDependent-convergences} of Corollary \ref{cor:equipartition-timeDependent}.
	
	\medskip
	The convergence of each of the integrals in the fourth line of \eqref{eq:limitIdentity} is handled by a version of the Reshetnyak continuity theorem adapted to the anisotropic setting \cite[Lemma 3.7]{Cicalese.Nagase.ea_2010_GibbsThomsonRelation}. We recall that the classical Reshetnyak continuity theorem \cite[Theorem 2.39]{Ambrosio.Fusco.ea_2000_FunctionsBoundedVariation} requires the convergence of the total variations $\int_\Omega d\vert\nabla u_n\vert \to \int_\Omega d\vert\nabla u\vert$, which we do not have in our setting. Instead, we have the convergence of some anisotropic total variations. Indeed, the last equality of Lemma \ref{lemma:equipartitionOfEnergy-timeIndependent}-\ref{lemma:equalEnergies} provides
	\[
	\text{for a.e.\ } t\in[0,T] \qquad	\lim_{n\to\infty} \int_\Omega \abs{\nabla\psi_{\varepsilon_n}(t,x)}_{\sfA(x)}\,dx = \gamma \int_\Omega \abs{\vec{\nu}(t,x)}_{\sfA(x)} \,d \abs{\nabla\phi(t)}(x).
	\]
	
	We show only the convergence of the first term on the fourth line of \eqref{eq:limitIdentity} since the other term is analogous. For a fixed but suppressed time $t\in[0,T]$,  
	\[
	\varepsilon_n\int_\Omega \Tr\big[((\sfA\nabla\phi_{\varepsilon_n}) \otimes \nabla\phi_{\varepsilon_n}) (\nabla \xi)\big] \,dx = \int_\Omega \frac{\sfA\nabla\phi_{\varepsilon_n}}{\abs{\nabla\phi_{\varepsilon_n}}_\sfA} \otimes \frac{\nabla\phi_{\varepsilon_n}}{\abs{\nabla\phi_{\varepsilon_n}}_\sfA} \colon (\nabla \xi)^\trans (\varepsilon_n \abs{\nabla\phi_{\varepsilon_n}}_\sfA^2) \,dx.
	\]
	Since each of the components of $\sfA$ and $\nabla \xi$ are continuous up to the boundary, and because the $\sfA(x)$-norm is, uniformly in $x$, comparable to the standard $2$-norm \ref{hyp-on-A:regularity}, \ref{hyp-on-A:ellipticity}, we have that 
	\[
	\frac{\sfA\nabla\phi_{\varepsilon_n}}{\abs{\nabla\phi_{\varepsilon_n}}_\sfA} \otimes \frac{\nabla\phi_{\varepsilon_n}}{\abs{\nabla\phi_{\varepsilon_n}}_\sfA} \colon (\nabla\xi)^\trans \text{ is bounded in } L^\infty(\left]0,T\right[\times\Omega) \text{ uniformly in } n.
	\] 
	By equipartition of the energy (Lemma \ref{lemma:equipartitionOfEnergy-timeIndependent}-\ref{lemma:equipartitionFeps}), we have $\varepsilon_n\abs{\nabla \phi_{\varepsilon_n}}_\sfA^2 - \abs{\nabla\psi_{\varepsilon_n}}_\sfA \to 0$ in $L^1(\left]0,T\right[ \times \Omega)$, so combining the above bound with this limit shows, up to some asymptotically negligible error, 
	\begin{gather*}
		\int_\Omega \frac{\sfA\nabla\phi_{\varepsilon_n}}{\abs{\nabla\phi_{\varepsilon_n}}_\sfA} \otimes \frac{\nabla\phi_{\varepsilon_n}}{\abs{\nabla\phi_{\varepsilon_n}}_\sfA} \colon (\nabla \xi)^\trans (\varepsilon_n \abs{\nabla\phi_{\varepsilon_n}}_\sfA^2) \,dx 
		= \int_\Omega \frac{\sfA\nabla\phi_{\varepsilon_n}}{\abs{\nabla\phi_{\varepsilon_n}}_\sfA} \otimes \frac{\nabla\phi_{\varepsilon_n}}{\abs{\nabla\phi_{\varepsilon_n}}_\sfA} \colon (\nabla \xi)^\trans \abs{\nabla\psi_{\varepsilon_n}}_\sfA \,dx + o_{n\to\infty}(1).
	\end{gather*}
	We now reintroduce the auxiliary function $\psi_{\varepsilon_n}\coloneq F(\phi_{\varepsilon_n})$, where $F'(v) = \sqrt{2 W_\ast(v)} \ge 0$. Using that $F'\ge0$, we find 
	\[
	= \int_\Omega \frac{\sfA\nabla \psi_{\varepsilon_n}}{\abs{\nabla\psi_{\varepsilon_n}}_\sfA} \otimes \frac{\nabla\psi_{\varepsilon_n}}{\abs{\nabla\psi_{\varepsilon_n}}_\sfA} \colon (\nabla\xi)^\trans \abs{\nabla\psi_{\varepsilon_n}}_\sfA\,dx + o_{n\to\infty}(1).
	\]
	Since phase separation (Proposition \ref{prop:phaseSeparation}) and equipartition of the energy (Lemma \ref{lemma:equipartitionOfEnergy-timeIndependent}-\ref{lemma:equalEnergies}) give
	\begin{gather*}
		\text{for each } t \in [0,T] \quad \lim_{n\to\infty}\psi_{\varepsilon_n} = \gamma \phi \text{ in } L^1(\Omega) \,\text{ and }\,
		\lim_{n\to\infty}\int_\Omega \abs{\nabla\psi_{\varepsilon_n}}_\sfA\,dx = \gamma\int_\Omega \abs{\vec{\nu}}_\sfA\,d\vert\nabla\phi\vert,
	\end{gather*} 
	the Reshetnyak continuity-type theorem \cite[Lemma 3.7]{Cicalese.Nagase.ea_2010_GibbsThomsonRelation} provides, for each $t \in [0,T]$, 
	\[
	\begin{aligned}
		\lim_{n\to\infty }\varepsilon_n\int_\Omega \Tr\big[((\sfA\nabla\phi_{\varepsilon_n}) \otimes \nabla\phi_{\varepsilon_n}) \nabla \xi\big] \,dx 
		= \gamma \int_\Omega  \frac{A\vec{\nu}}{\abs{\vec{\nu}}_\sfA} \otimes  \frac{\vec{\nu}}{\abs{\vec{\nu}}_\sfA} \colon (\nabla \xi)^\trans \abs{\vec{\nu}}_\sfA\,d\vert\nabla\phi\vert.
	\end{aligned}
	\]
	Lebesgue dominated convergence yields the convergence after integration over $t \in [0,T]$.
	
	\medskip
	It remains to show that the fifth line of \eqref{eq:limitIdentity} vanishes as $n\to\infty$. In particular, strongly in $L^1(0,T)$, 
	\[
	\begin{gathered}
		\lim_{n\to\infty} \frac{1}{\varepsilon_n}\int_\Omega (\sfc - {\ubar\sfc})g(\phi_{\varepsilon_n})\div\xi \,dx = 0 \qquad 
		\text{and} \qquad \lim_{n\to\infty} \frac{1}{\varepsilon_n} \int_\Omega g(\phi_{\varepsilon_n}) \nabla \sfc \cdot \xi \,dx = 0.
	\end{gathered}
	\]
	The first limit is an immediate consequence of equipartition of the energy, Lemma \ref{lemma:equipartitionOfEnergy-timeIndependent}-\ref{lemma:energyObstacle}. For the second limit, the nondegeneracy condition \ref{hyp-c:grad-nondegeneracy} and $\xi\cdot \vec{n}_0 = 0$ on $\partial\Omega_0$ imply $\abs{\nabla\sfc\cdot\xi} \le K_4 \norm{\xi}_{C^1(\cl U)}(\sfc - {\ubar\sfc})$ on some neighborhood $U$ of $\partial\Omega_0$ as in \eqref{eq:obstacleEstimate}, and then the convergence is handled by Lemma \ref{lemma:equipartitionOfEnergy-timeIndependent}-\ref{lemma:energyObstacle}.
\end{proof}

\begin{remark}
	The Reshetnyak continuity-type theorem  \cite[Lemma 3.7]{Cicalese.Nagase.ea_2010_GibbsThomsonRelation}  is proven for a sequence of $C^1(\Omega)$ functions $\seq{v_{\varepsilon_n}}$ that converges strongly in $L^1(\Omega)$ to some $v_0 \in BV(\Omega)$ and for which some inhomogeneous, anisotropic total variations converge $
	\lim_{n\to\infty} \int_\Omega \phi(x,\nabla v_{\varepsilon_n})\,dx = \int_\Omega \phi(x, \vec{\nu}_0) \,d\vert\nabla v_0\vert$, 
	where $\vec{\nu}_0$ is the Radon-Nikodym derivative of $\nabla v_0$ w.r.t.\ its total variation $\abs{\nabla v_0}$ and $\phi$ is a smooth, strictly convex Finsler norm, of which $(x,p) \mapsto \abs{p}_{\sfA(x)}$ is an example. For more on Finsler norms and their use in the variational setting of anisotropic mean-curvature flow, see,  \cite{Bellettini.Paolini_1996_AnisotropicMotionMean}, \cite[Section 2]{Cicalese.Nagase.ea_2010_GibbsThomsonRelation}, or \cite{Laux.Stinson.ea_2024_DiffuseinterfaceApproximationWeak}.  They derive $
	\lim_{n\to\infty} \int_\Omega F(x,\nabla v_{\varepsilon_n})\,dx = \int_\Omega F(x,\vec{\nu}_0) \,d\vert\nabla v_0\vert$
	for $F \in C(\Omega \times \bbR^d)$ that is positively $1$-homogeneous in the second variable and vanishes off of some compact set in the first variable. The vanishing condition is not necessary for their proof, and the result can be extended by density to sequences $\seq{v_{\varepsilon_n}} \subset W^{1,1}(\Omega)$ provided $F$ is also Lipschitz in its second variable uniformly in its first variable, which is the case we use the result. 
\end{remark}

\subsection{Derivation of the limiting pressure equation (Derivation of \texorpdfstring{\eqref{eq:weakFormFBCs}}{\eqref{eq:weakFormFBCs}} in Theorem \texorpdfstring{\ref{thm:main}-\ref{thm:conditionalConvergence}}{\ref{thm:main}-\ref{thm:conditionalConvergence}})} \label{subsec:derivationOfPressureEq}

Recall the weak formulation of the mass flux $\vec{j}_{\varepsilon_n} \coloneq \rho_{\varepsilon_n}\vec{v}_{\varepsilon_n}$ when viewing the evolution equation for the density as a continuity equation:
\begin{equation} \label{eq:weakMassFlux}
	\int_0^T \int_\Omega \vec{j}_{\varepsilon_n} \cdot \xi \,dx\,dt = \frac{1}{\varepsilon_n\alpha_0}\int_0^T \int_\Omega \rho_{\varepsilon_n}\nabla\phi_{\varepsilon_n} \cdot \xi + \big[\rho_{\varepsilon_n} f'(\rho_{\varepsilon_n}) - f(\rho_{\varepsilon_n}) \big] \div\xi  \,dx \,dt. 
\end{equation}
The estimate on the mass flux, Lemma \ref{lemma:aPrioriEst-continuityEq}-\ref{lemma:massFluxL2L1},  provides that the left-hand side is bounded over $n$:
\[
\abs{\int_0^T \int_\Omega \vec{j}_{\varepsilon_n}\cdot\xi\,dx\,dt} \le \big( \sup_{n\in\bbN} \norm{\vec{j}_{\varepsilon_n}}_{L^2(0,T;L^1(\Omega))} \big) \norm{\xi}_{L^2(0,T;L^\infty(\Omega;\bbR^d))},
\] 
and Proposition \ref{prop:continuityEq}-\ref{prop:continuityEq-convergence} shows $\seq{\vec{j}_{\varepsilon_n}}$ converges weakly to some $\vec{j} \in L^2(0,T;L^{\frac{2m}{m+1}}(\Omega); \bbR^d)$ that can be written as $\vec{j} = \rho \vec{v}$ for some velocity field $\vec{v} \in L^2(\left]0,T\right[\times \Omega, d\rho dt; \bbR^d)$. Consequently, the right-hand side also has a limit as $n\to\infty$, although the individual terms are not immediately bounded. 

To circumvent this, we introduce an approximate pressure
\[
\pi_{\varepsilon_n}(t,x) \coloneq  \rho_{\varepsilon_n}\frac{ f'(\rho_{\varepsilon_n}(t,x)) + \sfa -\phi_{\varepsilon_n}(t,x)}{\varepsilon_n},
\]
so the above weak formulation of the flux \eqref{eq:weakMassFlux} becomes
\begin{equation} \label{eq:approxPressureEq}
	\begin{gathered}
		\int_0^T \int_\Omega \vec{j}_{\varepsilon_n} \cdot \xi \,dx\,dt = -\frac{1}{\varepsilon_n\alpha_0}\int_0^T \int_\Omega \big[ f(\rho_{\varepsilon_n}) +  \sfa\rho_{\varepsilon_n} - \rho_{\varepsilon_n}\phi_{\varepsilon_n} \big] \div\xi - \rho_{\varepsilon_n}\nabla\phi_{\varepsilon_n} \cdot \xi\,dx\,dt \\
		+ \frac{1}{\alpha_0}\int_0^T \int_\Omega \pi_{\varepsilon_n} \div\xi  \,dx \,dt,
	\end{gathered}
\end{equation}
where we use $\xi \cdot \vec{n} = 0$ on $\partial\Omega$. Our final aim is to pass to the limit $n\to\infty$ in \eqref{eq:approxPressureEq} and thus derive \eqref{eq:weakFormFBCs}, the final equation in our weak formulation of \eqref{eq:HS-STKU}. Lemma \ref{lemma:limitIdentity} shows the first term on the right-hand side can be estimated in terms of $\norm{\xi}_{L^2(0,T; C^1(\cl\Omega;\bbR^d))}$, and Proposition \ref{prop:Reshetnyak-timeDep} shows it converges to the undercooling and curvature terms. Thus, the last term converges, but to identify the limit we need to bound  $\seq{\pi_{\varepsilon_n}}$. In the following, we use duality to obtain such a bound and pass to the limit, first in the case of $\sfc$ a positive constant function and second in the case of $\sfc$ satisfying \ref{hyp-c:grad-nondegeneracy}, \ref{hyp-c:interiorBoundary} but for the particular case of porous medium diffusion, i.e.,  $f(u) = u^m/(m-1) + \iota_{\left[0,+\infty\right[}(u)$ with $m>q'$.

\subsubsection{Spatially homogeneous destruction} \label{subsubsec:spatiallyHomogeneousLimit}
In this section, we suppose $\sfc \equiv {\underline\sfc}>0$, so there is no confinement effect in the limit since $\Omega_0 \coloneq \set{\sfc = {\ubar\sfc}} = \cl\Omega$. As in proofs of \cite[Theorem 3.2]{Jacobs.Kim.ea_2021_WeakSolutionsMuskat} or \cite[Section 5.1]{Kim.Mellet.ea_2023_DensityconstrainedChemotaxisHeleShaw}, given $\varphi \in L^2(0,T;C^{0,s}(\cl\Omega))$, with $0<s<1$, we solve for a.e.\ $t\in[0,T]$ the potential problem
\[
\begin{cases}
	\lap u(t) = \varphi(t) - (\varphi(t))_{\Omega} & \text{in } \Omega,	\\
	\nabla u(t) \cdot \vec{n} = 0	&	\text{on } \partial\Omega,
\end{cases}
\]
where $(\varphi(t))_\Omega \coloneq \frac{1}{\abs{\Omega}}\int_\Omega\varphi(t,x) \,dx$ is the mean of $\varphi(t)$ over $\Omega$. We take the mean-zero solution, so by  Schauder estimates \cite{Nardi_2015_SchauderEstimateSolutions}, we have $\norm{u}_{L^2(0,T; C^{2,s}(\cl\Omega))} \lesssim \norm{\varphi - (\varphi)_\Omega}_{L^2(0,T; C^{0,s}(\cl\Omega))}$. We set $\xi \coloneq \nabla u \in L^2(0,T; C^{1,s}(\cl\Omega))$, which satisfies $\div\xi = \lap u = \varphi - (\varphi)_\Omega$ on $\left]0,T\right[\times\Omega$ and $\xi\cdot \vec{n} = 0$ on $\left]0,T\right[\times \partial\Omega$. 

We test \eqref{eq:approxPressureEq} with this particular $\xi$. Let $p_{\varepsilon_n}(t,x) \coloneq \pi_{\varepsilon_n}(t,x) - (\pi_{\varepsilon_n}(t))_\Omega$, which satisfies $(p_{\varepsilon_n}(t))_\Omega = 0$ for a.e.\ $t\in[0,T]$. Since $\varphi(t) - (\varphi(t))_\Omega$ has mean-zero over $\Omega$, the weak formulation of the approximate pressure equation \eqref{eq:approxPressureEq} is equivalently
\begin{gather*}
	\frac{1}{\alpha_0}\int_0^T\int_\Omega p_{\varepsilon_n}(\varphi - (\varphi)_\Omega)\,dx\,dt \\
	= \int_0^T \int_\Omega \vec{j}_{\varepsilon_n}\cdot\xi \,dx\,dt 
	+ \frac{1}{\varepsilon_n\alpha_0}\int_0^T \int_\Omega \big[ f(\rho_{\varepsilon_n}) + \sfa\rho_{\varepsilon_n} - \rho_{\varepsilon_n} \phi_{\varepsilon_n}\big] \div\xi - \rho_{\varepsilon_n}\nabla \phi_{\varepsilon_n} \cdot \xi \,dx \,dt.
\end{gather*}
Thus, combining the estimate on the mass flux (Lemma \ref{lemma:aPrioriEst-continuityEq}-\ref{lemma:massFluxL2L1}), the estimate on the first variation (\eqref{eq:firstVariationEstimate} from Lemma \ref{lemma:limitIdentity}), and the Schauder estimate shows
\begin{equation*} \label{eq:pressureDivxi}
	\abs{\int_0^T \int_\Omega p_{\varepsilon_n} (\varphi - (\varphi)_\Omega ) \,dx\,dt} \lesssim \norm{\xi}_{L^2(0,T; C^1(\cl\Omega))} \coloneq \norm{\nabla u}_{L^2(0,T; C^1(\cl\Omega))} \lesssim \norm{\varphi - (\varphi)_\Omega}_{L^2(0,T; C^{0,s}(\cl\Omega))},
\end{equation*}
so $\seq{p_{\varepsilon_n}} \text{ is bounded in } L^2(0,T; \cX^\ast)$, where $\cX \coloneq \cset{\zeta \in C^{0,s}(\cl\Omega)}{(\zeta)_\Omega = 0}$ 
is a Banach subspace of $C^{0,s}(\cl\Omega)$. Consequently,  $\seq{p_{\varepsilon_n}}$ admits a subsequence converging  weakly-$\ast$ in $L^2(0,T; \cX^\ast)$. Then, after passing to a subsequence, we may pass to the limit in \eqref{eq:approxPressureEq} with $p_{\varepsilon_n}$ in place of $\pi_{\varepsilon_n}$ for any $\xi \in L^2(0,T; C^{1,s}(\cl\Omega))$ that is, for a.e.\ $t\in[0,T]$, tangential to $\partial\Omega$.

\subsubsection{Nonhomogeneous destruction} \label{subsubsec:derivationOfPressure-inhomogeneousDestruction}
Throughout this section, we suppose $\sfc$ satisfies \ref{hyp-c:grad-nondegeneracy}, \ref{hyp-c:interiorBoundary}  and  $f(u) = u^m/(m-1) + \iota_{\left[0,+\infty\right[}(u)$. We proceed in several steps. We remark that the only step that makes use of this particular choice of $f$ is Step 3.

\textbf{Step 1:} (Rewriting the pressure) For clarity, we introduce the functions $P,R$ on $\bbR\times\bbR$
\begin{equation} \label{eq:FenchelDiscrepancy}
	P(u,v) \coloneq f(u) - u(v-\sfa) + f^\ast(v-\sfa) \quad \text{and} \quad R(u,v) \coloneq [{\ubar\sfc}\,g]^\ast(u) - uv + {\ubar\sfc}\,g(v),
\end{equation}
which are each nonnegative by the Fenchel-Young inequality. Both of $\varepsilon^{-1}\int_\Omega P(\rho_\varepsilon,\phi_\varepsilon)\,dx$ and  $\varepsilon^{-1}\int_\Omega R(\rho_\varepsilon,\phi_\varepsilon)\,dx$ are contributions to the energy $\scG_\varepsilon[\rho_\varepsilon,\phi_\varepsilon]$ (cf.\ \eqref{eq:GepsModicaMortola} and \eqref{eq:augmentedEnergy}, resp.), so in particular the energy dissipation inequality \eqref{eq:GDissipationIneq} and well-preparedness of the initial data \ref{hyp:well-preparedG} provide
\begin{equation} 
	\seq{\frac{1}{\varepsilon_n}P(\rho_{\varepsilon_n},\phi_{\varepsilon_n})} \text{ is bounded in } L^\infty(0,T; L^1(\Omega)). \label{eq:boundedFenchelDiscrepancy} 
\end{equation}
Equipartition of the energy (Lemma \ref{lemma:equipartitionOfEnergy-timeIndependent}-\ref{lemma:equalPotentialEnergies}) provides 
\begin{equation}
	\frac{1}{\varepsilon_n}P(\rho_{\varepsilon_n},\phi_{\varepsilon_n}) \xrightarrow{n\to\infty} 0 \text{ strongly in } L^1(\left]0,T\right[\times\Omega) \label{eq:vanishingFenchelDiscrepancy}.
\end{equation}
The approximate pressure $\pi_{\varepsilon}$ is related to $\partial_u P$, the partial subdifferential of $P$ w.r.t.\ its first variable while the second variable is fixed. To isolate the multi-valued part of $\partial f$, we define $\wt{f} \colon \bbR \to \left[0,+\infty\right[$ to be $\wt{f}(u) \coloneq f(u)$ if $u \ge 0$ and $\wt{f}(u) \coloneq f(-u)$ if $u < 0$, so $f = \wt{f} + \iota_{\left[0,+\infty\right[}$. By \ref{hyp-f:regularity}, $\wt{f}$ is $C^1$ and strictly convex, so
\[
\partial f(u) = 
\begin{cases}
	\wt{f}'(u)					& \text{if } u > 0, \\
	\left]-\infty,0\right] 		& \text{if } u = 0, \\
	\emptyset 					& \text{if } u < 0 
\end{cases}
\quad\text{and}\quad 
\partial_u P(u,v) =
\begin{cases}
	\wt{f}'(u) - (v-\sfa) 				& \text{if } u>0, \\
	-v + \left]-\infty,\sfa\right] 		& \text{if } u=0, \\
	\emptyset 						& \text{if } u < 0.
\end{cases}
\]
The nonnegativity of $\rho_\varepsilon$ suppresses the multi-valued part of $\partial_u P(\rho_\varepsilon,\phi_\varepsilon)$ in the product $\rho_\varepsilon \partial_u P(\rho_\varepsilon,\phi_\varepsilon)$, which allows us to rewrite the approximate pressure as 
\begin{equation}\label{eq:rewritingPressure}
	\begin{gathered}	
		\pi_{\varepsilon_n} = \rho_{\varepsilon_n} \frac{\partial_u P(\rho_{\varepsilon_n}, \phi_{\varepsilon_n} )}{\varepsilon_n} 
		= \frac{1}{\varepsilon_n}P(\rho_{\varepsilon_n}, \phi_{\varepsilon_n}) \\
		+ \Big(\frac{ \sqrt{\rho_{\varepsilon_n} f'(\rho_{\varepsilon_n}) - f(\rho_{\varepsilon_n})} + \sqrt{f^\ast(\phi_{\varepsilon_n} - \sfa)}}{\sqrt{\varepsilon_n}} \Big) \Big( \frac{ \sqrt{\rho_{\varepsilon_n}f'(\rho_{\varepsilon_n}) - f(\rho_{\varepsilon_n})} - \sqrt{f^\ast(\phi_{\varepsilon_n}-\sfa)}}{\sqrt{\varepsilon_n}} \Big).
	\end{gathered}
\end{equation}
The last term on the r.h.s.\ of the first line is bounded over $n\in\bbN$ (cf.\ \eqref{eq:boundedFenchelDiscrepancy}) it vanishes as $n\to\infty$ (cf.\ \eqref{eq:vanishingFenchelDiscrepancy}), so we only need to study the bottom line. We note  $\rho_{\varepsilon_n}f'(\rho_{\varepsilon_n}) - f(\rho_{\varepsilon_n})$ is nonnegative because $f$ is convex and vanishes at $0$. 

The aim of the next two steps is to show that $\pi_{\varepsilon_n}(\sfc - {\ubar\sfc})^{1/2}$ vanishes in $L^2(0,T; L^1(\Omega))$. We will accomplish this by showing, after weighting by $(\sfc - {\ubar\sfc})^{1/2}$, that the first factor in the difference of squares is bounded in $L^\infty(0,T; L^2(\Omega))$ and that the second factor, without a weight, vanishes in $L^2(0,T; L^2(\Omega))$.

\medskip

\textbf{Step 2:} (Weighted density estimate) Consider the function 
\[
\ol{W}(x,u) \coloneq f(u) - [\sfc(x)g]^\ast(u) + \sfa u = W(u) + [{\ubar\sfc}\,g]^\ast(u) - [\sfc(x)g]^\ast(u) \ge 0,
\]
where the inequality is a consequence of nonnegativity of $W$ and the order-reversing property of the Legendre transform: $\sfc(x) g \ge {\ubar\sfc}\,g$ implies $[\sfc(x)g]^\ast \le [{\ubar\sfc} \, g]^\ast$. On the region $\Omega_0$, the function $\ol{W}(x,\cdot)$ is a double-well potential since it clearly coincides with $W$. If $x \in \cl\Omega \setminus \Omega_0$, then $\ol{W}(x,u) > W(u)$ for $u > 0$ and they coincide (vanish) only when $u=0$. We first show there exists a constant $C>0$ such that 
\begin{equation} \label{eq:olWCoercive}
	\ol{W}(x,u) \ge C(\sfc(x) - {\ubar\sfc})u^m \text{ for all } x \in \cl\Omega \text{ and } u \ge 0.
\end{equation}
The coercivity of $W$ (Lemma \ref{lemma:potentialW}-\ref{lemma:Wcoercive}) provides the estimate $\ol{W}(x,u) \gtrsim u^m \gtrsim (\sfc(x) - {\ubar\sfc})u^m$ for $u>R_3$, so it remains to obtain the estimate for $0 \le u \le 1/R_1$ and $1/R_1 \le u \le R_3$.

Nonnegativity of $W$ and an easy calculation using the first-order convexity inequality with $g^\ast$ give for all $x \in \cl\Omega$ and $u \ge 0$:
\begin{equation*}
	\ol{W}(x,u) \ge  [{\ubar\sfc}\, g]^\ast(u) - [\sfc(x) g]^\ast(u) 
	\ge (\sfc(x) - {\ubar\sfc}) \Big[ g^\ast\big(\frac{u}{{\ubar\sfc}}\big) +  (g^\ast)'\big(\frac{u}{{\ubar\sfc}}\big)\frac{u}{{\ubar\sfc}} \Big].
\end{equation*}
The term $(g^\ast)'(u/{\ubar\sfc})(u/{\ubar\sfc})$ in brackets is nonnegative. We estimate $g^\ast(\cdot/{\ubar\sfc})$ from below. By \ref{hyp-g:growthNear0}, we have for $0 \le u \le {\ubar\sfc}/R_1$ that $g^\ast(u/{\ubar\sfc}) \gtrsim u^m$, which gives the estimate for small $u$. In the intermediate range $1/R_1 \le u \le R_3$, since each of $u \mapsto g^\ast(u/{\ubar\sfc})$ and $u\mapsto u^m$ is continuous and bounded away from zero, there exists a constant such that for all $1/R_1 \le u \le R_3$ we have $g^\ast(u/{\ubar\sfc}) \gtrsim u^m$. This proves \eqref{eq:olWCoercive}. 

Let $\ol{R}(x,u,v)$ be defined in the same manner as $R(u,v)$ (cf.\ \eqref{eq:FenchelDiscrepancy}), but with $\sfc(x)$ in place of ${\ubar\sfc}$. Then, $\ol{R} \ge 0$, and we have for all $x\in\cl\Omega$, $u \ge 0$, and $v\in\bbR$
\[
\ol{W}(x,u) = W(u) + R(u,v) - \ol{R}(x,u,v) + (\sfc(x) - {\ubar\sfc})g(v) \le W(u) + R(u,v) + (\sfc(x) - {\ubar\sfc})g(v).
\]
After evaluating the above at $u = \rho_{\varepsilon_n}(t,x)$ and $v = \phi_{\varepsilon_n}(t,x)$ and integrating over $\Omega$, we find the right-hand side is bounded above by $\scG_{\varepsilon_n}[\rho_{\varepsilon_n}(t),\phi_{\varepsilon_n}(t)]$, so combining this with the estimate \eqref{eq:olWCoercive}, the energy dissipation inequality, and the well-preparedness of the initial data yields the weighted density estimate
\begin{equation} \label{eq:weightedDensityEstimate}
	0 \le \sup_{n\in\bbN}\esssup_{t\in[0,T]}\frac{1}{\varepsilon_n}\int_\Omega (\sfc(x) - {\ubar\sfc})\rho_{\varepsilon_n}(t,x)^m \,dx 
	\le \sup_{n\in\bbN} \esssup_{t\in[0,T]} \frac{1}{\varepsilon_n}\int_\Omega \ol{W}(x,\rho_{\varepsilon_n}(t,x))\,dx \le  \ol{\scG}.
\end{equation}

\medskip
\textbf{Step 3:} (Weighted approximate pressure convergence) The aim of this step is to show 
\begin{equation} \label{eq:weightedPressureVanishing}
	\lim_{n\to\infty}  \pi_{\varepsilon_n}(\sfc - {\ubar\sfc})^{1/2} = 0  \text{ strongly in } L^2(0,T; L^1(\Omega)).
\end{equation}

We use the identity \eqref{eq:rewritingPressure} derived in Step 1. For the first term, we combine $(\sfc - {\ubar\sfc})^{1/2} \in L^\infty(\Omega)$ with \eqref{eq:boundedFenchelDiscrepancy} to see that $L^\infty(0,T;L^1(\Omega))$-bound holds for $P(\rho_{\varepsilon_n},\phi_{\varepsilon_n})(\sfc - {\ubar\sfc})^{1/2}/\varepsilon_n$. Further, since $\int_\Omega P(\rho_{\varepsilon_n},\phi_{\varepsilon_n})(\sfc - {\ubar\sfc})^{1/2}/\varepsilon_n\,dx$ is bounded in $L^\infty(0,T)$ and vanishes in $L^1(0,T)$, it vanishes, in particular, in $L^2(0,T)$. It remains to address the difference of squares in \eqref{eq:rewritingPressure}.

We examine the nonlinearity in the first factor. For all $u\ge0$ and $v \in \bbR$
\[
\Big( \sqrt{uf'(u) - f(u)} + \sqrt{f^\ast(v-\sfa)}  \Big)^2 \le 2(uf'(u) - f(u)) + 2f^\ast(v-\sfa) \le 2(uf'(u) - f(u)) + 2g(v).
\]
The definition of $\sfa$ (cf.\ \ref{hyp:compatibilityConditionF&G}) and the order-reversing property of the Legendre transform give $f^\ast(\cdot-\sfa) \le g$, which yields the last inequality. Since $f$ is of porous medium-type, we have $uf'(u) - f(u) = u^m$ for all $u\ge0$. Then \eqref{eq:weightedDensityEstimate} and equipartition of the energy (Lemma \ref{lemma:equipartitionOfEnergy-timeIndependent}-\ref{lemma:energyObstacle}) give
\begin{equation} \label{eq:differenceOfSquares-Bounded}
	\begin{gathered}
		\sup_{n\in\bbN}\norm{\frac{ \sqrt{\rho_{\varepsilon_n}f'(\rho_{\varepsilon_n}) - f(\rho_{\varepsilon_n})} + \sqrt{f^\ast(\phi_{\varepsilon_n}-\sfa)} } {\sqrt{\varepsilon_n}} (\sfc - {\ubar\sfc})^{1/2}}_{L^\infty(0,T; L^2(\Omega))} \\
		\lesssim \sup_{n\in\bbN}\esssup_{t\in[0,T]} \left( \frac{1}{\varepsilon_n} \int_\Omega (\sfc - {\ubar\sfc}) \rho_{\varepsilon_n}(t)^m \,dx + \frac{1}{\varepsilon_n} \int_\Omega (\sfc - {\ubar\sfc}) g(\phi_{\varepsilon_n}(t)) \,dx  \right)^{1/2} \lesssim \sqrt{\ol{\scG}}.
	\end{gathered}
\end{equation}

We show the second factor from the difference of squares in  \eqref{eq:rewritingPressure} vanishes in $L^2(\left]0,T\right[\times\Omega)$.  We finally use the assumption that $f(u) = u^m/(m-1) + \iota_{\left[0,+\infty\right[}(u)$ with $m>q'$ (recall  \ref{hyp-f:growthCondition}), in which case $f^\ast(v) = (v/m')_+^{m'}$, where $(v)_+ \coloneq \max\set{v,0}$. Then for this particular $f$, the function $P$ from \eqref{eq:FenchelDiscrepancy} is
\[
P(u,v) = \frac{m'}{m}u^m - u(v-\sfa) + \big( \frac{v-\sfa}{m'} \big)_+^{m'}.
\] 
We claim $P$ dominates the nonlinearity in the first factor for all $u\ge0$ and $v\in\bbR$:
\[
N(u,v)\coloneq\big(\sqrt{u f'(u) - f(u)} - \sqrt{f^\ast(v-\sfa)} \big)^2 = u^m - 2u^{m/2}\big(\frac{v-\sfa}{m'}\big)_+^{m'/2}+\big(\frac{v-\sfa}{m'}\big)_+^{m'} \le mP(u,v).
\]
For clarity, we change variables $w\coloneq(v-\sfa)/m'$ and, by an abuse of notation, write $P(u,w) = \frac{m'}{m}u^m - m'uw + w_+^{m'}$ and $N(u,w) = u^m - 2u^{m/2}w^{m'/2}_+ + w_+^{m'}$. The bound $N(u,w) \le mP(u,w)$ is clearly true for all $u \ge 0$ and $w \le 0$ or $u=0$ and $w>0$, so we show $mP(u,w) - N(u,w) \ge 0$ for all $u,w>0$. Recall $m-1=m/m'$, so
\begin{gather*}
	mP(u,w) - N(u,w) =	(m' - 1) u^m - 2u^{m/2} w^{m'/2} + (m-1)w^{m'} - mm'uw \\
	= u^{m/2}w^{m'/2} \left[  \frac{m'u^{m/2}}{mw^{m'/2}} + 2 + \frac{m w^{m'/2}}{m'u^{m/2}} - mm' \frac{uw}{u^{m/2}w^{m'/2}} \right].
\end{gather*}
Again, we change variables: let $A\coloneq m' u^{m/2}$ and $B\coloneq m w^{m'/2}$, so the above reads
\begin{gather*}
	= \frac{AB}{m'm} \left[ \frac{A}{B} + 2 + \frac{B}{A} - AB \left(\frac{m'}{A}\right)^{2/m'}\left(\frac{m}{B}\right)^{2/m} \right] \\
	= \frac{AB}{m'm} \left[\sqrt{\frac{A}{B}} + \sqrt{\frac{B}{A}} + \left(m'\sqrt{\frac{B}{A}}\right)^{1/m'}\left(m \sqrt{\frac{A}{B}}\right)^{1/m} \right] 
	\times \left[\sqrt{\frac{A}{B}} + \sqrt{\frac{B}{A}} - \left(m'\sqrt{\frac{B}{A}}\right)^{1/m'}\left(m \sqrt{\frac{A}{B}}\right)^{1/m} \right].
\end{gather*}
On the second line, the first two factors are positive, and the last factor is nonnegative by Young's inequality, which is easily seen after another change of variables: $C^{m'} \coloneq m' \sqrt{B/A}$ and $D^m \coloneq m \sqrt{A/B}$.

Thus, \eqref{eq:vanishingFenchelDiscrepancy} gives
\begin{equation} \label{eq:differenceOfSquares-Vanishing}
	\begin{gathered}
		\limsup_{n\to\infty}\norm{\frac{\sqrt{\rho_{\varepsilon_n} f'(\rho_{\varepsilon_n}) - f(\rho_{\varepsilon_n})} - \sqrt{f^\ast(\phi_{\varepsilon_n}-\sfa)}}{\sqrt{\varepsilon_n}}}_{L^2_{t,x}} 
		\le \lim_{n\to\infty} \left(\frac{1}{\varepsilon_n} \norm{ P(\rho_{\varepsilon_n}, \phi_{\varepsilon_n})}_{L^1_{t,x}}\right)^{1/2} =0.
	\end{gathered}
\end{equation}

Then the identity \eqref{eq:rewritingPressure} and H\"older's inequality give
\begin{gather*}
	\norm{\pi_{\varepsilon_n}(\sfc - {\ubar\sfc})^{1/2}}_{L^2_tL^1_x} 
	\le \norm{ \frac{1}{\varepsilon_n} P(\rho_{\varepsilon_n},\phi_{\varepsilon_n})(\sfc - {\ubar\sfc})^{1/2}}_{L^2_tL^1_x} 
	\\+ \norm{\frac{ \sqrt{\rho_{\varepsilon_n}f'(\rho_{\varepsilon_n}) - f(\rho_{\varepsilon_n})} 
			+ \sqrt{f^\ast(\phi_{\varepsilon_n}-\sfa)} } {\sqrt{\varepsilon_n}} (\sfc - {\ubar\sfc})^{1/2}}_{L^\infty_tL^2_x} 
	\norm{\frac{\sqrt{\rho_{\varepsilon_n} f'(\rho_{\varepsilon_n}) - f(\rho_{\varepsilon_n})} - \sqrt{f^\ast(\phi_{\varepsilon_n}-\sfa)}}{\sqrt{\varepsilon_n}}}_{L^2_{t,x}},
\end{gather*}
so \eqref{eq:vanishingFenchelDiscrepancy}, \eqref{eq:differenceOfSquares-Bounded}, and \eqref{eq:differenceOfSquares-Vanishing} together prove \eqref{eq:weightedPressureVanishing}.

\medskip
\textbf{Step 4:} (Uniform pressure bound)
The purpose of this step is to show, by duality, 
\begin{equation} \label{eq:pressureDualityEstimate}
	\seq{p_{\varepsilon_n}} \text{ is bounded in } L^2(0,T; \cX^\ast), \quad \text{where } p_{\varepsilon_n}(t,x) \coloneq \pi_{\varepsilon_n}(t,x) - \sum_k (\pi_{\varepsilon_n}(t))_{\Omega_0^k} \chi_{\Omega_0^k}(x),
\end{equation} 
The family $\seq{\Omega_0^k}_k$ is the connected components of $\Omega_0$, and for $0<s<1$ from \ref{hyp-c:grad-nondegeneracy}, the space
$\cX \coloneq \cset{\zeta \in C^{1,s}(\cl\Omega)}{ \text{for all } k \,\, \zeta\vert_{\partial\Omega_0^k} = 0 \text{ and }  (\zeta)_{\Omega_0^k} = 0}$,
which is a Banach subspace of $C^{1,s}(\cl\Omega)$. By definition, $p_{\varepsilon_n}$ is mean-zero on each $\Omega_0^k$.

Test vector fields $\xi$ are obtained by solving potential problems. Given $\varphi \in L^2(0,T; \cX)$, we solve
\[
\text{for a.e.\ } t \in[0,T] \qquad 
\begin{cases}
	\lap u(t) = \varphi(t) & \text{in } \oc{\Omega}_0^k \text{ for each } k, \\
	\nabla u(t) \cdot \vec{n}_0 = 0 & \text{on } \partial\Omega_0^k \text{ for each } k,
\end{cases}
\]
which exists since $\varphi(t)$ is mean-zero on each $\Omega_0^k$. We take $u(t)$ to be the mean-zero solution, so Schauder estimates \cite{Nardi_2015_SchauderEstimateSolutions} give $\norm{u(t)}_{ C^{3,s}(\Omega_0)} \lesssim \norm{\varphi(t)}_{C^{1,s}(\cl\Omega)}$, and in particular $u(t) \in C^{2,1}(\Omega_0)$, so we can extend $u(t)$ to some $\wt{u}(t) \in C^{2,1}(\cl\Omega)$ such that each of $u(t)$, $\nabla u(t)$ are compactly supported in $\Omega$ and such that $\norm{\wt{u}(t)}_{C^{2,1}(\cl\Omega)} \lesssim \norm{u(t)}_{C^{3,s-1}(\Omega_0)}$. This may be achieved by first using a Whitney extension-type theorem \cite[Chapter VI, Theorem 4]{Stein_1971_SingularIntegralsDifferentiability},  \cite{Fefferman_2005_InterpolationExtrapolationSmooth} to extend $u(t)$ from $\Omega_0$ to $\bbR^d$ and second multiplying that extension by a smooth cut-off function that is one on $\Omega_0$ and zero on $\cset{x\in\Omega \setminus \Omega_0}{\dist(x,\partial\Omega_0) \ge \delta}$, for some $\delta \in \left]0, \dist(\Omega_0, \partial\Omega)\right[$. Only here do we use assumption \ref{hyp-c:interiorBoundary} that $(\partial\Omega_0) \cap \partial\Omega = \emptyset$. 

Let $\xi\coloneq \nabla \wt{u}$, which satisfies $\xi\cdot \vec{n}_0 = 0$ on $[0,T]\times\partial\Omega_0$ and $\xi = 0$ on $[0,T]\times\partial\Omega$, so in particular $\xi\cdot \vec{n} = 0$ on $[0,T]\times\partial\Omega$. We test \eqref{eq:approxPressureEq} with this particular $\xi$, and then  replace $\pi_{\varepsilon_n}$ by $p_{\varepsilon_n}$ on the region $\Omega_0$, which is possible since $\xi$ is divergence-free on each connected component of $\Omega_0$. This gives
\begin{gather*}
	\abs{\int_0^T\int_{\Omega_0} p_{\varepsilon_n} \varphi \,dx\,dt} 
	\le \abs{\alpha_0\int_0^T \int_\Omega \vec{j}_{\varepsilon_n} \cdot \xi + \frac{1}{\varepsilon_n}\big[ f(\rho_{\varepsilon_n}) - \rho_{\varepsilon_n}(\phi_{\varepsilon_n} - \sfa) \big] \div\xi - \frac{1}{\varepsilon_n}\rho_{\varepsilon_n}\nabla\phi_{\varepsilon_n} \cdot \xi \,dx\,dt} \\
	+ \abs{\int_0^T\int_{\Omega\setminus\Omega_0} \pi_{\varepsilon_n} \lap\wt{u}\,dx\,dt}.
\end{gather*}

We estimate each term on the right-hand side separately. Lemma \ref{lemma:aPrioriEst-continuityEq}-\ref{lemma:massFluxL2L1} and \eqref{eq:firstVariationEstimate} of Lemma \ref{lemma:limitIdentity} combined with the boundedness of our extension of $u(t)$ from $\Omega_0$ to $\cl\Omega$ and the Schauder estimate show the first integral is bounded by
\[
\lesssim \norm{\xi}_{L^2(0,T; C^{1,1}(\cl\Omega))} \le \norm{\wt{u}}_{L^2(0,T; C^{2,1}(\cl\Omega))} \lesssim \norm{u}_{L^2(0,T; C^{2,1}(\Omega_0))} \le \norm{u}_{L^2(0,T;C^{3,s}(\Omega_0))} \lesssim \norm{\varphi}_{L^2(0,T; \cX)}.
\]

The integral on the second line is also bounded by $\norm{\varphi}_{L^2(0,T; \cX)}$, and it additionally vanishes as $n \to \infty$. Since $\lap u(t)$ is continuous up to $\partial\Omega_0$, we have $\lap u(t)\vert_{\partial\Omega_0} = \varphi(t) \vert_{\partial\Omega_0} = 0$ from the definition of $\cX$. By our choice of extension $\wt{u}(t)$, $\lap\wt u(t)$ is Lipschitz on $\cl\Omega$, which is why we chose $\cX$ to be composed of $C^{1,s}$ functions. This provides 
\[
\text{for all } x \in \cl\Omega \setminus \Omega_0 \quad 	\abs{\lap \wt{u}(t,x)} \le \Lip(\lap \wt{u}(t)) \dist(x,\partial\Omega_0) \lesssim \Lip(\lap\wt{u}(t)) \big(\sfc(x) - {\ubar\sfc}\big)^{1/2}.
\]
The last inequality is a consequence of \ref{hyp-c:grad-nondegeneracy}, and we thus have the bound
\begin{equation} \label{eq:approximatePressureBound-Complement}
	\abs{\int_0^T \int_{\Omega \setminus \Omega_0} \pi_{\varepsilon_n} \lap{\wt u} \,dx\,dt} \le	\int_0^T \Lip(\lap \wt u(t)) \left(\int_{\Omega \setminus \Omega_0} \abs{\pi_{\varepsilon_n}(t,x)} \big(\sfc(x) - {\ubar\sfc}\big)^{1/2} \,dx\right) \,dt.
\end{equation}
Combining Cauchy-Schwarz for the integral over $\left]0,T\right[$, the bound 
\[
\norm{\Lip(\lap \wt u)}_{L^2(0,T)} \le \norm{\wt u }_{L^2(0,T;C^{2,1}(\cl\Omega))} \lesssim \norm{u}_{L^2(0,T;C^{3,s}(\Omega_0))} \lesssim \norm{\varphi}_{L^2(0,T; \cX)},
\]
and \eqref{eq:weightedPressureVanishing} from Step 3 provides both the desired estimate
\[	
\text{for all } \varphi \in L^2(0,T; \cX) \quad 	\sup_{n\in\bbN}\abs{\int_0^T\int_{\Omega_0} p_{\varepsilon_n} \varphi \,dx\,dt} \lesssim \norm{\varphi}_{L^2(0,T; \cX)},
\] 
which implies \eqref{eq:pressureDualityEstimate}, and that the right-hand side of \eqref{eq:approximatePressureBound-Complement} vanishes as $n\to\infty$. 

\medskip
\textbf{Step 5:} (Passage to the limit) In \eqref{eq:approxPressureEq}, take $\xi \in L^2(0,T; C^{2,s}(\cl\Omega))$ such that $\div\xi = \xi \cdot \vec{n}_0 = 0$ on $[0,T] \times \partial\Omega_0$ and $\xi \cdot \vec{n} = 0$ on $[0,T] \times \partial\Omega$, so $\div\xi \in L^2(0,T;\cX)$. For such $\xi$, we may then replace $\pi_{\varepsilon_n}$ in \eqref{eq:approxPressureEq} by $p_{\varepsilon_n}$ on the region $\Omega_0$, so
\begin{gather*}
	\frac{1}{\alpha_0}\int_0^T \int_{\Omega_0} p_{\varepsilon_n} \div\xi \,dx\,dt 
	= \int_0^T \int_\Omega \vec{j}_{\varepsilon_n} \cdot \xi \,dx\,dt \\
	+ \frac{1}{\varepsilon_n\alpha_0}\int_0^T \int_\Omega \big[ f(\rho_{\varepsilon_n}) + \sfa\rho_{\varepsilon_n} - \rho_{\varepsilon_n}\phi_{\varepsilon_n} \big] \div\xi - \rho_{\varepsilon_n} \nabla \phi_{\varepsilon_n} \cdot \xi \,dx\,dt 
	- \frac{1}{\alpha_0}\int_0^T \int_{\Omega \setminus \Omega_0} \pi_{\varepsilon_n} \div\xi\,dx\,dt.
\end{gather*}
The result \eqref{eq:pressureDualityEstimate} of Step 4 provides  $p \in L^2(0,T; \cX^\ast)$ and a subsequence of $\seq{p_{\varepsilon_n}}$ such that $p_{\varepsilon_n}$ converges to $p$ weakly-$\ast$ in $L^2(0,T; \cX^\ast)$, which, after extraction, allows to pass to the limit in the first term. The convergence of the next two terms is handled by Proposition \ref{prop:continuityEq}-\ref{prop:continuityEq-convergence} and Proposition \ref{prop:Reshetnyak-timeDep}, resp. The last term vanishes in the limit. To see this, one argues as in the derivation of \eqref{eq:approximatePressureBound-Complement}, but with $\div\xi(t)$ in place of $\lap\wt{u}(t)$, and combines the analogous inequality to \eqref{eq:approximatePressureBound-Complement} with the weighted convergence \eqref{eq:weightedPressureVanishing}.

\appendix

\section{A sufficient condition for nondegeneracy} 
\label{appendix:nondegeneracy}
The following is a sufficient condition for the nondegeneracy condition \ref{hyp-c:grad-nondegeneracy} to be satisfied. We recall \ref{hyp-c:grad-nondegeneracy} is used in the convergence of the partial first variation of $\scG_\varepsilon$ w.r.t.\ its first variable w.r.t.\ domain variations (Proposition \ref{prop:Reshetnyak-timeDep}) as well as in obtaining a uniform bound on the approximate pressures.
\begin{proposition}
	Let $\sfc \in C^{2,1}(\cl\Omega)$ and set ${\ubar\sfc} \coloneq \min_{x\in\cl\Omega} \sfc(x)$. Suppose $\Omega_0 \coloneq \set{\sfc = {\ubar\sfc}}$ is a regular closed set, i.e., $\cl\Omega_0 = \cl{\oc{\Omega}_0}$, and $\partial\Omega_0$ is $C^{2,1}$. If there exists a neighborhood $U$ of $\partial\Omega_0$ and a constant $\lambda>0$ such that 
	$
	\lap\sfc(x) \ge \lambda > 0 \text{ for all } x \in U \setminus \oc{\Omega}_0,
	$ 
	then 
	\begin{enumerate}[label = (\roman*)]
		\item \label{prop:Appendix-quadraticGrowth} $\sfc(x) - {\ubar\sfc} \ge \frac{\lambda}{2} \dist(x,\Omega_0)^2$ for all $x \in U$;
		\item \label{prop:Appendix-gradientGrowth} there exists a constant $K > 0$ such that for all vector fields $\xi \in C^1(\cl U; \bbR^d)$ such that $\xi \cdot \vec{n}_0 = 0$ on $\partial\Omega_0$, where $\vec{n}_0$ is the outward unit normal vector to $\partial\Omega_0$, we have for all $x\in U$ the estimate $\abs{\nabla\sfc(x) \cdot \xi(x)} \le K(\sfc(x) - {\ubar\sfc})\norm{\xi}_{C^1(\cl U; \bbR^d)}$.
	\end{enumerate}
\end{proposition}

\begin{proof}
	Since $\partial\Omega_0$ is compact, it suffices to show the estimates near a fixed $x_0 \in \partial\Omega_0$ and apply the resulting bounds to a finite open cover. After a possible translation and rotation of the coordinate axes, we can assume $x_0 = 0$ and the outer unit normal vector to $\partial\Omega_0$ at $0$ is $\vec{n}_0 = - \vec{e}_d = (0,0,\dots,-1)$. We write $\Omega_0$ locally as an epigraph, so let $(\Theta, B)$ be a $C^{2,1}$ local parameterization of $\partial\Omega_0$ around $0$. Then $x_d = \Theta(x')$, where $x' = (x_1,\dots,x_{d-1})$, satisfies $\Theta(0) = 0$ and $\nabla\Theta(0) = 0$. The second-order Taylor polynomial of $\Theta$ about $0$ gives $\Theta(x') = O(\abs{x'}^2)$. 
	
	\medskip
	\ref{prop:Appendix-quadraticGrowth}: Define
	\[
	S(y) \coloneq \sfc(y', \Theta(y') + y_d) - {\ubar\sfc}  \quad\text{ with } y = (y',y_d) \coloneq (x',x_d - \Theta(x')),
	\]
	so $S(y) = 0$ when $y_d \ge 0$ is small. Consequently, $\partial_k S(y',0)=\partial_{j,k}^2 S(y',0) = 0$ for all $j \in \set{1,\dots,d-1}$, $k\in\set{1,\dots,d}$. Further, 
	\[
	\lap_x \sfc(x) = \lap_y S(y) + \sum_{j=1}^{d-1} \Big[ \big( \abs{\partial_{x_j}\Theta}^2 -1 \big)\partial_{y_d,y_j}S - \partial_{x_j}\Theta\partial_{y_j}^2 S - \partial_{y_j}S \partial_{x_j}^2\Theta \Big]_{y=(x',x_d-\Theta(x'))}, 
	\]
	so the above implies that when $x_d = \Theta(x')$ we have $\partial_{y_d}^2 S (y',0) = \lap_x\sfc(x) \ge \lambda$.
	
	Define $h(s) \coloneq S(y',sy_d)$, so $h(0) = 0$ and the fundamental theorem of calculus give, for $y_d$ small,
	\begin{align*}
		S(y) = h(1) = \int_0^1 \int_0^s h''(t) \,dt \,ds = \int_0^1 \int_0^s y_d^2 \partial_{y_d}^2S (y',ty_d) \,dt\,ds  
		= y_d^2 \int_0^1 (1-t) \partial_{y_d}^2 S (y',ty_d)\,dt \ge \frac{\lambda}{2} y_d^2.
	\end{align*}
	Changing variables shows 
	\begin{equation} \label{eq:appendix-growthOfSigma-prelim}
		\sfc(x) - {\ubar\sfc} \ge \frac{\lambda}{2}(x_d - \Theta(x'))^2 \ge \frac{\lambda}{2} \dist(x,\partial\Omega_0)^2,
	\end{equation}
	which proves \ref{prop:Appendix-quadraticGrowth}.
	
	\medskip
	\ref{prop:Appendix-gradientGrowth}: A similar calculation using the fundamental theorem of calculus and $\sfc,\Theta \in C^{2,1}$ provides
	\begin{subequations}
		\begin{gather}
			\forall \, j \in \set{1,\dots,d-1} \qquad 	\partial_{y_j}S(y) = y_d^2 \int_0^1(1-t) \partial_{y_d,y_d,y_j}^3{S}(y',ty_d)\,dt = O(y_d^2), \label{eq:appendix-pdGrowth-nonOrthogonaDirections} \\
			\partial_{y_d}S(y) = y_d \partial_{y_d}^2 S(y',0) + y_d^2 \int_0^1 (1-t)\partial_{y_d}^3 S(y',ty_d)\,dt = y_d\lap\sfc(x) + O(y_d^2). \label{eq:appendix-pdGrowth-orthogonalDirection}
		\end{gather}
	\end{subequations}
	Writing $\xi'(x) \coloneq (\xi^1(x), \dots, \xi^{d-1}(x))$,  $\nabla_{x'} \sfc(x) = \nabla_{y'}S(y) - \partial_{y_d} S(y) \nabla_{x'}\Theta(x')$, and $\partial_{x_d}\sfc(x) = \partial_{y_d}S(y)$, we compute 
	\begin{align*}
		\nabla\sfc(x)\cdot\xi(x) = \nabla_{x'}\sfc(x) \cdot \xi'(x) + \partial_{x_d}\sfc(x)\xi^d(x)  
		= \nabla_{y'} S(y) \cdot \xi'(x)  - \partial_{y_d}S(y) \nabla_{x'}\Theta(x')\cdot \xi'(x) + \partial_{y_d}S(y)\xi^d(x).
	\end{align*}
	We note that \eqref{eq:appendix-pdGrowth-nonOrthogonaDirections} shows $\abs{\nabla_{y'}S(y)} = O(y_d^2)$. Also, the outer unit normal $\vec{n}_0$ to $\partial\Omega_0$ is locally 
	\[
	\vec{n}_0(x) = (\nabla_{x'}\Theta(x'),-1)/\sqrt{\abs{\nabla_{x'}\Theta(x')}^2 + 1},
	\]
	so, by hypothesis, $0 = \xi(x)\cdot \vec{n}_0(x) = \xi'(x)\cdot\nabla_{x'}\Theta(x') - \xi^d(x)$ on $\set{x_d = \Theta(x')}$. These, together with Taylor's theorem, give
	\begin{align*}
		\abs{\nabla\sfc(x)\cdot\xi(x)} &= O(y_d^2)\norm{\xi}_{C(\cl U)^d} - \partial_{y_d} S(y) \Big[ \xi'(x)\cdot\nabla_{x'}\Theta(x') - \xi^d(x) - \big( \xi'(x',\Theta(x'))\cdot\nabla_{x'}\Theta(x') - \xi^d(x',\Theta(x')) \big) \Big] \\
		&= O(y_d^2)\norm{\xi}_{C(\cl U)^d} + O(\abs{y_d}) \times \norm{\nabla\xi}_{C(\cl U)^{d\times d}} O(\abs{x_d - \Theta(x')}) = O(y_d^2)\norm{\xi}_{C^1(\cl U)^d}.
	\end{align*}
	In view of \eqref{eq:appendix-growthOfSigma-prelim}, the claim is proven.
\end{proof}

\printbibliography

\end{document}